\documentclass[11pt, oneside]{amsart}   	
\usepackage{geometry,amsmath,amssymb} \usepackage[table]{xcolor}
\usepackage{ytableau}
\ytableausetup{smalltableaux}
\usepackage{multirow}
\renewcommand{\arraystretch}{1.25}
\usepackage[utf8]{inputenc}
\usepackage[table]{xcolor}
\usepackage{arydshln}
\usepackage{ wasysym }
\usepackage{shuffle}
\usepackage{ stmaryrd }
\usepackage{tikz} 
\usetikzlibrary{arrows,shapes}

\usepackage{rotating}

\def\rad{\mathrm{rad}}

\def\ch{\mathrm{ch}}

\def \ker{\mathrm{ker}}
\def \im{\mathrm{im}}

\usepackage[numbers]{natbib}
\usepackage[mathscr]{euscript}
\usepackage{mathtools}
\usepackage{ bbold }
\usepackage{nicematrix}
\usepackage{ctable}

\usepackage{hyperref, enumitem}
\usepackage{xcolor,shuffle}

\usepackage[capitalize]{cleveref}

\usepackage[utf8]{inputenc}

\newtheorem{thm}{Theorem}[section]
\newtheorem{lemma}[thm]{Lemma}
\newtheorem{prop}[thm]{Proposition}
\newtheorem{cor}[thm]{Corollary}

\newtheorem{question}[thm]{Question}
\newtheorem{definition}[thm]{Definition}
\usepackage{comment}

\theoremstyle{definition}
\newtheorem{example}[thm]{Example}

\newtheorem{remark}[thm]{Remark}

\def\C{{\mathbb{C}}}

\def\stab{\mathrm{Stab}}
\def\sgn{\mathrm{sgn}}

\def\type{\mathrm{type}}
\def\ch{\mathrm{ch}}

\DeclareTextFontCommand{\emph}{\bf}

\newcommand{\patty}[1]{\textcolor{olive}{Patty:[#1]}}

\title{Invariant theory for the face algebra of the braid arrangement}
\author{Patricia Commins}
\begin{document}

\maketitle
\begin{abstract}
    {The faces of the braid arrangement form a monoid. The associated monoid algebra -- the face algebra-- is well-studied, especially in relation to card shuffling and other Markov chains. In this paper, we explore the action of the symmetric group on the face algebra from the perspective of invariant theory. Bidigare proved the invariant subalgebra of the face algebra is (anti)isomorphic to Solomon's descent algebra. We answer the more general question: what is the structure of the face algebra as a simultaneous representation of the symmetric group and Solomon's descent algebra?

Special cases of our main theorem recover the Cartan invariants of Solomon's descent algebra discovered by Garsia--Reutenauer and work of Uyemura-Reyes on certain shuffling representations. Our proof techniques involve the homology of intervals in the lattice of set partitions.}
\end{abstract}

\setcounter{tocdepth}{1}

\section{Introduction}
In this paper, we explore the interactions between the symmetric group and two related finite dimensional algebras: Solomon's descent algebra and the face algebra of the braid arrangement.

The descent algebra is a remarkable subalgebra of the symmetric group algebra -- and more generally, a Coxeter group algebra-- which was discovered by Solomon in \cite{Solomondescentbirth}. There has since been great progress in understanding its representation theory (see \cite{atkinson,Bergerons-hyper1, Bergerons-general, Bergerons-hyper2, blessonlaue, bonnafepfeiffer, garsia-reutenauer, pfeiffer2009quiver, saliolaloewytypeD, saliolaquiverdescalgebra, Schocker}). However, as a non semisimple algebra, its structure is rich and many open questions (such as the Cartan invariants in general Coxeter type) remain. The descent algebra is also studied for its connections with the free Lie algebra \cite{garsia-reutenauer, reutenauer}, higher Lie representations \cite{schocker2003multiplicities, Uyemura-Reyes}, and characters of Coxeter groups \cite{blesshohlshock, jollenbeckreutenauer, Solomondescentbirth}. See \cite{SchockerSurvey} for a very nice overview of these connections.

The faces of any hyperplane arrangement form a semigroup under a product first considered by Tits in \cite{TITS1976265}. This face semigroup is perhaps best known for its connections to card shuffling and other Markov chains, which were pioneered by Bidigare--Hanlon--Rockmore in \cite{BHR} and continued by many others (see, for example, \cite{browndiac, DIEKER2018427, lafreniere, reiner2014spectra, Uyemura-Reyes}). Beyond shuffling, hyperplane face semigroups are also interesting in their own right. For instance, the structure of the associated semigroup {algebras} is governed by the combinatorics of the arrangement. Face semigroup algebras have been studied extensively from this perspective by Aguiar--Mahajan in \cite{Aguiar-Mahajan},  Saliola in \cite{saliolaloewytypeD, saliolaquiverdescalgebra, saliolafacealgebra}, and Schocker in \cite{Schocker}. Finally, face semigroups are also studied as a canonical example of a left regular band, a special type of idempotent semigroup studied in \cite{aguiar2006coxeter, BrownonLRBs, margolis2015combinatorial, MSS,saliolaquiverlrb,  SteinbergRepMonoidsBook}.

The face algebra of the braid arrangement and the descent algebra are intrinsically linked by the symmetric group, which acts on the face algebra by algebra automorphisms. Specifically, in his PhD thesis \cite{Bidigare}, Bidigare proved the invariant subalgebra of the face algebra is (anti)isomorphic to the descent algebra. Bidigare's connection provides a useful perspective on the descent algebra. As Schocker wrote in \cite{SchockerSurvey}, it offers ``fully transparent proofs of [several of] Solomon's results and sheds new light on the representation theory of the descent algebra.'' As concrete examples, Saliola used this geometric viewpoint of the descent algebra to compute its quiver in types $A$ and $B$ in \cite{saliolaquiverdescalgebra} and its Loewy length in type $D$ in \cite{saliolaloewytypeD}. 

Inspired by classical invariant theory, we aim to build on this story by further exploring the connections between the descent algebra, the face algebra, and the symmetric group. In particular, when a finite group $G$ acts on a ring $R$ by ring automorphisms, it is natural to consider the structure of $R$ as a module over its invariant ring $R^G.$  Bidigare's result reveals the face algebra is a (right) module over Solomon's descent algebra, opening the door to studying a concrete, combinatorial representation of the descent algebra. In \cite[\S 9.3]{Schocker}, Schocker initiated study in this direction. Specifically, he studied the indecomposable projective and simple modules of the face algebra as descent algebra modules in \cite[Prop 9.4, Prop 9.6, Thm 9.7]{Schocker}. 

Invariant theory also asks a more refined question. Provided $R$ is an algebra over a field of characteristic zero, $R$ decomposes into its \textit{$G$-isotypic subspaces}, \[R = \bigoplus_{\chi}R^\chi,\] as $\chi$ varies over the irreducible characters of $G.$ The trivial isotypic subspace is precisely the invariant subalgebra of $R$ and each isotypic subspace $R^\chi$ is a module over the invariant subalgebra $R^G.$ Studying each $G$-isotypic subspace as a module over the invariant subalgebra examines $R$ as a module over both $G$ and $R^G$ simultaneously.

It is thus a natural extension of Bidigare's result to consider the other symmetric group isotypic subspaces of the face algebra. In fact, although he did not consider its module structure over the descent algebra,  Bidigare also studied the structure of the sign isotypic subspace of the face algebra in his thesis \cite[\S 3.5.3]{Bidigare}. In what follows, let $S_n$ be the symmetric group on $n$ letters and let $\C \mathcal{F}_n$ denote the face algebra (over the complex numbers $\C$) of the associated braid arrangement.  Our main question is as follows.

\begin{question}\label{question:OG-pre}
    What is the structure of each $S_n$-isotypic subspace of $\C \mathcal{F}_n$ as a module over the descent algebra?
\end{question}

 Our answer to \cref{question:OG-pre} is (informally) summarized in the following theorem, which combines \cref{prop:decomp-isos},  \cref{prop:transform-to-group-rep-q}, and \cref{conj:typeA-comp-factors}. The irreducible representations of both the symmetric group and the descent algebra are indexed by partitions of $n.$ Let $\chi^\nu$ be the irreducible character of $S_n$ associated to the partition $\nu$ and define $f^\nu$ to be its dimension. Let $M_\lambda$ be the irreducible representation of the descent algebra associated to the partition $\lambda.$ We write $s_\nu$ for the Schur function indexed by the partition $\nu$ and $L_\rho$ for the symmetric function corresponding to the higher Lie character indexed by the partition $\rho.$ Any remaining notation is discussed in \cref{sec:full-answer}. 
\begin{thm}\label{thm:informal-main}
    The $\chi^\nu$-isotypic subspace of the face algebra $\C \mathcal{F}_n$ has a decomposition into descent algebra submodules $N_{\mu}^\nu$  labelled by partitions $\mu$ of $n:$ \[ \left(\C \mathcal{F}_n\right)^{\chi^{\nu}} = \bigoplus_{\substack{\nu \text{ weakly}\\ \text{ dominates }\mu}} N_{\mu}^\nu.\] The $\C$-dimensions of the descent algebra modules $N_{\mu}^{\nu}$ have simple combinatorial formulas and the composition multiplicity of the descent algebra simple $M_\lambda$ within $N_{\mu}^{\nu}$ is \[f^\nu \cdot \left\langle s_\nu, \text{coefficient of }\mathbf{y}_\lambda \mathbf{z}_\mu \text{ in } F\right\rangle,\] where $F$ is the generating function
    \[F = \prod_{\substack{\text{Lyndon}\\\text{words } w}} \,\,
    \sum_{\substack{\text{partitions}\\ \rho}}\mathbf{y}_{\rho \cdot |w|} \, \, \mathbf{z}_{w^{|\rho|}} \,\, L_\rho [h_w].\]
\end{thm}
The composition multiplicities in the special cases $\nu = n$ and $\mu = 1^n$ recover work of Garsia--Reutenauer on the Cartan invariants of the descent algebra and Uyemura-Reyes on certain shuffling representations, respectively. 

\vspace{.5cm}
\noindent \textbf{Outline of paper.}  In \cref{sec:two-algebras}, we provide background on the descent algebra, the face algebra of the braid arrangement, Bidigare's theorem, and some representation theoretic concepts. In \cref{sec:quest} we provide a preliminary answer (\cref{prop:decomp-isos}) which reduces \cref{question:OG-pre} to a more refined question (\cref{quest:2'}). In \cref{sec:full-answer} we state and give examples of our main theorem (\cref{conj:typeA-comp-factors}), explain how it recovers work of Uyemura-Reyes and Garsia--Reutenauer, and give a more explicit answer for the sign isotypic subspace (\cref{prop:sign-iso}). \cref{sec:proof} is dedicated to the proof of our main theorem. For readibility, the proof is first outlined in \cref{sec:outline-proof} and then broken into smaller subsections. Finally, we conclude with a brief discussion on generalizations to other Coxeter types and left regular bands in \cref{sec:generalizations}.

\vspace{.5cm}
\noindent\textbf{Acknowledgements.} The author would like to thank her advisor Vic Reiner for proposing this project and providing truly invaluable guidance at every stage; Sheila Sundaram for several very helpful discussions about her work in \cite{SundaramAdvances} and for sparking the idea of an improved proof for step two of the proof outline; Franco Saliola for very helpful discussions early on about his work in \cite{ saliolaquiverdescalgebra, saliolafacealgebra}; Esther Banaian, Sarah Brauner, Darij Grinberg, Trevor Karn, and Ben Steinberg for valuable references and conversations. Finally, the author wishes to thank the SageMath developers \cite{sagemath}.  The author is supported by an NSF GRFP Award \# 2237827 and also benefited from NSF grant DMS-2053288 for related travel.

\section{Background}\label{sec:two-algebras}
\subsection{The braid arrangement and its face algebra}
Write $\mathbf{x}:= (x_1, x_2, \ldots, x_n)$ to denote an element of the vector space $\mathbb{R}^n.$ The \emph{braid arrangement $\mathcal{B}_n$} is the hyperplane arrangement in $\mathbb{R}^n$ consisting of the hyperplanes $\{\mathbf{x}: x_i = x_j\}$ for all $1 \leq i < j \leq n.$ Each hyperplane $\{\mathbf{x}: x_i = x_j\}$ partitions $\mathbb{R}^n$ into three subsets: the halfspace $H_{ij}^+ = \{\mathbf{x}: x_i > x_j\},$ the halfspace $H_{ij}^- = \{\mathbf{x}: x_i < x_j\},$ and the hyperplane itself $H_{ij}^0.$ The \emph{faces} of $\mathcal{B}_n$ are the nonempty intersections of the form \[\bigcap_{1 \leq i < j \leq n}H_{ij}^{\sgn_{ij}}\] for some set of choices $\sgn_{ij}$ in $\{+, -, 0\}.$ 

The faces of $\mathcal{B}_n$ naturally correspond to strings of inequalities relating all coordinates. For example, one face $F $ of $\mathcal{B}_7$  corresponds to the string $x_4 < x_1 = x_5 < x_7 < x_2 = x_3 = x_6$. Combinatorially, these strings (and their corresponding faces) are {\emph{ordered set partitions}} of the set $[n]:= \{1, 2, \ldots, n\}$. For example, $F$ corresponds to the ordered set partition $\left (\{4\}, \{1, 5\}, \{7\}, \{2, 3, 6\}\right),$ which we write as $(4, 15, 7, 236).$ The symmetric group \emph{$S_n$} acts on the faces of $B_n$ by $\pi \left(P_1, P_2, \ldots, P_k\right) := \left(\pi(P_1), \pi(P_2), \ldots, \pi(P_k)\right).$ 

\begin{example} \label{ex:braid3labelled}
The braid arrangement $\mathcal{B}_3$ (intersected with the plane $x_1 + x_2 + x_3 = 0$) and its thirteen faces are shown below. The colors point out the four $S_3$-orbits of faces.
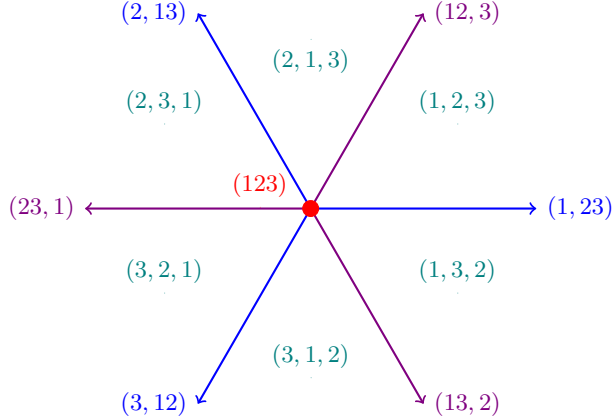
\begin{center}
    \begin{tikzpicture}[scale=3]

\draw[blue, thick, -> ]  (0, 0) -- (-.5, -0.86602540378
) node[anchor=east]{\footnotesize$(3,12)$}; 

\draw[violet, thick, -> ] (0, 0) -- (.5,  0.86602540378) node[anchor=west] {\footnotesize$(12,3)$}
;
\draw[blue, thick, <-] (-.5, 0.86602540378
) node[anchor=east]{\footnotesize$(2,13)$} --(0, 0); 


\draw[violet, thick, ->](0, 0)-- (.5,  -0.86602540378)
node[anchor=west] {\footnotesize$(13,2)$};
 \draw[violet, thick, <-]  (-1, 0
)-- (0, 0);
\filldraw[violet] (-1,0) circle (0pt) node[anchor=east] {\footnotesize$(23,1)$};
\draw[blue, thick, ->](0, 0)-- (1, 0) node[anchor=west] {\footnotesize$(1,23)$}
;

\filldraw[red] (0,0) circle (1pt) ;

\filldraw[red] (-.225, .005) circle (0pt) node[anchor=south]{\footnotesize$(123)$};

\filldraw[teal] (.6495,.375)  circle (0pt) node[anchor=south]{\footnotesize$(1,2,3)$};
\filldraw[teal] (0,.75)  circle (0pt) node[anchor=north]{\footnotesize$(2,1,3)$};
\filldraw[teal] (.6495,-.375)  circle (0pt) node[anchor=south]{\footnotesize$(1,3,2)$};
\filldraw[teal] (-.6495, .375)  circle (0pt) node[anchor=south]{\footnotesize$(2,3,1)$};
\filldraw[teal] (-.6495, -.375)  circle (0pt) node[anchor=south]{\footnotesize$(3,2,1)$};
\filldraw[teal] (0, -.75)  circle (0pt) node[anchor=south]{\footnotesize$(3,1,2)$};
\end{tikzpicture}
\end{center}
\normalsize
\end{example}

The faces of $\mathcal{B}_n$ have an associative, \textit{noncommutative} multiplicative structure, first considered by Tits in \cite{TITS1976265}. In terms of ordered set partitions, 
\[\left(P_1, P_2, \ldots, P_k \right) \cdot \left(Q_1, Q_2, \ldots, Q_\ell\right) := \left(P_1 \cap Q_1, P_1 \cap Q_2, \ldots, P_1 \cap Q_\ell, P_2 \cap Q_1, \ldots P_k \cap Q_\ell \right)^{\wedge},\] where $\wedge$ indicates the removal of empty sets. For example, in $\mathcal{B}_7,$
\[(4, 15, 7, 236) \cdot (245, 367, 1) = (4, 5, 1, 7, 2, 36).\] 

The ordered set partition with a single block $(12\ldots n)$ is an identity element, so the faces have the structure of a monoid, which we denote by \emph{$\mathcal{F}_n.$} Note that for any $f, g \in \mathcal{F}_n,$
\begin{equation}\label{eqn:lrb-property}
    fgf = fg.
\end{equation}
 By setting $g = (12 \cdots n)$ in \cref{eqn:lrb-property}, one sees that each element in $\mathcal{F}_n$ is an idempotent. In fact,  \cref{eqn:lrb-property} implies that $\mathcal{F}_n$ belongs to a very special class of idempotent semigroups, called left regular bands, which we (briefly) return to in \cref{sec:LRBs}.

Let \emph{$\Pi_n$} denote the lattice of \textit{unordered} set partitions of $[n]$ under the refinement ordering, so that the set partition with all numbers in a single block is a maximum element. There is a map, often called the \emph{support map} $$\sigma: \mathcal{F}_n \to \Pi_n$$
given by removing the ordering on the blocks of the ordered set partitions. For instance $\sigma\left(\left(4, 15, 7, 236\right)\right)$ is the unordered set partition $\{15, 236, 4, 7\}.$ This map has the special property that for elements $f, g\in \mathcal{F}_n,$ 
\begin{equation}\label{eqn:multipy-faces-meet}
    \sigma(fg) = \sigma(f) \wedge \sigma(g),
\end{equation}
 where $\wedge$ denotes the \emph{meet} operation of a lattice.

Symmetric group actions are of central importance in this paper. It is routine to check that the action of $S_n$ on $\mathcal{F}_n$ is by monoid automorphisms. We shall also consider the standard action of $S_n$ on $\Pi_n,$ defined as $\pi\left\{P_1, \ldots, P_k\right\}=
\left\{\pi(P_1), \ldots, \pi(P_k)\right\}
.$ One can check that the support map $\sigma$ commutes with this action of $S_n$. The $S_n$-orbits of set partitions in $\Pi_n$ are indexed by (integer) partitions $\lambda$ whose parts sum to $n$ (which we write as \emph{$\lambda \vdash n$}).  For a set partition $X \in \Pi_n,$ and integer partition $\lambda \vdash n$ we write \emph{$X \in \lambda$} if the blocks of $X$ have sizes given by the parts of $\lambda.$ For instance, $\{15, 236, 4, 7\} \in (3, 2, 1, 1).$

Rather than the monoid $\mathcal{F}_n$ itself, we are primarily interested in the \emph{face algebra} $\mathbb{C}\mathcal{F}_n$ which is the free $\mathbb{C}$-module with basis $\mathcal{F}_n$ and multiplication
\[\left(\sum_{f \in \mathcal{F}_n}c_ff \right)\cdot \left(\sum_{g \in \mathcal{F}_n}d_gg\right) := \sum_{f, g \in \mathcal{F}_n}c_fd_g \,\, f g.\]

Since the symmetric group acts on $\mathcal{F}_n$ by monoid automorphisms, its linearly extended action on $\mathbb{C}\mathcal{F}_n$ is by algebra automorphisms. The \emph{invariant subalgebra of the face algebra}, written as \emph{$\left(\C \mathcal{F}_n \right)^{S_n}$}, consists of the elements $x \in \C \mathcal{F}_n$ for which $w(x) = x$ for all $w \in S_n.$ Observe that the action of the symmetric group algebra \emph{$\C S_n$} commutes with both the left- and right- actions of the invariant subalgebra on the face algebra. In particular, for $x \in \C S_n$, $y \in \C \mathcal{F}_n$, and $z \in \left(\C \mathcal{F}_n \right)^{S_n},$  the following hold: 
\begin{align} \label{lem:commuting-action}
        x(y z) &= x(y) z, \  \text{ and } \ 
        x(z y) = z x(y).
\end{align}

\subsection{Solomon's descent algebra and Bidigare's theorem}
Each permutation  $\pi \in S_n$ has an associated (right) \emph{descent set} $\mathrm{Des}(\pi) := \{i: \pi(i) > \pi(i + 1)\}\subseteq [n - 1].$ For each subset $J \subseteq [n - 1]$, define an element $\mathbf{x}_J$ in $\C S_n$ by \[\mathbf{x_J} := \sum_{\pi: \mathrm{Des}(\pi) \subseteq J}\pi.\] In \cite{Solomondescentbirth}, Solomon {proved}
that the  $\C$- span of the elements $\left \{x_J: J \subseteq [n- 1]\right\}$ is closed under multiplication, so it is a subalgebra of $\mathbb{C}S_n.$ The structure constants with respect to this basis have a simple combinatorial formula (see\footnote{In the notation of \cite{garsia-reutenauer}, $x_J$ is $B_{\alpha(J)},$ where $\alpha(J)$ is the composition we define before \cref{thm:bidigare}.} {\cite[Proposition 1.1]{garsia-reutenauer}}). This subalgebra is 
known as \emph{Solomon’s descent algebra}, or simply the descent algebra, which we will denote by \emph{$\Sigma_n$}. The descent algebra is intimately linked to the face algebra. In particular, for a subset $J = \{a_1<a_2 < \ldots < a_k\} \subseteq [n - 1],$ write ${\alpha(J)}$ to be the integer composition $(a_1, a_2 - a_1, a_3 - a_2, \ldots, a_k - a_{k - 1}, n - a_k).$ Bidigare proved the following connection in \cite[Theorem 3.8.1]{Bidigare}.

\begin{thm}(Bidigare)\label{thm:bidigare}
The $S_n$-invariant subalgebra of the face algebra is antiisomorphic to Solomon's descent algebra.  An antiisomorphism from $\Sigma_n$ to $\left(\C \mathcal{F}_n\right)^{S_n}$ is given by \[\Phi:x_J \mapsto \sum_{\substack{\text{Faces }F\\ \text{with block }\\ \text{sizes } \alpha(J)}}F.\] 
\end{thm}

\begin{example}
     In one-line notation, the element $x_{\{1\}} = 123 + 213 + 312 \in \Sigma_3$ is mapped under Bidigare's antiisomorphism to the sum of the three rays colored blue in \cref{ex:braid3labelled}.
\end{example}

\subsection{Representation theory of finite dimensional algebras}\label{subsec:rep-theory-fin-dim-algebras} In order to explain the representation theory of the descent algebra, we first give a brief overview of the representation theory of finite dimensional algebras. We follow the explanations given in \cite[Appendix $D$]{Aguiar-Mahajan}, \cite[Ch. 1]{assem}, \cite[Ch. 9]{etingof}, \cite[\S 4.2]{MSS}, and \cite[Ch. 6, 7]{webbrepntheory}.

 A \emph{$\C$-algebra $A$}, which we will often just call an algebra, is a $\C$-vector space which is also a ring with unit in which $c(ab) = (ac)b = a(cb) = (ab) c$ for all $c \in \C, a, b \in A.$ We will assume $A$ is \emph{finite dimensional}, which means it is finite dimensional as a $\C$-vector space. A left (or right) \emph{representation of an algebra $A$} is a left (or right) $A$-module, so we will use the language of modules and representations somewhat interchangeably. For simplicity, we assume we are dealing with {left} \footnote{Even though the isotypic subspaces of the face algebra are \textit{right} modules over the descent algebra, we will mainly view them as \textit{left} modules over the invariant subalgebra of the face algebra. These approaches are equivalent by Bidigare's theorem.} $A$-modules for the remainder of this subsection, but the corresponding analogues hold for right $A$-modules too. An $A$-module $U$ is 
\begin{itemize}
    \item \emph{indecomposable} if it cannot be written as a nontrivial direct sum of $A$-modules $U = U_1 \oplus U_2.$ 
    \item \emph{simple} if it has no nontrivial submodules. 
    \item \emph{projective} if there is a free $A$-module $F$ and another $A$-module $U'$ such that $F \cong U \oplus U'$ as $A$-modules.
\end{itemize}

The \emph{idempotents} in an algebra are the elements $E$ for which $E^2 = E,$ and  play a key role in the representation theory of finite dimensional algebras. A family $E_1, E_2, \ldots, E_n$ of nonzero idempotents in $A$ is said to be 
\begin{itemize}
    \item \emph{complete} if $E_1 + E_2 + \ldots + E_n = 1.$
    \item \emph{orthogonal} if $E_iE_j = 0$ for all $i \neq j.$
    \item \emph{primitive} if for all $i,$ there do not exist two nonzero orthogonal idempotents $E, E'$ for which $E_i = E + E'.$
\end{itemize}

\emph{Complete families of primitive, orthogonal idempotents} (or \emph{cfpois})  $\{E_1, E_2, \ldots, E_n\}$ of $A$ biject with decompositions of $A$ into indecomposable left $A$-modules by

\begin{align*}
    A = \bigoplus_{i = 1}^n A E_i,
\end{align*}
 see, for instance, \cite[Proposition 7.2.1]{webbrepntheory}. The decomposition of $A$ into indecomposables is unique, up to re-indexing and $A$-module isomorphisms, by the Krull-Schmidt theorem. Since each $A$-module $AE_i$ is a direct summand of $A,$ each is projective. Further, up to isomorphism, these are the only projective indecomposable modules, so the isomorphism classes of these summands are called \emph{the projective indecomposable modules of $A$} (see \cite[Lemma I.5.3(b)]{assem}). So, for any two cfpois $\{E_i: 1 \leq i \leq n\} $ and $\{F_j: 1 \leq j \leq m\}$ of $A,$ we know that $m = n$ and there exists a reordering $\sigma \in S_n$ such that $E_i \cong F_{\sigma(i)}$ as $A$-modules.  

 \begin{remark}\label{rem:conjugation-idems}
   We will use the following two facts several times. Let $A$ be a finite dimensional $\C$-algebra.  Conjugating any cfpoi of $A$ by a unit of $A$ produces another cfpoi of $A$. Conversely, any two cfpois of $A$ are conjugate by some unit of $A$. (This latter fact follows from \cite[Lemma D.26 and its preceding discussion]{Aguiar-Mahajan}.)
\end{remark}

The \emph{radical} of a finite dimensional algebra $A,$ written \emph{$\rad(A)$}, is the intersection of all of its maximal left ideals. It turns out to be a nilpotent, two-sided ideal of $A$ (\cite[Corollary I.1.4, Corollary I.2.3]{assem}). More generally, the radical of an $A$-module $M,$ written \emph{$\rad(M)$} is the intersection of all its maximal submodules. One can also define it as $\rad(A) M$ (\cite[Proposition I.3.7]{assem}). Up to isomorphism, the set of simple $A$-modules is $$\left\{AE_i / \rad(AE_i): 1 \leq i \leq n \right\}.$$ The projective module $AE_i$ is a \emph{projective cover} for the simple module $AE_i / \rad(A E_i)$ (\cite[Corollary I.5.9]{assem}). Further, one has an isomorphism of the simples $AE_i / \rad(A E_i) \cong A E_j / \rad(A E_j)$ if and only if one has a corresponding isomorphism of their projective covers $AE_i \cong AE_j.$ In general, it is possible that $AE_i \cong AE_j$ for $i \neq j$. If we condense and write the list of isomorphism classes of simple and projective indecomposable $A$-modules as $\{S_\alpha\}$ and $\{P_\alpha\}$ as $\alpha$ ranges in some indexing set $\mathscr{A},$ it turns out that as $A$-modules,
 \begin{align} \label{eqn:algebra-decomposition}
     A \cong \bigoplus_{\alpha \in \mathscr{A}}P_\alpha^{\oplus \dim_\C S_\alpha},
 \end{align}
 
 see for instance \cite[Theorem 7.3.9]{webbrepntheory}.

 An algebra $A$ is \emph{semisimple} if every $A$-module can be decomposed into a direct sum of simple $A$-modules. By the Wedderburn-Artin theorem, an algebra $A$ is semisimple if and only if $\rad(A) = 0$ (\cite[Theorem I.3.4]{assem}). Maschke's theorem gives that group algebras $\C G$ for finite groups $G$ are semisimple. Unlike group algebras, many finite-dimensional $\C$-algebras are \textit{not} semisimple (including the face algebra and the descent algebra).

To understand an $A$-module $M$ for a non-semisimple algebra $A$, alternative methods to decomposing $M$ into a direct sum of simple $A$-modules must be used. A \emph{composition series} of a finite dimensional $A$-module $M$ is a sequence of $A$-modules $$0 \subsetneq M_1 \subsetneq M_2 \subsetneq \ldots \subsetneq M_{n - 1} \subsetneq M$$ such that each successive quotient $M_i / M_{i - 1}$ (called a \emph{composition factor}) is a simple $A$-module. Although composition series of an $A$-module are not unique, the Jordan H\"{o}lder theorem explains that the composition factors  (up to isomorphism) and their multiplicities are. The multiplicity of a simple $S_\alpha$ as a composition factor of any composition series of an $A$-module $M$ is called the \emph{composition multiplicity of $S_\alpha$ in $M$}, and written as \emph{$[M: S_\alpha].$} 
 
 One way to compute composition multiplicities without constructing an explicit composition series uses the projective indecomposables as projective covers for the simple modules. In particular, let $M$ be a finite dimensional $A$-module, let $S = AE / \rad(A E)$ be a simple $A$-module, and let $P = AE$ be its projective cover. 
 Since we are working over $\C$, it turns out that (see \cite[Proposition 7.4.1]{webbrepntheory}, specializing to $\C$),
\begin{align}\label{eqn:composition-mults}
     [M: S] = \dim_\C \mathrm{Hom}_A\left(P, M\right) = \dim_\C \mathrm{Hom}_A\left(AE, M\right) = \dim_\C EM.
 \end{align}

 A case of special importance is understanding the composition factors of the projective indecomposable modules. Let $\{S_\alpha\}$ and $\{P_\alpha\}$ be the isomorphism classes of simple $A$-modules and their corresponding projective indecomposables, as $\alpha$ varies in some indexing set $\mathscr{A}$. The composition multiplicities $[P_\beta: S_\alpha]$ as $\alpha, \beta$ vary in $\mathscr{A}$ are called the \emph{Cartan invariants} of $A.$ Using the above equation, \begin{align}
     [P_\beta: S_\alpha] = \dim_{\C} E_\alpha A E_\beta.
 \end{align}The Cartan invariants of the descent algebra are known in types A and B, and have beautiful combinatorial interpretations; we will discuss these in \cref{sec:recovering-results} and \cref{sec:generalizations}.

 \subsection{Representation theory of Solomon's descent algebra} \label{sec:rep-theory-descent-algebra}
 The representation theory of the descent algebra involves elegant combinatorics. Its (right) representation theory was studied in great depth by {Garsia and Reutenauer} in \cite{garsia-reutenauer} using its connections with the free Lie algebra. It was soon after studied in other Coxeter types by Bergeron, Bergeron, Howlett, and Taylor in \cite{Bergerons-hyper1,  Bergerons-general, Bergerons-hyper2}. Atkinson recovered some of Garsia--Reutenauer's results using other methods in \cite{atkinson}. Bidigare's theorem provided a new perspective on the representation theory of the descent algebra, studied by Saliola in \cite{saliolaquiverdescalgebra}, Aguiar--
Mahajan in \cite{Aguiar-Mahajan}, and Schocker in \cite{Schocker}.

 Both the descent algebra and the face algebra turn out to belong to a well-behaved class of algebras called \emph{elementary algebras} (meaning $A / \rad(A)$ is a commutative algebra isomorphic to $\C^n$ for some $n$). A $\C-$ algebra is elementary precisely if its simple modules are all one-dimensional. Aguiar--Mahajan have a very nice appendix \cite[Appendix D.8]{Aguiar-Mahajan} explaining properties of elementary algebras to which we will sometimes refer.

\subsubsection{Radical} 
Recall that the descent algebra $\Sigma_n$ has dimension $2^{n - 1}.$ In Solomon's original paper \cite{Solomondescentbirth}, he proved the radical $\rad(\Sigma_n)$ has dimension $2^{n - 1} - \# \{\lambda \vdash n\}$ and is spanned by the differences $x_J - x_{J'}$ for which $\alpha(J)$ and $\alpha(J')$ (defined just before \cref{thm:bidigare}) rearrange to the same partition. Hence, $\Sigma_n$ is not semisimple for all $n \geq 3.$
 
\subsubsection{Simples and projective indecomposables }
As we will see in \cref{subsec:idemps}, the idempotents in a cfpoi for the descent algebra $\Sigma_n$ are indexed by integer partitions of $n$. Since $\Sigma_n$ is elementary, each simple $\Sigma_n$-module is one-dimensional. By \cref{eqn:algebra-decomposition}, each idempotent in a cfpoi for $\Sigma_n$ then generates a distinct (up to isomorphism) projective indecomposable module, so the $\Sigma_n$- simples and projective indecomposables are also indexed by partitions. (See also \cite[Theorem D.35]{Aguiar-Mahajan}.) One can see an explicit construction of these simples in \cite[\S 3]{atkinson}. 
We will write \emph{${M_\lambda}$} to denote the $\Sigma_n$-simple indexed by the partition $\lambda$ and  \emph{$P_\lambda$} for the the projective indecomposable which is the projective cover of $M_\lambda.$  Then, \[\Sigma_n = \bigoplus_{\lambda \vdash n}P_\lambda.\] 
\subsubsection{Idempotents}\label{subsec:idemps}
Cfpois for the descent algebras are well-studied. Initially, they were studied from the perspective of Coxeter groups  (see \cite{Bergerons-hyper1,Bergerons-general,  Bergerons-hyper2, garsia-reutenauer}). Using Bidigare's theorem, Saliola provided a geometric perspective by constructing a cfpoi \cite{saliolaquiverdescalgebra} from the cfpois he discovered for the face algebra \cite{saliolaquiverdescalgebra, saliolafacealgebra}. Aguiar--Mahajan further study Saliola's idempotents and generalizations extensively in \cite{Aguiar-Mahajan}. For work connecting both perspectives, see \cite{Brauner2022} and \cite[\S 16]{Aguiar-Mahajan}. For our purposes, the geometric perspective shall prove very useful, so we shall consider cfpois $\{E_\lambda: \lambda \vdash n\}$ of $\left(\C \mathcal{F}_n \right)^{S_n},$ rather than the (equivalent) cfpois $\{\Phi^{-1}(E_\lambda): \lambda \vdash n\}$ of $\Sigma_n,$ where $\Phi$ is Bidigare's antiisomorphism from \cref{thm:bidigare}. 

The idempotents in a cfpoi of the face algebra $\C\mathcal{F}_n$ turn out to be indexed by the set partitions in $\Pi_n$. In \cite[\S 5]{saliolaquiverdescalgebra}, Saliola constructed {explicit cfpois}\footnote{Saliola's construction was quite general and provided multiple possible explicit cfpois (see \cite[\S 5.1, 5.2] {saliolaquiverdescalgebra} for some specific examples). When we write $\{\widehat{E}_X\},$ we mean any of the cfpois for $\C \mathcal{F}_n$ coming from his construction in \cite[Theorem 5.2]{saliolaquiverdescalgebra}.} $\{\widehat{E}_X: X \in \Pi_n\}$ for $\C \mathcal{F}_n$ for which $\pi \left(\widehat{E}_X\right) = \widehat{E}_{\pi(X)}$ for all $\pi \in S_n.$ He proved the $S_n$-orbit sums of such families $\left \{\widehat{E}_\lambda: \lambda \vdash n\right \}$ where $\widehat{E}_\lambda:= \sum_{X \in \lambda} \widehat{E}_X$, form cfpois for $\left(\C \mathcal{F}_n\right)^{S_n}.$ Thus, their images under Bidigare's antiisomorphism form a cfpoi for $\Sigma_n.$ 

Although we will not need Saliola's explicit construction for our purposes, we will repeatedly use that his idempotents $\widehat{E}_{X}$ satisfy three very nice {properties}\footnote{Notation warning: Saliola uses the opposite ordering as this paper on the poset $\Pi_n.$}:

\begin{enumerate}
    \item[(a)] For $f \in \mathcal{F}_n,$ if $\sigma(f)\ngeq X,$ then $f\widehat{E}_X = 0.$ (\cite[Lemma 5.3]{saliolaquiverdescalgebra})
    \item[(b)]$\widehat{E}_X \in \mathrm{span}_\C \left\{f \in \mathcal{F}_n : \sigma(f) \leq X \right\}.$ (\cite[Theorem 5.2]{saliolaquiverdescalgebra})
    \item[(c)] Let \emph{$\mathfrak{c}_X: \C \mathcal{F}_n \to \C$} be the $\C$-linear map which sends $f \in \mathcal{F}_n$ to $1$ if $\sigma(f) = X$ and sends $f$ to $0$ otherwise. Then, $\mathfrak{c}_X\left(\widehat{E}_X\right) = 1.$ (\cite[Theorem 5.2]{saliolaquiverdescalgebra}) 
\end{enumerate}

By conjugating Saliola's idempotents, properties (a) and (b) turn out to hold more generally for {any}\footnote{We assume appropriate indexing here, in that if $u$ is an invertible element of $\C \mathcal{F}_n$ which conjugates $\{\widehat{E}_X:X \in \Pi_n\}$ to produce $\{E_X: X \in \Pi_n\}$, then $E_X = u\widehat{E}_Xu^{-1}$ for all $X \in \Pi_n.$ Put differently, for all $X \in \Pi_n,$ $E_X$ and $\widehat{E}_X$ lift the same idempotent in the unique cfpoi of $\C \mathcal{F}_n / \rad(\C \mathcal{F}_n).$ (See \cite[Thm D.33]{Aguiar-Mahajan}.)} cfpoi of $\C \mathcal{F}_n.$ We should point out that each of the lemmas in the remainder of this idempotent section appear (often with much stronger statements and in more generality) in the remarkable characterizations of cfpois for $\C \mathcal{F}_n$ and $\left(\C \mathcal{F}_n \right)^{S_n}$ given by Aguiar--Mahajan \cite[Chapters 11, 16, Appendix D]{Aguiar-Mahajan}. For simplicity, we state only the minimal information that we will need, but we highly encourage the curious reader to consult their work.

\begin{lemma}
    \label{cor:supports-of-kb-idemps}
    Let $\{E_X: X \in \Pi_n\}$ be \emph{any} cfpoi for $\C \mathcal{F}_n$. The following statements hold. 
    \begin{enumerate}
        \item[(i)] Let $f \in \mathcal{F}_n.$ If $\sigma(f) \ngeq X,$ then $fE_X = 0.$ 
        \item[(ii)]$E_X \in \mathrm{span}_\C \{f \in \mathcal{F}_n: \sigma(f) \leq X\}$. 
        \item[(iii)]  If $Y \ngeq X,$ then $E_Y\C \mathcal{F}_nE_X = 0.$ 
        \item[(iv)] If $\sigma(f) = X,$ then there are constants $d_{f'} \in \C$ such that $fE_X$ can be written as \[fE_X = f + \sum_{f': \sigma(f') < X}d_{f'}f'.\] 
    \end{enumerate}
\end{lemma}

\begin{proof}
 Recall $\{E_X\}$ is conjugate to the family $\{\widehat{E}_X\}$ by some invertible element of $\C\mathcal{F}_n$, say $u:= \sum_{f \in \mathcal{F}_n}c_{f}f.$

\begin{itemize}
    \item[(i)] Let $f \in \mathcal{F}_n$ with $\sigma(f) \ngeq X.$ By \cref{eqn:multipy-faces-meet}, for any $g \in \mathcal{F}_n,  \sigma(fg) = \sigma(f) \wedge \sigma(g) \ngeq X.$ Using property (a) of Saliola's idempotents,
    \begin{align*}
        fE_X = fu\widehat{E}_Xu^{-1} = \left(\sum_{g \in \mathcal{F}_n}c_g(fg)\widehat{E}_X\right)u^{-1} = 0.
    \end{align*}
\item[(ii)] It follows from \cref{eqn:multipy-faces-meet} that $\mathrm{span}_\C \{f \in \mathcal{F}_n: \sigma(f) \leq X\}$ is closed under left and right multiplication by $\C \mathcal{F}_n.$ Property (b) of Saliola's idempotents explains that $\widehat{E}_X$ is in $\mathrm{span}_\C \{f \in \mathcal{F}_n: \sigma(f) \leq X\}.$ Hence, $\widehat{E}_X = uE_Xu^{-1}$ is also in $\mathrm{span}_\C \{f \in \mathcal{F}_n: \sigma(f) \leq X\}.$
\item[(iii)]  By part (ii), $E_Y \in \mathrm{span}_\C \{f\in \mathcal{F}_n: \sigma(f) \leq Y\}.$ For $f$ with $\sigma(f) \leq Y,$ $\sigma(f) \ngeq X$ (since otherwise, $Y \geq \sigma(f) \geq X).$ By part (i), $E_Y \C \mathcal{F}_n E_{X} = 0.$
\item[(iv)] By part (ii), it suffices to prove that the coefficient of $f$ in $fE_X$ is one and the coefficient of $f' \neq f$ for $\sigma(f') = X$ is zero. By property (b) of Saliola's idempotents, there exist constants $d_{f'}$ for which \[fE_X = fu\widehat{E}_Xu^{-1} = fu \left(\sum_{\substack{f' \in \mathcal{F}_n:\\ \sigma(f') \leq X}}d_{f'}f'\right)u^{-1} = \sum_{\substack{f' \in \mathcal{F}_n:\\ \sigma(f') \leq X}}d_{f'}fuf'u^{-1}.\] 
By expanding $u, u^{-1}$ as a linear combination of faces and applying \cref{eqn:multipy-faces-meet}, observe that if $\sigma(f') < X$, then  $\mathfrak{c}_X \left(fuf'u^{-1}\right) = 0.$ Hence, we can ignore the terms with $\sigma(f') < X$ the above sum. From the ordered set partition definition of face multiplication, observe that if $\sigma(f') = X,$ then $ff' = f.$ Using this fact and  \cref{eqn:lrb-property}, \[fuf'u^{-1} = fuff'u^{-1} = fufu^{-1} = fuu^{-1} = f.\]

Hence, the coefficient of $f' \neq f$ with $\sigma(f) = X$ in $fE_X$ is $0$ as desired, and the coefficient of $f$ in $fE_X$ is $\sum_{f': \sigma(f') = X}c_{f'} = 1$ by property (c) of Saliola's idempotents. 
\end{itemize}
\end{proof}

The proof of the following lemma imitates the proofs of Lemma 11.29 in \cite{Aguiar-Mahajan} and Corollary 4.4 of \cite{MSS}.
\begin{lemma}\label{prop:bE_b-form-basis}
    Let $\{E_X: X \in \Pi_n\}$ be any cfpoi for $\C\mathcal{F}_n.$ The set $\left \{fE_{\sigma(f)}: f \in \mathcal{F}_n \right\}$ forms a $\C$-basis for the vector space $\C \mathcal{F}_n.$
\end{lemma}
\begin{proof}
    Pick a total ordering on the faces $f \in \mathcal{F}_n$ such that if $\sigma(f) < \sigma(f'),$ then $f$ precedes $f'.$ Then, the matrix representing the linear transformation $f \mapsto f E_{\sigma(f)}$ is upper triangular with ones along the diagonal by \cref{cor:supports-of-kb-idemps}(iv). Hence,  $\left \{fE_{\sigma(f)}: f \in \mathcal{F}_n \right\}$ is also a basis for $\C \mathcal{F}_n.$
\end{proof}

\begin{lemma}\label{lem:invariant-idems-are-sums-of-idems}

Let $\{E_\lambda: \lambda \vdash n\}$ be any cfpoi for $\left(\C \mathcal{F}_n\right)^{S_n}.$ Then, there exists a cfpoi $\{E_X: X \in \Pi_n\}$ for the face algebra $\C \mathcal{F}_n$ which satisfies all three of the following properties.
\begin{enumerate}
    \item[(i)] $\pi(E_X) = E_{\pi(X)}$ for all $X \in \Pi_n$ and all $\pi \in S_n,$
    \item [(ii)] $E_\lambda = \sum_{X \in \lambda}E_X, $ and
    \item[(iii)] There is an invertible $u \in \left(\C \mathcal{F}_n\right)^{S_n}$ such that $E_X = u\widehat{E}_X u^{-1}$ for all $X \in \Pi_n.$
\end{enumerate}
\end{lemma}

\begin{proof}
    Let $\{E_\lambda: \lambda \vdash n\}$ be a cfpoi for $\left(\C \mathcal{F}_n\right)^{S_n}.$ By \cref{rem:conjugation-idems}, there exists an invertible element $u \in \left(\C \mathcal{F}_n\right)^{S_n}$ such that for each $\lambda \vdash n,$ $E_\lambda = u\widehat{E}_\lambda u^{-1}.$ Hence, for each $\lambda \vdash n,$\[E_\lambda = \sum_{X \in \lambda} u\widehat{E}_X u^{-1}.\]
    By \cref{rem:conjugation-idems}, the family $\left \{u\widehat{E}_Xu^{-1}: X \in \Pi_n\right\}$ is a cfpoi for $\C \mathcal{F}_n.$ Since for $\pi \in S_n,$ $\pi\left(u^{-1}\right) = \left(\pi(u)
\right)^{-1} = u^{-1},$ the inverse $u^{-1}$ is also in $\left(\C \mathcal{F}_n \right)^{S_n}.$  Hence, for $\pi \in S_n,$ \[\pi \left(u \widehat{E}_X u^{-1}\right) = \pi(u)\pi \left(\widehat{E}_X\right)\pi(u^{-1}) = u \widehat{E}_{\pi(X)}u^{-1}.\] Setting $E_X := u \widehat{E}_Xu^{-1}$ completes the proof.


  
\end{proof}

A similar statement to \cref{cor:supports-of-kb-idemps}(ii) holds for the cfpois of the invariant subalgebra $\left(\C \mathcal{F}_n \right)^{S_n}$.
\begin{lemma} \label{lem:invariant-saliola-lemma}
   Let $\{E_\mu: \mu \vdash n\}$ be any choice of cfpoi for $\left (\C \mathcal{F}_n \right)^{S_n}.$ Let $f \in \mathcal{F}_n$ with $\sigma(f) \in \lambda.$ If the partition $\mu$ does not refine $\lambda,$ then 
   \[f E_\mu = 0.\]
\end{lemma}
\begin{proof}
    By \cref{lem:invariant-idems-are-sums-of-idems}, there exists a cfpoi $\{E_X: X \in \Pi_n\}$ for $\C \mathcal{F}_n$ permuted by $S_n$ for which $E_\mu = \sum_{X \in \mu} E_X.$ If $\sigma(f) \geq X$ for some $X \in \mu,$ then the set partition $X$ refines $\sigma(f),$ so $\mu$ would refine $\lambda.$ Hence, $\sigma(f) \ngeq X$ for all $X \in \mu.$ Thus, by \cref{cor:supports-of-kb-idemps}(ii), 

    \begin{align*}
        fE_\mu = \sum_{X \in \mu}f E_X = 0.
    \end{align*}
\end{proof}

For the remainder of the paper, we shall fix:
\begin{itemize}
    \item some cfpoi $\{E_\lambda: \lambda \vdash n\}$ of $\left(\C \mathcal{F}_n\right)^{S_n},$ and
    \item a cfpoi $\{E_X: X \in \Pi_n\}$ for $\C \mathcal{F}_n$ satisfying the three conditions of \cref{lem:invariant-idems-are-sums-of-idems} (given our choice of $\{E_\lambda\}$).
\end{itemize}

\subsection{Representation theory of the symmetric group and wreath products}

We will assume some familiarity with the representation theory of finite groups (especially symmetric groups) as well as symmetric functions. See \cite{sagan2013symmetric}, \cite[Ch. 7]{ec2}, and \cite{Alexandersson2020} for good references. 

Let \emph{$\chi^\lambda$} denote the irreducible character of $S_n$ indexed by partition $\lambda;$ we will often write \emph{$\mathbb{1}_{S_n}$} and \emph{$\sgn_{S_n}$} to denote the trivial and sign characters of $S_n,$ respectively. Throughout this paper, let \emph{$\ch$} denote the \emph{Frobenius characteristic map} from characters of symmetric groups (or their corresponding representations) to the ring of symmetric functions. We will write \emph{$s_\lambda$} for the Schur function indexed by $\lambda$ and \emph{$h_i$} for the complete homogeneous symmetric function of degree $i.$  Recall that
\begin{itemize}
    \item $\ch \left(\chi^\lambda\right) = s_\lambda,$
    \item  $\ch \left(\sgn_{S_n}\right) = \ch \left(\chi^{1^n}\right) = s_{1^n},$ and
    \item $\ch \left(\mathbb{1}_{S_n}\right) = \ch \left(\chi^{n}\right) = s_n = h_n.$
\end{itemize}

We write \emph{$\alpha \vDash n$} if $\alpha = (\alpha_1, \alpha_2, \ldots, \alpha_k)$ is an \emph{integer composition} of $n$ (a sequence of positive integers summing to $n$). We use the convention \emph{$h_\alpha := h_{\alpha_1}h_{\alpha_2} \ldots h_{\alpha_k}$} for $\alpha$ an integer composition, partition, or even a word on the positive integers. Note that $h_{\alpha} = h_{\alpha'}$ if $\alpha$ \emph{rearranges} to $\alpha'$ (which we write as \emph{$\alpha \sim \alpha'$}).
For $\alpha \vDash n,$ we shall view \emph{$S_{\alpha} := S_{\alpha_1} \times S_{\alpha_2} \times \ldots \times S_{\alpha_k}$} as a subgroup of $S_n,$ by considering it as the subgroup which permutes the the first $\alpha_1$ positive integers amongst themselves, the next $\alpha_2$ amongst themselves, and so on.

Recall that for $\rho_1, \rho_2, \ldots, \rho_k$ characters of $S_{\alpha_1}, S_{\alpha_2}, \ldots, S_{\alpha_k},$
\[\ch\left( \left(\rho_1 \otimes \rho_2 \otimes \ldots \otimes \rho_k \right) \Big \uparrow_{S_\alpha}^{S_n}\right) = \ch \left(\rho_1\right) \ch \left(\rho_2\right) \ldots \ch\left(\rho_k\right),\] where the notation \emph{$(-)\uparrow_H^G$} indicates the \emph{induction} of a representation of a subgroup $H$ to the full group $G.$ Later on, we will also consider the \textbf{restriction} of a $G$-representation to an $H$-represenation, written \emph{$(-)\downarrow_H^G.$}

We use the inner product symbol \emph{$\left \langle A, B \right \rangle$} to mean either
\begin{itemize}
    \item the Hall inner product on symmetric functions if $A$ and $B$ are symmetric functions, or
    \item The inner product on the characters of a group $G$ if $A$ and $B$ are $G$-characters (or the inner product on the corresponding characters if $A, B$ are $G$-representations).
\end{itemize}
Recall that for $S_n$-representations $\rho, \psi,$ one has $\left \langle \rho, \psi \right \rangle = \left \langle \ch (\rho), \ch (\psi) \right \rangle.$

We remind the reader of \textbf{Young's rule} (see \cite[\S 2.11]{sagan2013symmetric}). Let $\lambda, \mu$ be partitions of $n.$ Recall the \textbf{Kostka coefficients} $K_\lambda, \mu$ which count semistandard Young tableaux with shape $\lambda$ and content $\mu.$ Young's rule states that
\begin{align} \label{eqn:Young's-rule}
    \left \langle \mathbb{1} \uparrow_{S_\mu}^{S_n}, \chi^\lambda\right \rangle = \left \langle h_\mu, s_\lambda\right \rangle = K_{\lambda, \mu}.
\end{align}

It is well known that $K_{\lambda, \mu} = 0$ if $\lambda$ does not (weakly) dominate $\mu.$

Wreath products of symmetric groups with finite groups will play an important role in this work. Let $G$ be a finite group. The \emph{wreath product $S_n[G]$} is defined to be the group with elements $\{(g_1, g_2, \ldots, g_n; \sigma): g_i \in G, \sigma \in S_n\}$ under the operation \[(g_1, g_2, \ldots, g_n; \sigma)\cdot(h_1, h_2, \ldots, h_n; \tau) = (g_1h_{\sigma^{-1}(1)}, g_2h_{\sigma^{-1}(2)}, \ldots, g_n h_{\sigma^{-1}(n)}; \sigma \tau).\]

Let $V$ be a representation of $S_n$ and let $W$ be a representation of $G$. We write \emph{$V[W]$}, like Wachs \cite{WachsPosetTopology} and Sundaram \cite{SundaramAdvances}, to denote the $S_n[G]$ representation with underlying vector space \[\left (\otimes ^n W \right) \otimes V,\] and action 
\[
\begin{aligned}
&(g_1, g_2, \ldots, g_n; \pi) \cdot (w_1 \otimes w_2 \otimes \ldots \otimes w_n) \otimes v \\
&= \left(g_1 \cdot w_{\pi^{-1}(1)}\otimes g_2 \cdot w_{\pi^{-1}(2)} \otimes \ldots \otimes g_n \cdot w_{\pi^{-1}(n)}\right)\otimes \pi \cdot v.
\end{aligned}
\] 

We shall often consider the wreath product $S_m[S_n]$ as a subgroup of $S_{mn},$ as illustrated by the following example. 

\begin{example}\label{ex:wr-prod-1}
We consider $S_3[S_4]$ as the subgroup of $S_{12}$ generated by the following permutations (in cycle notation):
\begin{itemize}
    \item nine transpositions 
    {\color{red} $(1, 2), (2, 3), (3, 4), (5, 6), (6, 7), (7, 8), (9, 10), (10, 11), (11, 12)$} 
    which generate the subgroup $S_{4^3}$ and are represented by the red arrows in \cref{S3wrS4}, and 
    \item two products of transpositions 
    {\color{blue}$(1, 5)(2, 6)(3,7)(4,8)$} and 
    {\color{blue}$ (5,9)(6,10)(7,11)(8,12)$} 
    which generate the \textit{wholesale} permutations of the three rows in \cref{S3wrS4} and are represented by the blue arrows.
\end{itemize}
    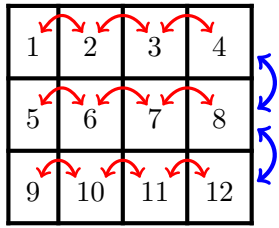
\begin{figure}[h]
\begin{center}
    \renewcommand{\arraystretch}{2}
\begin{NiceTabular}{cccc}[hvlines, create-medium-nodes]
   $1$   & $2$ & $3$ & $4$ \\
  $5$ &  $6$   & $7$   & $8$   \\
  $9$ & $10$   & $11$   & $12$ \CodeAfter
  \begin{tikzpicture} 
        \draw[red, very thick, <->] ([yshift=0.5mm]3-2.north west) .. controls +(-.1,.3) and +(.1,.3) ..  ([yshift=0.5mm]3-1.north east) ; 
    \draw[red, very thick, <->] ([yshift=0.5mm]3-4.north west) .. controls +(-.1,.3) and +(.1,.3) ..  ([yshift=0.5mm]3-3.north east) ; 
    \draw[red, very thick, <->] ([yshift=0.5mm]3-3.north west) .. controls +(-.1,.3) and +(.1,.3) ..  ([yshift=0.5mm]3-2.north east) ; 
        \draw[red, very thick, <->] ([yshift=0.5mm]1-2.north west) .. controls +(-.2,.3) and +(.2,.3) ..  ([yshift=0.5mm]1-1.north east) ; 
    \draw[red, very thick, <->] ([yshift=0.5mm]1-4.north west) .. controls +(-.2,.3) and +(.2,.3) ..  ([yshift=0.5mm]1-3.north east) ; 
    \draw[red, very thick, <->] ([yshift=0.5mm]1-3.north west) .. controls +(-.2,.3) and +(.2,.3) ..  ([yshift=0.5mm]1-2.north east) ; 
      \draw[red, very thick, <->] ([yshift=0.5mm]2-2.north west) .. controls +(-.2,.3) and +(.2,.3) ..  ([yshift=0.5mm]2-1.north east) ; 
    \draw[red, very thick, <->] ([yshift=0.5mm]2-4.north west) .. controls +(-.2,.3) and +(.2,.3) ..  ([yshift=0.5mm]2-3.north east) ; 
    \draw[red, very thick, <->] ([yshift=0.5mm]2-3.north west) .. controls +(-.2,.3) and +(.2,.3) ..  ([yshift=0.5mm]2-2.north east) ; 
    \draw[blue, ultra thick, <->] ([xshift=.4cm]1-4.south east) .. controls +(.3,-.2) and +(.3,.2) ..  ([xshift=.4cm]2-4.north east) ;
    \draw[blue, ultra thick, <->] ([xshift=.4cm]2-4.south east) .. controls +(.3,-.2) and +(.3,.2) ..  ([xshift=.3cm]3-4.north east) ;
  \end{tikzpicture}
\end{NiceTabular}
\end{center}
    \caption{$S_3[S_4] \leq S_{12}$}
    \label{S3wrS4}
\end{figure}

\end{example}

Wreath products are highly related to the notion of \emph{plethysm $f[g]$} of symmetric functions $f, g.$ (We assume familiarity with plethysm and recommend  \cite[Ch. 7 Appendix 2]{ec2}, \cite{plethystic-calculus}, \cite[\S 4]{sundaram1995plethysm} and \cite[\S 1.8]{macdonald} as references.) If $V$ is a representation of $S_m$ with $\ch(V) = f$ and $W$ a representation of $S_n$ with $\ch(W) = g$, then \[\ch\left( V[W] \big \uparrow_{S_m[S_n]}^{S_{mn}}\right) = f[g].\]

The following properties of plethysm hold for general symmetric functions: see equation (8.7) and remark (8.3) of \cite{macdonald}. 

\begin{lemma} \label{prop:plethysm-properties}     Let $f$, $g$, and $h$ be symmmetric functions. Then, 
    \begin{enumerate}
        \item (Associativity of plethysm):$f \left[ g \left [ h \right] \right] = f [g][h].$
        \item $f[h] + g[h] = (f + g)[h]$
        \item $(f[h])\cdot (g[h]) = (fg)[h]$
    \end{enumerate}
\end{lemma} 

It remains an open problem in algebraic combinatorics to understand plethysm coefficients, i.e. to understand how $f[g]$ positively expands into Schur functions for two homogeneous Schur-positive symmetric functions $f, g.$

\section{Refining the question}\label{sec:quest}
By Maschke's theorem, the face algebra $\C\mathcal{F}_n$ decomposes into a direct sum of irreducible $S_n$-representations. Although this decomposition is not unique, the coarser decomposition given by grouping the irreducibles of the same isomorphism type is; this is the decomposition of $\C \mathcal{F}_n$ into its \emph{isotypic subspaces}. Hence, there is an $S_n$-representation decomposition \[\C\mathcal{F}_n = \bigoplus_{\nu \vdash n}\left(\C \mathcal{F}_n \right)^\nu,\] where $\left(\C \mathcal{F}_n\right)^\nu$ is the $S_n$-isotypic subspace associated to $\chi^\nu.$ (In general, for a finite group $G$, an irreducible $G$-character $\chi,$ and a representation $V$ of $G,$ we write $\left(V\right)^\chi$ to denote the $\chi-$isotypic subspace of $V.$) Each isotypic subspace $\left(\C \mathcal{F}_n \right)^\nu$ can be accessed by acting on the left of $\C \mathcal{F}_n$ by the corresponding isotypic projector \emph{$\pi_\nu \in \C S_n,$} defined as
\begin{align*}
    \pi_\nu = \frac{\chi^\nu(1)}{n!}\sum_{w \in S_n}\chi^\nu(w)w.
\end{align*}
 We will use that an element $x \in \C \mathcal{F}_n$ is in the isotypic subspace $\left(\C \mathcal{F}_n\right)^\nu$ if and only if $\pi_\nu(x) = x.$

The trivial isotypic subspace $\left(\C \mathcal{F}_n\right)^{n}$ of the face algebra is precisely its invariant subalgebra $\left(\C \mathcal{F}_n\right)^{S_n}$. Moreover, the isotypic subspaces are not only $S_n$-representations; each carries an additional, rich structure as a left module over $\left(\C \mathcal{F}_n\right)^{S_n}.$ Hence, by Bidigare's theorem, each isotypic subspace is actually a (right) module over the descent algebra $\Sigma_n$ by the action \[f \cdot x := \Phi(x) f \text{ for } f \in \left(\C \mathcal{F}_n\right)^\nu, x \in \Sigma_n.\] 

Since $\Sigma_n$ is not semisimple, we are unable in general to decompose each $\Sigma_n$-module $\left(\C \mathcal{F}_n E_\mu \right)^\nu$ into a direct sum of simple $\Sigma_n$-modules. As discussed in \cref{subsec:rep-theory-fin-dim-algebras}, we take the alternative approach of understanding composition multiplicities and we restate our main question from the introduction (\cref{question:OG-pre}) more explicitly.

\begin{question}\label{question:OG}
What are the $\Sigma_n$-module composition multiplicities
for each $S_n$-isotypic subspace $\left(\mathbb{C} \mathcal{F}_n \right)^{\nu}$?
\end{question}

We will answer \cref{question:OG} with \cref{conj:typeA-comp-factors}. Specifically, we will reduce \cref{question:OG} to understanding specific symmetric group representations, whose structure we analyze up to longstanding open problems. 

\subsection{Preliminary answer}
 As a first step towards answering \cref{question:OG}, we decompose each $S_n$-isotypic subspace $\left(\C \mathcal{F}_n\right)^\nu$ into a direct sum of smaller $\Sigma_n$-modules. 
 
\begin{lemma}\label{prop:decomposing-monoid-algebra} The $\C$- linear map 
\begin{align*}
    \varphi: \mathrm{span}_\C \{f \in \mathcal{F}_n: \sigma(f) \in \mu\} &\longrightarrow \C \mathcal{F}_n E_\mu\\
    f &\mapsto fE_{\sigma(f)}
\end{align*}
is an isomorphism of $S_n$-representations. Hence,  \begin{align*}
   \ch \left( \C \mathcal{F}_n E_{\mu}\right)  = \# \{\alpha \vDash n: \alpha \sim \mu \}\cdot h_\mu.
\end{align*}
\end{lemma}
\begin{proof} 
Since $E_\mu = \sum_{X \in \mu} E_X,$ there is a vector space decomposition, \[\C \mathcal{F}_n E_\mu = \bigoplus_{X \in \mu} \C \mathcal{F}_n E_X.\] Hence if $\sigma(f) \in \mu,$ then $fE_{\sigma(f)} \in \C \mathcal{F}_n E_\mu.$ The map $\varphi$ is injective because \[\varphi \left(\sum_{g}c_g g\right) = 0 \implies \sum_{g}c_g gE_{\sigma(g)} = 0 \implies c_g = 0 \text{ for all }g \in \mathcal{F}_n\] by \cref{prop:bE_b-form-basis}. To show that $\varphi$ is surjective, it suffices to prove that $\{fE_{\sigma(f)}: \sigma(f) = X\}$ spans $\C \mathcal{F}_n E_X.$ Let $yE_X \in \C\mathcal{F}_n E_X.$ By \cref{prop:bE_b-form-basis}, there are constants $c_g$ for which $y = \sum_{g \in \mathcal{F}_n}c_ggE_{\sigma(g)}.$ By orthogonality, 
\[yE_X = \sum_{g \in \mathcal{F}_n}c_ggE_{\sigma(g)}E_X = \sum_{\substack{g \in \mathcal{F}_n:\\ \sigma(g) = X}} c_g g E_{\sigma(g),}\]
as desired. Therefore, $\varphi$ is a vector space isomorphism.

The map $\varphi$ respects the action of $S_n$ since for $f$ with $\sigma(f) \in \mu$, \[\varphi(\pi(f)) = \pi(f) E_{\sigma(\pi f)} = \pi(f) E_{\pi \sigma(f)} = \pi(f)\pi \left(E_{\sigma(f)}\right) = \pi \left(fE_{\sigma(f)}\right) = \pi \varphi(f).\]

The $S_n$-orbits of faces in  $\mathcal{F}_n$ are labelled by integer composition $\alpha \vDash n,$ consisting of the ordered set partitions which have block sizes $\alpha.$ The face-orbits whose supports belongs to the support-orbit labelled  by $\mu$ are those for which $\alpha$ rearranges to $\mu.$ The $S_n$-stabilizer of a face in the $S_n$-orbit labelled by $\alpha$ is conjugate to the subgroup $S_\alpha$. Hence, using that $\varphi$ is a $S_n$-representation isomorphism and basic properties of induced representations (see, for instance \cite[Proposition 4.3.2]{webbrepntheory}),

\begin{align*}
   \ch \left( \C \mathcal{F}_n E_\mu\right) = \ch \left(\bigoplus_{\substack{\alpha \vDash n:\\ \alpha \sim \mu}} \mathbb{1} \big \uparrow_{S_\alpha}^{S_n}\right) = \sum_{\substack{\alpha \vDash n:\\ \alpha \sim \mu}} h_\alpha = \# \{\alpha \vDash n: \alpha \sim \mu\}\cdot h_\mu.
\end{align*}
\end{proof}

We are now in the position to give a preliminary answer to our question. In particular, the following proposition explains the modules $N_\mu^{\nu}$ from the informal statement of the main theorem in the introduction (\cref{thm:informal-main}) as $N_\mu^{\nu} = \left(\C \mathcal{F}_n E_\mu\right)^{\nu}.$ It also provides the promised ``simple combinatorial formulas'' for their dimensions over $\C.$ Recall that $f^\nu$ denotes the dimension of $\chi^\nu$ (or equivalently the number of standard Young tableaux of shape $\nu$) and that we write $\alpha \sim \mu$ if the composition $\alpha$ rearranges to the partition $\mu.$

\begin{prop}\label{prop:decomp-isos}
As (right) $\Sigma_n$-modules,
    \[\left(\C \mathcal{F}_n\right)^{\nu} =\bigoplus_{\mu \vdash n}(\C \mathcal{F}_n E_\mu)^{\nu}.\] The $\C$-dimension of each $\Sigma(W)$-submodule $\left(\C \mathcal{F}_nE_\mu \right)^{\nu}$ is \[f^\nu \cdot \#\{\alpha \vDash n \ | \ \alpha \sim \mu\}\cdot K_{\nu, \mu}.\] 
    In particular, if $\nu$ does not (weakly) dominate $\mu,$ then $\left(\C \mathcal{F}_n E_\mu \right)^\nu = 0.$
\end{prop}

\begin{proof}
Since the idempotents $\left\{E_{\mu}: \mu \vdash n\right\} \in \left (\C \mathcal{F}_n \right)^{S_n}$ are complete and orthogonal, there is decomposition of $\Sigma_n$- representations,
\[\left(\C \mathcal{F}_n\right)^\nu = \bigoplus_{\mu \vdash n}\left(\C \mathcal{F}_n\right)^\nu E_{\mu}.\]

By \cref{lem:commuting-action}, $\pi_\nu (yE_\mu) = \pi_\nu(y) E_\mu$ for all $y \in \C \mathcal{F}_n. $ Hence, 
 \[\left(\C \mathcal{F}_n E_{\mu}\right)^\nu = \left(\C \mathcal{F}_n \right)^\nu E_{\mu}\] and there is a decomposition of $\Sigma_n$-representations \[\left(\C\mathcal{F}_n \right)^\nu = \bigoplus_{\mu \vdash n} \left(\C \mathcal{F}_n E_\mu\right)^\nu.\]

By \cref{prop:decomposing-monoid-algebra}, the $\C$-dimension of $\left(\C \mathcal{F}_n E_{\mu}\right)^\nu$ is $ \dim \left(\chi^\nu\right) \cdot \# \{\alpha \sim \mu: \alpha \vDash n\} \cdot \left \langle s_\nu, h_\mu \right \rangle.$ Recall from \cref{eqn:Young's-rule} that $\left \langle s_\nu, h_\mu \right \rangle= K_{\nu, \mu},$ which is $0$ unless $\nu$ (weakly) dominates $\mu.$ 
\end{proof}

\cref{prop:decomp-isos} reduces \cref{question:OG} to understanding the composition multiplicities of each $\Sigma_n$-module $\left(\C \mathcal{F}_n E_\mu \right)^\nu$ for any two partitions $\mu, \nu$ of $n.$  Thus, we have refined \cref{question:OG} as follows.
\begin{question}\label{quest:2'}
For any triple of partitions $\mu, \nu, \lambda$ of $n,$  what is the composition multiplicity \[[(\C \mathcal{F}_nE_\mu)^{\nu}:M_\lambda]?\]
\end{question}

We can use the finite dimensional algebra tools discussed in \cref{subsec:rep-theory-fin-dim-algebras} to further reduce our question. 

\begin{prop}\label{prop:transform-to-group-rep-q}
The composition multiplicity of the $\Sigma_n$- simple $M_{\lambda}$ in $\left(\C \mathcal{F}_n E_{\mu}\right)^{\nu}$ is 
    \[\left[(\C \mathcal{F}_nE_{\mu})^{\nu}:M_{\lambda}\right] = f^\nu \cdot \left \langle s_\nu, \ch \left(E_{\lambda}\C \mathcal{F}_n E_{\mu}\right) \right \rangle.\]
\end{prop}

\begin{proof}
By \cref{eqn:composition-mults} (adjusted for right modules), 
\begin{align*}
    \left[(\C \mathcal{F}_nE_{\mu})^{\nu}:M_{\lambda}\right] 
    = \dim_{\C} \left((\C \mathcal{F}_nE_{\mu})^{\nu} \Phi^{-1}(E_\lambda) \right)= \dim_{\C} E_{\lambda}\left(\C \mathcal{F}_n E_{\mu}\right)^{\nu}.
\end{align*}

By \cref{lem:commuting-action},  $E_\lambda \pi_\nu \left(y E_\mu\right) = \pi_\nu \left(E_\lambda yE_\mu\right)$ for all $y \in \C \mathcal{F}_n.$ Hence,
\[E_{\lambda}\left(\C \mathcal{F}_n E_{\mu}\right)^\nu = \left(E_{\lambda}\C \mathcal{F}_n E_{\mu}\right)^\nu.\]

Therefore,
\begin{align*}
    \left[(\C \mathcal{F}_nE_{\mu})^{\nu}:M_\lambda\right] &= \dim_\C \left(E_\lambda\C \mathcal{F}_n E_{\mu}\right)^{\nu}\\
    &= \dim_\C \left(\chi^\nu \right) \cdot \left \langle\chi^\nu, E_{\lambda}\C \mathcal{F}_n E_{\mu}  \right \rangle\\
    &= f^\nu \cdot \left \langle s_\nu, \ch \left (E_\lambda \C \mathcal{F}_n E_\mu \right)\right \rangle.
\end{align*}
\end{proof}

Given \cref{prop:transform-to-group-rep-q}, our final conversion of \cref{question:OG} is the question below.

\begin{question}\label{quest:sn-rep-structure}
    What is the Frobenius characteristic of $E_\lambda \C\mathcal{F}_nE_\mu$?
\end{question}

\section{Our answer}\label{sec:full-answer}
\subsection{Higher Lie characters and Lyndon words}\label{subsec:wreathplethysmlielyndon}

Thrall studied a collection of $S_n$-representations in \cite{Thrall} which are (also) indexed by partitions of $n$ and often called the \emph{higher Lie representations}.  We write \emph{$L_\lambda$} to denote the Frobenius image of the higher Lie representation \emph{$\mathscr{L}_\lambda$} associated to $\lambda.$ These representations have several interesting {interpretations} and are closely tied to the free Lie algebra.  For our purposes, it is most revealing to {define}\footnote{This is equivalent to the standard definition by work of Stanley \cite[Theorem 7.3]{stanley_grps_on_postes}, Hanlon \cite[Theorem 4]{Hanlon}, and Klyachko \cite{Klyachko}; see also a combinatorial proof by Barcelo in \cite{BARCELO199093}.} $L_{n}$ as the Frobenius image of the $S_n$-representation carried by the top homology of the order complex of the proper part of the set partition {lattice}\footnote{see \cref{sec:poset-top} for details and our conventions on poset topology} tensored with the sign representation:
\begin{align*}
    L_n := \ch \left(\Tilde{H}_{n - 3} \left(\Pi_n\right) \otimes \sgn_n\right).
\end{align*}

More generally, for a partition $\lambda = 1^{m_1}2^{m_2} \ldots k^{m_k},$ let 
\begin{equation} \label{eqn:def-higher-lie}
    L_\lambda := \prod_{i = 1}^k h_{m_i}[L_i].
\end{equation} Positively expanding $L_\lambda$ into Schur functions is a longstanding open problem known as \emph{Thrall's problem}. The furthest result in this direction is due to Schocker, who wrote $L_\lambda$ as a signed sum of Schur functions \cite{schocker2003multiplicities}. One fact we will use several times is that the higher Lie representations decompose the regular representation of $S_n,$ which implies
\begin{equation}\label{eqn:higher-Lies-decompose-h1n}
    \sum_{\lambda \vdash n}L_\lambda = h_{1^n}.
\end{equation}

    A \emph{Lyndon word} is a finite nonempty word on alphabet $\{1, 2, \ldots \}$  that is lexicographically \textit{strictly} smaller than all of its cyclic rearrangements. It is well-known (see \cite[Theorem 5.1.5]{Lothaire_1997}) that any word on alphabet $\{1, 2, \ldots\}$ factors uniquely into the concatenation of weakly (lexicographically) decreasing Lyndon words. For instance, \[2431122231121 = (243)(112223)(112)(1).\]
    
    Finally, we define some notation in order to state our main theorem. Let $\alpha = (\alpha_1, \alpha_2, \ldots, \alpha_k)$ be an integer composition, partition, or word  on $\{1, 2, \ldots \}.$ Let $\mathbf{y} = \{y_1, y_2, \ldots\}$ be an infinite variable set and $m$ an integer. Define \emph{$\mathbf{y}_\alpha, |\alpha|, \alpha \cdot m, \alpha^m$} as follows.
    \begin{itemize}
        \item $\mathbf{y}_\alpha:= y_{\alpha_1}y_{\alpha_2} \ldots y_{\alpha_k}.$ 
         \item $|\alpha|:= \alpha_1 + \alpha_2 + \ldots + \alpha_k.$
        \item $\alpha \cdot m:= (\alpha_1 \cdot m, \alpha_2 \cdot m, \ldots, \alpha_k \cdot m)$
        \item $\alpha^m:= \underbrace{(\alpha_1, \alpha_2, \ldots, \alpha_k, \alpha_1, \alpha_2, \ldots, \alpha_k, \ldots, \alpha_1, \alpha_2, \ldots, \alpha_k)}_{mk \text{ total parts}}$
    \end{itemize}

\subsection{The main theorem}
We now have the necessary definitions to state our main theorem that was paraphrased as part of \cref{thm:informal-main} in the introduction. Note that we use the convention that there exists an empty partition $().$

\newtheorem*{conj:typeA-comp-factors}{\cref{conj:typeA-comp-factors}}

\begin{conj:typeA-comp-factors} There is an equality of generating functions
\begin{align}
    \sum_{n \geq 0} \sum_{\substack{\lambda \vdash n\\ \mu \vdash n}}  \mathbf{y}_\lambda \mathbf{z}_\mu\cdot \ch(E_\lambda \C \mathcal{F}_n E_\mu) = \prod_{\substack{\text{Lyndon }\\ w}} \sum_{\substack{\text{partition}\\\rho}} \mathbf{y}_{\rho \cdot |w|}\mathbf{z}_{w^{|\rho|}}L_\rho[h_w].
\end{align}
\end{conj:typeA-comp-factors}

\cref{conj:typeA-comp-factors} explains the structure of $\C \mathcal{F}_n$ as a module over $S_n$ and $\Sigma_n$ simultaneously, answering \cref{question:OG}. Indeed, \cref{prop:transform-to-group-rep-q} and \cref{conj:typeA-comp-factors} combine to give
\begin{equation}\label{eqn:answer}
    \left[\left(\C\mathcal{F}_nE_\mu\right)^\nu: M_\lambda\right] = f^\nu \cdot \left \langle s_\nu, \ [\mathbf{y}_\lambda\mathbf{z}_\mu] \prod_{\substack{\text{Lyndon }\\ w}} \sum_{\substack{\text{partition}\\\rho}} \mathbf{y}_{\rho \cdot |w|}\mathbf{z}_{w^{|\rho|}}L_{\rho}[h_w] \right \rangle,
\end{equation}
where \emph{$[\mathbf{y}_\lambda \mathbf{z}_\mu]$} denotes the coefficient of the monomial $\mathbf{y}_\lambda \mathbf{z}_\mu.$ Since Thrall's problem and understanding plethysm coefficients are longstanding open problems in algebraic combinatorics, this is as far as we are able to simplify our answer for now.

\subsection{Example of main theorem}\label{subsec:examples}

As an example of \cref{conj:typeA-comp-factors}, we analyze the case $n = 4$ in the table below. The box in row $\nu$ and column $\mu$ is filled with $[\left(\C \mathcal{F}_4E_\mu\right)^\nu: M_\lambda]$ copies of $\lambda,$ and the numbers in parentheses indicate multiplicities.
\begin{center}\scalebox{.7}{
\begin{tabular}{cc}
     & \Huge{$\mu$} \\[.5em]
 \Huge{$\nu$} \ & 
\begin{tabular}{c||c|c|c|c|c}
\centering
& \Large $4$ &  \Large $3,1$ & \Large $2,2$ & \Large $2, 1, 1$ &  \Large $1, 1, 1, 1$\\
\hline
\hline
 \Large $1,1,1,1$ & & & & & \ydiagram{2,2}\\[2em]
\hline
 \Large {$2, 1, 1$} & & & & \cellcolor{pink}{\ydiagram{4} \ (3) \hspace{.5cm} \ydiagram{3,1} \  (3)} & {\ydiagram{4} \ (3) \hspace{.5cm} \ydiagram{3,1} \ (3)}\\ 
 & & & & \cellcolor{pink}{\ydiagram{2,2} \ (3)} &{\ydiagram{2,1,1}\ (3)\multirow{-2}{*}{}}\\[2em]
\hline
\Large{$2, 2$} & & & \ydiagram{2,2} \ (2) & \cellcolor{pink}{\ydiagram{4} \ (2) \hspace{.5cm} \ydiagram{3,1} \ (2)} & {\ydiagram{3,1} \ (2) \hspace{.5cm} \ydiagram{2,2}\  (2) }\\[1.2em]
 & & &  & \cellcolor{pink}{\ydiagram{2, 1, 1} \ (2)}&{\multirow{-2}{*}{}}\\[2em] 
\hline
\Large{$3, 1$} & & \ydiagram{4} \ (3) \hspace{.5cm}\ydiagram{3,1} \ (3) & \ydiagram{4} \ (3) & \cellcolor{pink}{\ydiagram{4} \ (6) \hspace{.5cm} \ydiagram{3,1} \ (6)} & {\ydiagram{4} \ (3) \hspace{.5cm} \ydiagram{3,1} \ (3)}\\[1.5em] 
 & &  &  & \cellcolor{pink}{\ydiagram{2,2} \ (3) \hspace{.5cm} \ydiagram{2,1,1} \ (3) }& { \ydiagram{2,1,1} \ (3)}\\[2em]
\hline
 \Large{$4$} \normalsize& {\ydiagram{4}} & {\ydiagram{4} \hspace{.5cm} \ydiagram{3,1}}& {\ydiagram{2,2}}& \cellcolor{pink}{\ydiagram{4} \hspace{.5cm}\ydiagram{3,1}} & \ydiagram{1,1,1,1}\\ 
 &  & & & \cellcolor{pink}{{\ydiagram{2,1,1}}} & {}\multirow{-2}{*}{}\\[2em] 
\end{tabular}
\end{tabular}}
\end{center}

To illustrate how to fill this table, we work out the $\Sigma_4$- composition factors of $\left(\C \mathcal{F}_{4}E_{211}\right)^{\nu}$ as $\nu$ varies over partitions of $4,$ i.e. the pink boxes. Although this is the most technical column of the table, it is also the most rich. We invite the reader try the other columns using the same process. 

Each term in the expansion of the generating function $\displaystyle \prod_{\substack{\text{Lyndon }\\ w}} \sum_{\substack{\text{partition}\\\rho}} \mathbf{y}_{\rho \cdot |w|}\mathbf{z}_{w^{|\rho|}}L_{\rho}[h_w]$ is formed by choosing one (potentially empty) partition $\rho$ for each Lyndon word factor $w$ and multiplying the corresponding terms $\mathbf{y}_{\rho \cdot |w|}\mathbf{z}_{w^{|\rho|}}L_{\rho}[h_w].$ To obtain terms with $\mathbf{z}$-weight $\mathbf{z}_{211} = z_2z_1^2,$ the only Lyndon word factors $w$ from which one can choose a nonempty partition are $w = 1, w = 2, w = 12,$ and $w = 112.$ Writing these relevant factors first, our generating function is:
\begin{align*}
\underbrace{\left(\sum_{\rho}\mathbf{y}_{\rho}\mathbf{z}_1^{|\rho|}L_\rho[h_1]\right)}_{w = 1}\underbrace{\left(\sum_{\rho}\mathbf{y}_{\rho\cdot 2}\mathbf{z}_2^{|\rho|}L_\rho[h_2]\right)}_{w = 2}\underbrace{\left(\sum_{\rho}\mathbf{y}_{\rho\cdot 3}\mathbf{z}_{12}^{|\rho|}L_\rho[h_{12}]\right)}_{w = 12}\underbrace{\left(\sum_{\rho}\mathbf{y}_{\rho\cdot 4}\mathbf{z}_{112}^{|\rho|}L_\rho[h_{112}]\right)}_{w = 112}\ldots.
\end{align*}
Labelling each term by the factors $w$ for which a nonempty $\rho$ was chosen, the coefficient of $\mathbf{z}_{211}$ in the generating function is 
\begin{align*}
    \underbrace{\mathbf{y}_{2} L_{2}[h_1]}_{\substack{w = 1\\ \rho = 2}} \cdot \underbrace{\mathbf{y}_{2}L_1[h_2]}_{\substack{w = 2\\ \rho = 1}} + \underbrace{\mathbf{y}_{11} L_{11}[h_1]}_{\substack{w = 1\\ \rho = 11}}\cdot \underbrace{\mathbf{y}_{2}L_1[h_2]}_{\substack{w = 2\\ \rho = 1}} + \underbrace{\mathbf{y}_1L_1[h_1]}_{\substack{w = 1\\ \rho = 1}}\cdot \underbrace{\mathbf{y}_{3}L_1[h_{12}] }_{\substack{w = 12\\ \rho = 1}} + \underbrace{\mathbf{y}_4L_{1}[h_{112}]}_{\substack{w = 112\\ \rho = 1}}\\
    = \mathbf{y}_{22}\left(L_2[h_1]L_1[h_2]\right) + \mathbf{y}_{211}\left(L_{11}[h_1]L_1[h_2]\right) + \mathbf{y}_{31}\left(L_1[h_1]L_1[h_{12}] \right) + \mathbf{y}_4 \left(L_1[h_{112}]\right).
\end{align*}
{Expanding}\footnote{One can do this with SageMath, since plethysm and the higher Lie characters (under the name gessel\_reutenauer symmetric functions) are implemented. } into Schur {functions}, we obtain:
\begin{equation}\label{eqn:ex-pink}
\mathbf{y}_{22} \left(s_{211} + s_{31}\right) + \mathbf{y}_{211} \left(s_4 + s_{22} + s_{31}\right) + {\mathbf{y}_{31}}\normalcolor \left(s_4 + {2s_{31}}\normalcolor + s_{22} + s_{211}\right) + \color{blue}{\mathbf{y}_{4}}\normalcolor\left(s_4 + \color{blue}{2s_{31}}\normalcolor + s_{22} + s_{211}\right).
    \end{equation}
This process reveals how to fill each box of the pink column. For instance, the composition multiplicity of $M_{4}$ in $\left(\C\mathcal{F}_4E_{211}\right)^{31}$ is $6 = 3 \cdot 2$ because $f^{(3,1)} = 3$ and the coefficient of $\mathbf{y}_{4}s_{31}$ in \cref{eqn:ex-pink} is $2,$ as indicated by the blue coloring above.

\subsection{Recovering results of Garsia--Reutenauer and Uyemura-Reyes}\label{sec:recovering-results}

As further examples, we explain how \cref{conj:typeA-comp-factors} specializes to recover results of Garsia--Reutenauer in \cite{garsia-reutenauer} and Uyemura-Reyes in \cite{Uyemura-Reyes}.

\subsubsection{The bottom row ($\nu = n$): Garsia--Reutenauer's Cartan invariants of $\Sigma_n$}
In \cite[Theorem 5.4]{garsia-reutenauer}, Garsia and Reutenauer computed the Cartan invariants\footnote{They actually proved a stronger result by finding a basis for each space $\Phi^{-1}(E_\mu)\Sigma_n \Phi^{-1}(E_\lambda).$} of the descent algebra. Let \emph{$\mathrm{type}(\alpha)$} for a composition $\alpha$ be the partition obtained by reordering $|w_1|, |w_2|, \ldots, |w_k|$ where $w_1 w_2\ldots w_k$ is the unique factorization of $\alpha$ into weakly decreasing (lexicographically) Lyndon words (see \cref{subsec:wreathplethysmlielyndon}). 

\begin{thm}[Garsia--Reutenauer]\label{thm:garsia-reutenauer} The composition multiplicity of the $\Sigma_n$-simple module $M_\lambda$ in the projective indecomposable module $P_\mu$ is
\[[P_\mu: M_\lambda ] = \# \{\alpha \sim \mu : \mathrm{type}(\alpha) = \lambda\}.\]
\end{thm}

\begin{example}
The compositions rearranging to $(2, 1, 1)$ are $(2, 1, 1)$, $(1, 2, 1)$, and $(1, 1, 2)$. By taking the Lyndon types of these compositions, as the table below illustrates, we obtain that the composition factors of $P_{211}$ are one copy each of $M_{211}$, $M_{31}$, and $M_{4}$. Compare this to the box in row $(4)$ and column $(2, 1, 1)$ of the table in \cref{subsec:examples}.
\begin{table}[h]\label{table:excompmulttypea}
    \centering
    \begin{tabular}{c|c|c}
         $\alpha$ & Lyndon Factorization & $\type(\alpha)$  \\
         \hline
         $(2, 1, 1)$ & $(2)(1)(1)$ & $(2, 1, 1)$\\
         $(1,2, 1)$ & $(1,2)(1)$ & $(3, 1)$\\
          $( 1, 1, 2)$ & $(1,1,2)$ & $(4)$
    \end{tabular}
\end{table}
\end{example}
As right descent algebra modules, $P_\mu = \Phi^{-1}(E_\mu)\Sigma_n \cong\left(\C \mathcal{F}_n\right)^{n} E_\mu = \left ( \C \mathcal{F}_n E_\mu \right)^n$ by \cref{lem:commuting-action}. So, by \cref{eqn:answer}, 
\begin{align*}
    [P_\mu: M_\lambda] =  \left \langle s_n, [\mathbf{y}_\lambda\mathbf{z}_\mu] \prod_{\substack{\text{Lyndon }\\ w}} \sum_{\substack{\text{partition}\\\rho}} \mathbf{y}_{\rho \cdot |w|}\mathbf{z}_{w^{|\rho|}}L_{\rho}[h_w] \right \rangle.
\end{align*}

We show  \cref{conj:typeA-comp-factors} recovers Garsia--Reutenauer's result by computing the composition multiplicity $[P_\mu: M_\lambda]$ using the above equation. To do so, we will need \cref{lem: mult-of-triv}. Our proof of \cref{lem: mult-of-triv} uses the following fact (appearing as Lemma 2.25 in \cite{Albach-Swanson}) which will also be useful for a proof later on. 

\begin{lemma}\label{conj:induction-wreath-prods}(\cite[Lemma 2.25]{Albach-Swanson}).
Let $H$ be a subgroup of $G$. Let $W$ be an $H$-representation and $V$ be an $S_n$-representation. Then, as $S_n[G]$ representations,
\[V[W] \Big \uparrow_{S_n[H]}^{S_n[G]} \cong V\left[  W \big \uparrow_H^{G} \right].\]
\end{lemma}

\begin{lemma} \label{lem: mult-of-triv}
Let $\nu, \mu$ be partitions of $j, k$ respectively. Set $n = j\cdot k.$ Then,
\begin{align*}
   \left  \langle s_n, L_{\nu}[h_\mu]\right  \rangle \  = \begin{cases}
    1 &\text{ if $\nu = 1^{j}$},\\
    0 &\text{ otherwise.}
    \end{cases}
\end{align*}
\end{lemma}
\begin{proof} 
We first show that $\sum_{\nu \vdash j}\left \langle s_n, L_{\nu}[h_\mu] \right \rangle = 1.$  By linearity of the Hall inner product, \cref{prop:plethysm-properties}(2), and \cref{eqn:higher-Lies-decompose-h1n},
\begin{align*}
    \sum_{\nu \vdash j}\langle s_n, L_{\nu}[h_\mu] \rangle \  = \left \langle s_n, \sum_{\nu \vdash j} L_{\nu}[h_\mu] \right \rangle \ = \left \langle s_n, \left(\sum_{\nu \vdash j} L_{\nu}\right)[h_\mu] \right \rangle = \left \langle s_n, h_{1^j}[h_\mu]\right \rangle. 
\end{align*}
By \cref{prop:plethysm-properties}(3) and the fact that $h_1[f] = f,$ we can simplify the above to 
\begin{align*}
\left \langle s_n, \left(h_1[h_\mu]\right)^j\right \rangle = \left \langle s_n, \left(h_\mu\right)^j\right \rangle=\left \langle s_n, h_{\mu^j} \right \rangle.
\end{align*}
\cref{eqn:Young's-rule} then implies that 
\begin{align*}
    \sum_{\nu \vdash j} \left\langle s_n, L_{\nu}[h_\mu] \right\rangle = \langle s_n, h_{\mu^j}\rangle  = K_{n, \mu^j} = 1.
\end{align*}

Since $L_\nu[h_\mu]$ is the Frobenius characteristic of an $S_n$-representation, namely $\mathscr{L}_\nu\left[\mathbb{1} \uparrow_{S_\mu}^{S_k}\right] \Big \uparrow_{S_{j}[S_k]}^{S_n},$  we know that each $L_\nu[h_\mu]$ is a nonnegative integer linear combination of Schur functions. Thus, there is exactly one partition $\nu \vdash j$ for which $\left \langle s_n,L_\nu[h_\mu] \right \rangle$ is $1$, and $\left \langle s_n,L_\nu[h_\nu] \right \rangle$ must be $0$ otherwise. So, to prove the lemma, it suffices to prove that $\left\langle s_n, L_{(1^j)}[h_\mu] \right \rangle = 1$.

Since  $ L_{1^j} = h_j[L_1] = h_j,$ the inner product $\left \langle s_n, L_{1^j}[h_\mu] \right\rangle 
   = \left \langle s_n, h_j[h_\mu] \right\rangle.$ By \cref{conj:induction-wreath-prods} and transitivity of induction, $h_j[h_\mu]$ is the Frobenius image of \[\mathbb{1}\left[\mathbb{1} \uparrow_{S_\mu}^{S_k}\right]\Big \uparrow_{S_j[S_k]}^{S_n}\cong \mathbb{1}\left[\mathbb{1}\right]\big \uparrow_{S_j[S_\mu]}^{S_j[S_k]} \Big \uparrow_{S_j[S_k]}^{S_n} \cong \mathbb{1}\left[\mathbb{1}\right] \Big \uparrow_{S_j[S_\mu]}^{S_n} \cong \mathbb{1}\uparrow_{S_j[S_\mu]}^{S_n},\] which is the coset representation on $S_n/S_j[S_\mu].$ As a transitive permutation representation, it contains the the trivial representation with multiplicity one, completing the lemma. 
\end{proof}

\begin{prop} \label{prop:recovering-GR}
\cref{conj:typeA-comp-factors} recovers Garsia--Reutenauer's Cartan invariants of $\Sigma_n.$ Namely, \begin{align*}
  \left \langle s_n, [\mathbf{y}_\lambda\mathbf{z}_\mu] \prod_{\substack{\text{Lyndon }\\ w}} \sum_{\substack{\text{partition}\\\rho}} \mathbf{y}_{\rho \cdot |w|}\mathbf{z}_{w^{|\rho|}}L_{\rho}[h_w] \right \rangle =  \# \{\alpha \sim \mu : \mathrm{type}(\alpha) = \lambda\}.
\end{align*}
\end{prop}
\begin{proof} 
Observe that 
\begin{align*}
[\mathbf{y}_\lambda \mathbf{z}_\mu]\prod_{\substack{\text{Lyndon }\\ w}} \sum_{\substack{\text{partition}\\\rho}} \mathbf{y}_{\rho \cdot |w|}\mathbf{z}_{w^{|\rho|}}L_{\rho}[h_w] &= \sum_{\substack{(\rho^w)}}[\mathbf{y}_\lambda \mathbf{z}_\mu] \left(\prod_{w}\mathbf{y}_{\rho^w \cdot |w|}\mathbf{z}_{w^{|\rho^w|}}L_{\rho^w}[h_w]\right).
\end{align*} where the sum on the right side is over \textit{vectors} $(\rho^w)$ of partitions (one partition $\rho^w$ for each Lyndon word $w$) with only finitely many of the partitions nonempty. Extracting the coefficient of $\mathbf{y}_\lambda \mathbf{z}_\mu$ from each term, we simplify the above to 
\begin{align*}
\sum_{\substack{(\rho^w) \in S(\lambda, \mu)}}\prod_w L_{\rho^w}[h_w],
\end{align*}
where $S(\lambda, \mu)$ is the set of partition vectors $(\rho^w)$ for which 
\begin{itemize}
    \item the concatenation of the partitions $\rho^w\cdot|w|$ rearranges to $\lambda$, and 
    \item the concatenation of the partitions $w^{|\rho^w|}$ rearranges to $\mu.$
\end{itemize}
By decomposing each $f_i$ into Schur functions and applying the Littlewood--Richardson rule, one can show that if $f_1, f_2, \ldots, f_k$ are homogeneous symmetric functions of degree $m_1, m_2, \ldots, m_k$ and $n = \sum_{i = 1}^k m_i,$ then 
\[\left \langle s_n, \prod_{i = 1}^k f_i \right \rangle = \prod_{i = 1}^k \left \langle s_{m_i}, f_i\right \rangle.\]

Applying this fact, we have

\begin{align*}
     \sum_{(\rho^w) \in S(\lambda, \mu)} \left \langle s_n,  \prod_w L_{\rho^w}[h_w] \right \rangle= \sum_{(\rho^w) \in S(\lambda, \mu)}   \prod_w \left \langle s_{(|\rho^w|\cdot |w|)},L_{\rho^w}[h_w] \right \rangle.
    \end{align*}

    By \cref{lem: mult-of-triv}, for a fixed $(\rho^w) \in S(\lambda, \mu),$ the above is equivalent to 
    \begin{align*}
        \prod_w \left \langle s_{(|\rho^w|\cdot |w|)},L_{\rho^w}[h_w] \right \rangle = \begin{cases}
            1 &\text{if $\rho^w = 1^{|\rho^w|}$ for each $w,$}\\
            0 & \text{otherwise.}
        \end{cases}
    \end{align*}

    Hence, $\displaystyle \sum_{(\rho^w) \in S(\lambda, \mu)}   \prod_w \left \langle s_{(|\rho^w|\cdot |w|)},L_{\rho^w}[h_w] \right \rangle$ counts the number of ways to pick a vector $(m_w)$ of nonnegative integers (one integer $m_w$ for each Lyndon word $w$ with only finitely many $m_w \neq 0$) so that the vector $(1^{m_w}) \in S(\lambda, \mu).$

The set of such vectors is in bijection with the set $\{\alpha \sim \mu: \type(\alpha) = \lambda\}$ . Specifically, consider some vector $(m_w)$ of nonnegative integers satisfying the described conditions. The composition obtained by concatenating $m_w$ copies of each Lyndon word $w$ for which $m_w \neq 0$ (in lexicographically weakly decreasing order) is necessarily a composition in the set $\{\alpha \sim \mu: \type(\alpha) = \lambda\}.$ For example, if $m_{112} = 2,$ $m_{14} = 3,$ and $m_w = 0$ for all other $w,$ then the vector $(1^{m_w})$ is in $S(5^34^2, 4^32^21^5)$ and its corresponding composition $141414112112$ rearranges to $4^32^21^5$ and has Lyndon type $5^34^2.$ Since Lyndon factorization is unique and always exists, this map defines a bijection, completing the proof.
\end{proof}

\subsubsection{The rightmost column $(\mu = 1^n)$: Uyemura-Reyes's shuffling representations}\label{sec:Uyemura--Reyes}
    It is simple to check that the coefficient of $\mathbf{z}_{1^n}$ in \cref{eqn:ogf} is \[\displaystyle \sum_{\substack{ \lambda \vdash n}}\mathbf{y}_{\lambda}L_\lambda[h_1] = \sum_{\substack{\lambda \vdash n}}\mathbf{y}_{\lambda}L_\lambda.\] Hence, \cref{conj:typeA-comp-factors} reveals that 
    \begin{align} \label{eqn:face-algebra}
        [\left(\C \mathcal{F}_n E_{1^n}\right)^\nu: M_\lambda] = f^\nu \cdot \left \langle s_\nu, L_\lambda\right \rangle.
    \end{align}

    This turns out to recover a result of Uyemura--Reyes which was originally conjectured by R. Stanley. In particular, since $\Sigma_n$ is an elementary algebra, the (potentially non-distinct) {eigenvalues }
     of an element $x \in \Sigma_n$ acting on the \textit{right} of the $S_n$-isotypic subspace $\left(\C S_n\right)^\nu$ {are} indexed by partitions $\lambda \vdash n.$ The eigenvalue associated to $\lambda$ is $\chi_{M_\lambda}(x),$ where $\chi_{M_\lambda}$ is the character associated to the simple $\Sigma_n$-representation $M_\lambda$ and has multiplicity $[\left(\C S_n\right)^\nu: M_\lambda];$ see \cite[Theorem D.38]{Aguiar-Mahajan}. In \cite[Theorem 4.1]{Uyemura-Reyes}, Uyemura--Reyes studied the $S_n-$representation-structure of the eigenspace associated to $\lambda$ for an element $x \in \Sigma_n$ acting semisimply on $\C S_n$. His work gives that the Frobenius image of the $\lambda$-eigenspace is $L_\lambda.$ Hence,
     \begin{align}\label{eqn:descent-algebra}[\left(\C S_n\right)^\nu: M_\lambda] = f^\nu \cdot \left \langle s_\nu, L_\lambda \right \rangle.\end{align}
    
    \cref{eqn:face-algebra} and \cref{eqn:descent-algebra} are actually equivalent. Specifically, it turns out out that $\C \mathcal{F}_n E_{1^n}$ is \textit{equal to} the $\C$-span of the \emph{chambers} (top-dimensional faces) of $\mathcal{B}_n.$ (This follows from the fact that $E_{1^n}$ can be written as the sum $\sum_{X \in 1^n}E_X$ and each $E_X$ is a sum of chambers by \cref{cor:supports-of-kb-idemps}(ii).) The symmetric group acts simply transitively on the chambers which provides an $S_n$-isomorphism between the chamber space and the group algebra. Under this $S_n$-isomorphism, acting by $x$ on the right of $\C S_n$ is equivalent to acting on the left of the chamber space by $\Phi^{-1}(x)$ (see, for instance, \cite[Theorem 8]{BrownonLRBs}).

In \cite[Theorem 4.2]{reiner2014spectra}, Reiner--Saliola--Welker generalize the result of Uyemura--Reyes to all Coxeter groups in terms of the (twisted) cohomology of lower intervals in the lattice of intersections. Their result recovers the case $\mu = 1^n$ in \cref{cor:induced-cohomology} and more generally $[K] = [\emptyset]$ in \cref{gen-of-cor:induced-cohomology}. 

\subsection{More explicit answer for the sign isotypic subspace}
Although \cref{conj:typeA-comp-factors} is as complete of an answer as we can provide for most isotypic subspaces, we are able to provide a more explicit answer for the sign isotypic subspace.
\begin{prop}\label{prop:sign-iso}
As $\Sigma_n$-modules, the sign isotypic subspace $\left(\C \mathcal{F}_n\right)^{1^n}$ of the face algebra is isomorphic to $M_\lambda$ where $\lambda = \left(2^{\frac{n}{2}}\right)$ if $n$ is even and $\lambda = \left(2^{\frac{n - 1}{2}}, 1\right)$ if $n$ is odd.
\end{prop}

\begin{proof}
By \cref{prop:decomp-isos} the dimension of $\left(\C \mathcal{F}_nE_\mu\right)^{1^n}$ is $\# \{\alpha \vDash n : \alpha \sim \mu \} \cdot K_{1^n, \mu},$ which is one if $\mu = 1^n$ and $0$ otherwise. Thus, $\left(\C \mathcal{F}_n\right)^{1^n} = \left(\C \mathcal{F}_nE_{1^n}\right)^{1^n}$ and is one dimensional. To determine the $\Sigma_n$-isomorphism type of $\left(\C \mathcal{F}_n\right)^{1^n} = \left(\C \mathcal{F}_nE_{1^n}\right)^{1^n},$ it suffices to find the unique $\lambda \vdash n$ for which $\left[ \left(\C \mathcal{F}_n E_{1^n}\right)^{1^n}: M_\lambda \right]$ is nonzero. By \cref{eqn:face-algebra},\[\left[ \left(\C \mathcal{F}_n E_{1^n}\right)^{1^n}: M_\lambda \right] = \left \langle s_{1^n}, L_\lambda \right \rangle.\]

A special case of a beautiful result of Gessel--Reutenauer (see \cite[Theorem 2.1]{gesselreutenauer}) explains that $\langle L_\lambda, s_{1^n}\rangle$ counts permutations in $S_n$ with cycle type $\lambda$ and descent set $\{1, 2, \ldots, n - 1\}.$ Hence, the scalar product is zero except for when $\lambda$ is the cycle type of the longest word of $S_n$, which is well known to be $\left(2^{\frac{n}{2}}\right)$ for $n$ even and $\left(2^{\frac{n - 1}{2}}, 1\right)$ for $n$ odd.

\end{proof}

\section{Proof of the main theorem}\label{sec:proof}

    \subsection{Outline of proof}\label{sec:outline-proof}
        The proof of \cref{conj:typeA-comp-factors} is quite long. There are also portions which may be of interest in their own right, involving the homology of intervals in the set partition lattice and the interpretation of $L_n$ in terms of necklaces. In this section, we outline the main ideas in an attempt to (1) make our involved proof more digestible and (2) guide the curious reader towards these ingredients of potential interest. All gaps in this proof outline will be filled in Sections \ref{sec:poset-top}, \ref{sec:base-case}, \ref{sec:general-case}, and \ref{subsec:final-step}.

\vskip.1in
\noindent
{\sf \bf{Step 1. Reduction to homology of intervals in the set partition lattice.}}  Let \emph{$\stab_{S_n}(X)$} denote the stabilizer subgroup of the set partition $X.$ We shall prove the following proposition.

\newtheorem*{prop:double-idemp-sapces}{Proposition \ref{prop:double-idemp-sapces}}
\begin{prop:double-idemp-sapces}
If $\mu$ does not refine $\lambda,$ then $E_\lambda \C \mathcal{F}_n E_\mu = 0.$ Otherwise, as $S_n$-representations,
    \begin{align*}
        E_{\lambda}\C \mathcal{F}_nE_{\mu} \cong \bigoplus_{ [X \leq Y]} E_Y \C \mathcal{F}_n E_X\Big\uparrow_{\stab_{S_n}(X) \cap \stab_{S_n}(Y)}^{S_n},
    \end{align*}
     where the direct sum is over $S_n$-orbits of pairs $X \leq Y$ in $\Pi_n$ with $X \in \mu, Y \in \lambda.$
\end{prop:double-idemp-sapces}
\cref{prop:double-idemp-sapces} reduces our proof to understanding the $\stab_{S_n}(X) \cap \stab_{S_n}(Y)$- representations $E_Y \C \mathcal{F}_n E_X.$ A twisting character appears when studying these representations. The set partition lattice $\Pi_n$ is ($S_n$-equivariantly) isomorphic to the  lattice of intersections of the hyperplanes of $\mathcal{B}_n$.  Define \emph{$\det(Y)$} be the $\pm 1$-valued $\stab_{S_n}(Y)$-character sending $\pi$ to its determinant when restricted to the subspace associated to $Y.$  For more details about this determinant character, see \cref{sec:det-char}.

Saliola explicitly proves the non-equivariant version of the following lemma in \cite[\S 10.2]{saliolafacealgebra}, and the  determinant twists making it equivariant appear implicitly in his work in \cite[Theorem 6.2]{saliolaquiverdescalgebra}. Aguiar and Mahajan also state it in a more general context in \cite[Proposition 14.44]{Aguiar-Mahajan}. We will include a proof for completeness. Let the \emph{length $\ell(X)$} of a set partition mean the number of blocks.

\newtheorem*{prop:change-coho}{Lemma \ref{prop:change-coho}}
\begin{prop:change-coho}
  Assume $X, Y \in \Pi_n$ with $X$ refining $Y$ and $\ell(X) - \ell(Y) = k.$ Then, as representations of the subgroup $\stab_{S_n}(X) \cap \stab_{S_n}(Y)$, 
    \[E_Y \C \mathcal{F}_n E_X \cong \Tilde{H}^{k - 2}\left(X, Y\right) \otimes \det(Y) \otimes \det(X),\] where $\Tilde{H}^{k - 2}\left(X, Y\right)$ is the reduced cohomology of the order complex of the open interval $(X, Y)$ in $\Pi_n.$  
\end{prop:change-coho}
Note that we shall explain our poset topology conventions in \cref{sec:poset-top}. Combining \cref{prop:double-idemp-sapces}, \cref{prop:change-coho}, and properties of dual representations, we prove the following reformulation. For a partition $\mu,$ let \emph{ $\ell(\mu)$} denote its \emph{length.}

\newtheorem*{cor:induced-cohomology}{\cref{cor:induced-cohomology}}

\begin{cor:induced-cohomology}
If $\mu$ does not refine $\lambda,$ then $E_\lambda \C \mathcal{F}_n E_\mu = 0.$ Otherwise, let $\ell(\mu) - \ell(\lambda) = k.$ Then, as $S_n$-representations,
    \begin{align*}
        E_{\lambda}\C \mathcal{F}_nE_{\mu} \cong \bigoplus_{ [X \leq Y]} \left(\Tilde{H}_{k - 2}\left(X, Y\right) \otimes \det(Y) \otimes \det(X)\right)\Big\uparrow_{\stab_{S_n}(X) \cap \stab_{S_n}(Y)}^{S_n}
    \end{align*}
     where the direct sum is over $S_n$-orbits of pairs $X \leq Y$ in $\Pi_n$ with $X \in \mu, Y \in \lambda$ and $\Tilde{H}_{k - 2}$ means the reduced homology.
\end{cor:induced-cohomology}

The proofs of \cref{prop:double-idemp-sapces}, \cref{prop:change-coho}, and \cref{cor:induced-cohomology} appear in \cref{sec:poset-top}.

\vskip.1in
\noindent
{\sf\bf{Step 2: Understand upper intervals}}.
The next key step in proving \cref{conj:typeA-comp-factors} is understanding $\ch \left(E_\lambda \C \mathcal{F}_n E_\mu\right)$ in the special case $\lambda = n.$ We do so by proving the following proposition. 

\begin{prop}\label{prop:base-case}
\begin{align*}
     \sum_{\mu \neq \emptyset}\mathbf{z}_\mu \cdot \ch \left(E_{|\mu|}\C \mathcal{F}_{|\mu|}E_\mu \right) = \sum_{\substack{\text{Lyndon}\\w}} \ \sum_{m \geq 1}\mathbf{z}_{w^m}\cdot L_m[h_w].
\end{align*}
\end{prop}

Indeed, if $\{X_\mu: \mu \vdash n\}$ are set partitions where $X_\mu$ has block sizes $\mu$, then by \cref{cor:induced-cohomology}, the left side of \cref{prop:base-case} can be rewritten as 
\begin{align} \label{eqn:base-case-eqn}
     \sum_{\mu \neq \emptyset} \mathbf{z}_\mu \cdot \ch \left({\Tilde{H}}_{\ell(\mu) - 1}\left(X_\mu, \hat{1}\right) \otimes \det\left(X_\mu\right)\Big \uparrow_{\stab_{S_{|\mu|}}\left(X_\mu\right)}^{S_{|\mu|}} \right) .
\end{align}Sundaram studied the homology of the partition lattice in great depth in \cite{SundaramAdvances}. In \cite[proof of Thm 1.4]{SundaramAdvances}, she studies the $\stab_{S_{|\mu|}}\left(X_\mu\right)$-representations ${\Tilde{H}}_{\ell(\mu) - 1}\left(X_\mu, \hat{1}\right).$ By adjusting her work with the appropriate determinant twists, we can prove the next lemma. (Note that plethysm of formal power series is explained in \cref{def:plethysm-formal-power-series}.)

\newtheorem*{lem:upper-intervals-plethysm}{Lemma \ref{lem:upper-intervals-plethysm}}

\begin{lem:upper-intervals-plethysm}
 \[\sum_{\mu \neq \emptyset} \mathbf{z}_\mu \cdot \ch \left(E_{|\mu|}\C \mathcal{F}_{|\mu|}E_\mu \right)  = \sum_{r \geq 1} L_r[z_1h_1 + z_2h_2  + \ldots].\]
\end{lem:upper-intervals-plethysm}

The symmetric functions $L_r$ can be written as a a generating function counting combinatorial objects called primitive necklaces. Using this interpretation, we construct a technical necklace bijection to prove the following lemma. 

\newtheorem*{lem:necklace-bijection}{\cref{lem:necklace-bijection}}

\begin{lem:necklace-bijection} 
    \[ \sum_{r \geq 1} L_r[z_1h_1 + z_2h_2  + \ldots] = \sum_{\substack{\text{Lyndon}\\w}} \ \sum_{m \geq 1}\mathbf{z}_{w^m}L_m[h_w].\]
\end{lem:necklace-bijection}

To complete the proof of \cref{prop:base-case} it suffices to provide proofs of \cref{lem:upper-intervals-plethysm} and \cref{lem:necklace-bijection}, which we do in \cref{sec:base-case}. 

\vskip.1in
\noindent
{\sf\bf{Step 3: Use upper intervals to understand general intervals}}:
By identifying the intersections of the subgroups $\stab_{S_n}(X) \cap \stab_{S_n}(Y)$ (\cref{lem:normalizer-intersections}), we recast their action on intervals of the partition lattice  as the action of (wreath) products of smaller subgroups on products of smaller partition lattices (with appropriate determinant twists). Using Sundaram's work in \cite{SundaramJerusalem}, we can express the twisted homology representations $\Tilde{H}_{k - 2}(X, Y) \otimes \det(X) \otimes \det(Y)$ in terms of the upper intervals from step 2 (\cref{prop:homology-of-intervals.}).  All of these arguments appear in \cref{sec:general-case}.

\vskip.1in

\noindent{\sf\bf{Step 4: Proof of main theorem}}:
We use step 3 to obtain the Frobenius characteristic of $E_\lambda \C \mathcal{F}_n E_\mu$ in \cref{prop:reduction-to-lambda=n}. With this result, we complete the proof of our main theorem (\cref{conj:typeA-comp-factors}) using a generating function argument. These two arguments appear in \cref{subsec:final-step}.

    \subsection{Reduction to homology of intervals in $\Pi_n$}\label{sec:poset-top}
        In this subsection we fill in the gaps to step 1 of the proof outline in \cref{sec:outline-proof}. We prove \cref{prop:double-idemp-sapces} in \cref{subsubsec:prop5.2proof}. We then give background on the determinant character in \cref{sec:det-char}, on a special map defined by Saliola in \cref{sec:Saliola'smap}, and on poset topology in \cref{sec:poset-top}, so that we can prove \cref{prop:change-coho} and \cref{cor:induced-cohomology} in \cref{sec:precise-conversion}.

\subsubsection{Proof of \cref{prop:double-idemp-sapces}} \label{subsubsec:prop5.2proof}

\begin{prop}\label{prop:double-idemp-sapces}
    If $\mu$ does not refine $\lambda,$ then $E_\lambda \C \mathcal{F}_n E_\mu = 0.$ Otherwise, as $S_n$-representations,
    \begin{align*}
        E_{\lambda}\C \mathcal{F}_nE_{\mu} \cong \bigoplus_{ [X \leq Y]} E_Y \C \mathcal{F}_n E_X\Big\uparrow_{\stab_{S_n}(X) \cap \stab_{S_n}(Y)}^{S_n}
    \end{align*}
     where the direct sum is over $S_n$-orbits of pairs $X \leq Y$ in $\Pi_n$ with $X \in \mu, Y \in \lambda.$
\end{prop}

\begin{proof}
    Since $E_\lambda = \sum_{Y \in \lambda} E_Y$ and $E_\mu = \sum_{X \in \mu}E_X,$ and the idempotents $\{E_X: X \in \Pi_n\}$ are orthogonal, there is a vector space decomposition 
    \begin{align*}
        E_\lambda \C \mathcal{F}_n E_\mu = \bigoplus_{\substack{Y \in \lambda, \\ X \in \mu}} E_Y \C \mathcal{F}_n E_X. 
    \end{align*} By \cref{cor:supports-of-kb-idemps}(iii), the spaces $E_Y \C \mathcal{F}_n E_X$ are zero unless $X \leq Y.$ Hence,
    \begin{align*}
        E_\lambda \C \mathcal{F}_n E_\mu = \bigoplus_{X \leq Y} E_Y \C \mathcal{F}_n E_X, 
    \end{align*}
    where the direct sum is over pairs $X \leq Y$ in $\Pi_n$ with $X \in \mu$, $Y \in \lambda.$ Observe that $\pi\left(E_Y \C \mathcal{F}_n E_X\right) = E_{\pi(Y)} \C \mathcal{F}_n E_{\pi(X)}$ for $\pi \in S_n,$ so the direct summands above are permuted by $S_n.$ The $S_n$-stabilizer of the summand $E_Y \C \mathcal{F}_nE_X$ is $\stab_{S_n}(X) \cap \stab_{S_n}(Y).$ The proposition then follows from standard theory of induced representations (see, for instance, \cite[Proposition 4.3.2]{webbrepntheory}).
\end{proof}

\subsubsection{The determinant character} \label{sec:det-char}
Each set partition $X = \{X_1, X_2, \ldots, X_k\}$ of $[n]$ into $k$ parts is associated to a $k$-dimensional subspace of $\mathbb{R}^n.$ In particular, letting $e_1, e_2, \ldots, e_n$ be the standard basis vectors for $\mathbb{R}^n$, $X$ is associated to the space spanned by the vectors $v_1, v_2, \ldots, v_k,$ where \[v_i:= \sum_{j \in X_i}e_j.\]

An \emph{orientation} of a subspace $V$ of $\mathbb{R}^n$ is some choice of an ordered basis  $\mathcal{V}$ for $V.$ Given this choice, an ordered basis $\mathcal{V}'$ of $V$ is said to be \emph{positive} if the change of basis matrix between $\mathcal{V}$ and $\mathcal{V'}$  has positive determinant and \emph{negative} otherwise.

A permutation $\pi \in S_n$ maps an ordered basis of the subspace associated to set partition $X$ to an ordered basis of the subspace associated to $\pi X.$ If one fixes an orientation on each subspace of $\mathbb{R}^n$ associated to a set partition of $\Pi_n,$ define \emph{$\sigma_X(\pi)$} to be $+1$ if $\pi$ maps a positively oriented basis of the space associated to $X$ to a positively oriented basis of the space associated to $\pi X$ and $-1$ otherwise.

\textit{Crucially}, note that when $\pi \in \stab_{S_n}(X),$ then $\sigma_X(\pi)$ does \textit{not} depend on the choice of orientation. In this situation, $\sigma_X(\pi) = \det(X)(\pi).$ (Recall that the \emph{determinant character }\emph{$\det(X)$} is the $\pm 1$-valued $\stab_{S_n}(X)$ character sending $\pi$ to the determinant of its action on the subspace associated to $X$). We provide an example of $\det(X)$ below.

\begin{example}
    Let $X = \{15, 26, 37, 48\}.$ One choice of orientation for the subspace of $\mathbb{R}^8$ associated to $X$ is $(v_1, v_2, v_3, v_4)$ with $v_i = e_i + e_{i + 4}.$ Using cycle notation, the permutation $\pi = (1, 5)(3, 8)(4, 7)(2, 6)$ lives in $\stab_{S_n}(X)$ and (expressed with respect to  the chosen orientation) has matrix $\left[ \begin{smallmatrix}
       1 & 0 & 0 & 0\\
       0 & 1 & 0 & 0\\
       0 & 0 & 0 & 1 \\
       0 & 0 & 1 & 0 
    \end{smallmatrix}\right].$ Since this matrix has determinant $-1$, $\det(X)(\pi) = -1,$ even though $\sgn(\pi) = \det\left(\mathbb{R}^8\right)(\pi) = +1.$
\end{example}

\subsubsection{Saliola's map}\label{sec:Saliola'smap} We assume some basic familiarity with quivers and path algebras in this subsection; see \cite[\S 3]{saliolaquiverdescalgebra} for more details.

In \cite{saliolaquiverdescalgebra, saliolafacealgebra}, Saliola proved the remarkable fact that the \emph{quiver} \emph{$\mathcal{Q}_n$} of the face algebra  $\C \mathcal{F}_n$ \textit{is} the (directed) Hasse diagram of the set partition lattice $\Pi_n$. Namely, the vertices \emph{$\varepsilon_X$} of $\mathcal{Q}_n$ are indexed by set partitions $X$ in $\Pi_n$ and there is an arrow $\varepsilon_X \to \varepsilon_Y$ if and only if $Y$ covers $X$ in $\Pi_n.$ The \emph{path algebra $\C \mathcal{Q}_n$} has all paths in $\mathcal{Q}_n$ as a $\C$-basis; such paths correspond to unrefinable chains of $\Pi_n.$  Given choices of orientations on each subspace associated to a set partition, Saliola defines an action of $S_n$ on $\C \mathcal{Q}_n.$ In particular,
\begin{align}\label{eqn:quiver-action}
    \pi\left(\varepsilon_{X_1} \to \varepsilon_{X_2} \to \ldots \to \varepsilon_{X_k}\right) := \sigma_{X_1}(\pi)\sigma_{X_k}(\pi)\left(\varepsilon_{\pi\left(X_1\right)} \to \varepsilon_{\pi \left(X_2\right)} \to \ldots \to \varepsilon_{\pi \left(X_k\right)}\right).
\end{align}

 By the definition of multiplication in the path algebra, the $\C$-vector space $\varepsilon_Y \C \mathcal{Q}_n \varepsilon_X$ has as a $\C-$basis all paths $\varepsilon_X \to \ldots  \to \varepsilon_Y$ in $\mathcal{Q}_n,$ which correspond with unrefinable chains from $X$ to $Y$ in $\Pi_n.$ Hence, for $X \leq Y$ in $\Pi_n,$ the space $\varepsilon_Y \C \mathcal{Q}_n \varepsilon_X$ is closed under the action of $\stab_{S_n}(X) \cap \stab_{S_n}(Y).$ Importantly, since $\sigma_X(\pi)\sigma_Y(\pi) = \det(X)(\pi)\det(Y)(\pi)$ for $\pi \in \stab_{S_n}(X) \cap \stab_{S_n}(Y),$ the action of $\stab_{S_n}(X) \cap \stab_{S_n}(Y)$ on $\varepsilon_Y \C \mathcal{Q}_n \varepsilon_X$ is \textit{independent} of the choices of orientation.

In \cite[Theorem 6.2]{saliolaquiverdescalgebra}, Saliola defines an explicit, \textit{surjective} algebra homomorphism  \[\varphi: \C \mathcal{Q}_n \twoheadrightarrow \C \mathcal{F}_n.\]  We will not need $\varphi$'s precise definition for our purposes, but we will use two of its  properties, listed below (see \cite[Theorem 6.2]{saliolaquiverdescalgebra}). Recall from \cref{subsec:idemps} that $\{\widehat{E}_X: X \in \Pi_n\}$ are Saliola's explicit {idempotents}\footnote{Saliola's map can be adjusted to work for the general $\{E_X\}$ fixed at the end of \cref{subsec:idemps}. However, to avoid having to explain his proof, we circumvent this issue by sticking with his idempotents and connecting them back to the fixed system $\{E_X\}$ in \cref{lem:saliolas-idems-to-ours-for-hom}}for $\C\mathcal{F}_n. $
\begin{enumerate}
    \item[(i)] $\varphi\left(\varepsilon_X \right) = \widehat{E}_X.$ 
    \item[(ii)] $\varphi$ is $S_n$-equivariant with respect to the action in \cref{eqn:quiver-action}.
\end{enumerate}

Since $ \ \ker \ \varphi \vert_{\varepsilon_Y \C \mathcal{Q}_n \varepsilon_X} = \varepsilon_Y \left( \ker \ \varphi  \right)\varepsilon_X,$  there is an isomorphism of  $\stab_{S_n}(X) \cap \stab_{S_n}(Y)$- representations:
\begin{align}\label{quiver-double-idem-spaces} 
\varepsilon_Y \C \mathcal{Q}_n \varepsilon_X\  / \varepsilon_Y \left( \ker \ \varphi  \right)\varepsilon_X &\longrightarrow \widehat{E}_Y \C \mathcal{F}_n \widehat{E}_X \\
p + \varepsilon_Y \left(\ker \varphi\right) \varepsilon_X &\mapsto \varphi(p).\nonumber 
\end{align}

 In \cite[Theorem 6.2]{saliolaquiverdescalgebra}, Saliola proves the kernel of $\varphi$ is the two-sided ideal of $\C \mathcal{Q}_n$ generated by the sum of all the paths of length two -- we shall call this special sum $s$. Write \emph{$\lessdot$} to denote the \emph{covering} relation in $\Pi_n.$

\begin{lemma}\label{lem:kernel}
   $\varepsilon_Y \left( \ker \ \varphi  \right)\varepsilon_X$ is the $\C$-span of the elements 
\begin{align*}
    x_C &:= \sum_{X_i \lessdot Z \lessdot X_{i + 1}}\varepsilon_{X} \to \varepsilon_{X_1} \to \ldots \to \varepsilon_{X_i} \to \color{blue}{\varepsilon_Z}\normalcolor \to \varepsilon_{X_{i + 1}} \to \ldots \to \varepsilon_{X_k} \to \varepsilon_{Y}
    \end{align*}
    as one varies over chains $C$ of the form \[C = X \lessdot X_1 \lessdot \ldots \lessdot X_i \color{blue}{<}\normalcolor  X_{i + 1} \lessdot \ldots \lessdot X_k \lessdot Y\] in $\Pi_n$ where $\ell(X_{i + 1}) - \ell(X_i) = 2.$
\end{lemma}
\begin{proof}
    Observe that \begin{align*}
   \sum_{X_i \lessdot Z \lessdot X_{i + 1}}&\varepsilon_{X} \to \varepsilon_{X_1} \to \ldots \to \varepsilon_{X_i} \to \varepsilon_Z \to \varepsilon_{X_{i + 1}} \to \ldots \to \varepsilon_{X_k} \to \varepsilon_{Y}\\
    &= \left(\varepsilon_{X_{i + 1}} \to \ldots \to \varepsilon_{X_k} \to \varepsilon_{Y}\right)\cdot s\cdot \left(\varepsilon_{X} \to \varepsilon_{X_1} \to \ldots \to \varepsilon_{X_i}\right).
\end{align*}
 Hence, each $x_C$ is  in $\varepsilon_Y \ker \left(\mathcal{Q}_n \right) \varepsilon_X.$ Any element of $\varepsilon_Y \ker \left(\mathcal{Q}_n \right) \varepsilon_X$ can be written as \[\left(\sum_{y} c_{y} y\right) \cdot s \cdot \left(\sum_{x} c_x x\right)  = \sum_{x, y} c_xc_y (y\cdot s \cdot p),\] where the sums are over paths $x$ beginning at $\varepsilon_X$ and paths $y$ ending at $Y.$ The element $(y\cdot s \cdot p)$ is zero if the end of path $x$ and the start of path $y$ are not connected by a path of length two and otherwise is one of the $x_C,$ completing the proof.
\end{proof}

 If $X = Y,$ the restriction of $\ker \  \varphi$ to $\varepsilon_Y \C \mathcal{Q} \varepsilon_X$ is empty, so as $\stab_{S_n}(X)$-representations, $\widehat{E}_X \C \mathcal{F}_n \widehat{E}_X \cong \varepsilon_X \C \mathcal{Q} \varepsilon_X \cong \mathrm{span}_{\C}\{\varepsilon_X\},$ the trivial representation.
 
\subsubsection{Poset homology and cohomology}
We give a (very brief) recollection of poset homology and cohomology here. For more details, see \cite{WachsPosetTopology} as an excellent source. We will be particularly interested in the topology of \textit{open intervals} within posets. Throughout, assume that $X \leq Y$ in some ambient poset $P.$

 Assume first that $X < Y.$ An \emph{$i$-chain} in $(X, Y)$ is a sequence of elements $(p_1 < p_2 < \ldots < p_{i + 1}),$ with each $p_j$ in the open interval $(X, Y).$ Define \emph{$C_i(X, Y)$} to be the $\C$-vector space spanned by the $i$-chains of $(X, Y).$  Note that we allow $i = -1,$ so that $C_{-1}(X, Y) = \C,$ the one-dimensional $\C$-vector space spanned by the empty chain $().$ For all $i < -1$, $C_{i}(X, Y) = 0.$ For $i \geq -1,$ define the boundary map \emph{$\partial_i: C_i(X, Y) \longrightarrow C_{i - 1}(X, Y)$} to be the $\C$-linear map which maps an $i$-chain as
\begin{align*}
    (p_1 < p_2 < \ldots < p_{i + 1}) \mapsto \sum_{j = 1}^{i + 1}(-1)^j (p_1 < p_2 < \ldots < p_{j - 1}< p_{j + 1}< \ldots  < p_{i + 1}).
\end{align*}
 
 Now, for the case $X = Y,$ we set $C_{-2}(X, X) = \C$ and $C_{i}(X, X) = 0$ for all other $i.$ The map $\partial_{-2}$ is defined to be the zero map $\partial_{-2}:C_{-2}(X, Y) = \C \to C_{-3}(X, Y) = 0.$ 
 
 In both cases ($X = Y$ and $X < Y$), these boundary maps give rise to a chain complex 
\begin{align*}
 \ldots \xrightarrow{\partial_4} C_3(X, Y) \xrightarrow{\partial_3} C_2(X, Y)\xrightarrow{\partial_2} C_1(X, Y)\xrightarrow{\partial_1} C_0(X, Y) \xrightarrow{\partial_0} C_{-1} (X, Y)\xrightarrow{\partial_{-1}} C_{-2}(X, Y) \xrightarrow{\partial_{-2}} 0.
\end{align*}
The \emph{$i$-th reduced homology group of $(X, Y)$}, written \emph{$\Tilde{H}_i(X, Y)$} is defined to be \[\Tilde{H}_i(X, Y): = \ker(\partial_i) / \im(\partial_{i + 1}).\]

Note that this homology definition agrees with the ordinary reduced homology of the \textit{order complex} of the open interval $(X, Y)$ in $P.$

We now define the dual notion of poset cohomology. Assume first that $X < Y.$ For $i \geq -1,$ define the dual boundary map \emph{$\partial_i^\ast: C_i(X, Y) \longrightarrow C_{i + 1}(X, Y)$} to be the $\C$-linear map sending
\begin{align*}
    (p_1 < p_2 < \ldots < p_{i + 1}) \mapsto \sum_{j = 1}^{i + 2} \,\, \sum_{p_{j - 1} < p < p_j}(-1)^j (p_1 < \ldots < p_{j - 1} < p < p_j  < \ldots <  p_{i + 1}).
\end{align*}
If $X = Y,$ the map \emph{$\partial_{-2}^\ast$} is the zero map from $ C_{-2}(X, Y) = \C \to C_{-1}(X, Y) = 0.$ Again, in both cases, these maps form a (co-)chain complex
\begin{align*}
 \ldots \xleftarrow{\partial_3^\ast} C_3(X, Y) \xleftarrow{\partial_2^\ast} C_2(X, Y)\xleftarrow{\partial_1^\ast} C_1(X, Y)\xleftarrow{\partial_0^\ast} C_{0}(X, Y) \xleftarrow{\partial_{-1}^\ast} C_{-1}(X, Y) \xleftarrow{\partial_{-2}^\ast} C_{-2}(X, Y) \xleftarrow{} 0.
\end{align*}
For $i \geq -1,$ the \emph{$i$-th reduced cohomology group of $(X, Y)$}, written \emph{$\Tilde{H}^i(X, Y)$} is defined to be \[\Tilde{H}^i(X, Y): = \ker(\partial_i^\ast) / \im(\partial_{i - 1}^\ast).\]

If a poset $P$ has an absolute minimum element $\hat{0}$ and an absolute maximum element $\hat{1},$ we use the convention of writing \emph{$\Tilde{H}_i(P)$} and \emph{$\Tilde{H}^i(P)$} to denote $\Tilde{H}_i(\hat{0}, \hat{1})$ and $\Tilde{H}^i(\hat{0}, \hat{1}),$ respectively.

\begin{remark}\label{rem:top-cohomology}
    Assume $P$ is a graded poset and $X < Y$ in $P.$ Set $r -2$ as the length of the maximal chains of $(X, Y).$ Then, there is precisely one position in which any $(r - 3)$-chain in $C_{r - 3}(X, Y)$ can be refined. Thus, $\partial_{r - 3}^\ast$ sends any $(r - 3)$-chain $ p_1 \lessdot  p_2 \lessdot  \ldots \lessdot p_i \color{blue}{<}\normalcolor  p_{i + 1} \lessdot \ldots \lessdot   p_{r - 2}$ in $(X, Y)$ to 
    \begin{align*}
        (-1)^i \sum_{p_i \lessdot p \lessdot p_{i + 1}}p_1 \lessdot \ldots \lessdot p_i \lessdot \color{blue}{p} \normalcolor \lessdot p_{i + 1} \lessdot \ldots \lessdot p_{r-2}. 
    \end{align*}
    Hence, the $\C$-vector space $\im \left(\partial_{r - 3}^\ast\right)$ is the $\C$-span of elements of the above form as one varies over all $(r - 3)$-chains of $(X, Y).$
\end{remark}

\begin{remark} \label{rem:group-poset}
    Let $P$ be a finite poset. If a finite group $G$ acts on the elements of $P$ by poset automorphisms, then $\stab_G(X) \cap \stab_G(Y)$ acts on each chain group $C_i(X, Y)$ for any $X \leq Y$ in $P.$ If $X < Y,$ we assume here that $\stab_G(X) \cap \stab_G(Y)$ acts trivially on the empty chain spanning $C_{-1}(X, Y).$ Similarly, if $X = Y,$ we assume that $\stab_G(X)$ acts trivially on $C_{-2}(X, X) = \C.$ It can be checked that $G$ commutes with the maps $\partial_i$ and $\partial_i^\ast$. Therefore,  $\Tilde{H}_i\left(X, Y\right)$ and $\Tilde{H}^i\left(X, Y\right)$ are both $\stab_G(X) \cap \stab_G(Y)$-representations.  As Wachs explains in \cite[p31]{WachsPosetTopology}, they are \textit{dual} representations (as we are working over a field, namely $\C$).
\end{remark}

\subsubsection{Proofs of \cref{prop:change-coho} and \cref{cor:induced-cohomology}}\label{sec:precise-conversion}
To prove \cref{prop:change-coho}, we first need a lemma translating between the representations given by our fixed cfpoi of $\C \mathcal{F}_n$ and Saliola's.
\begin{lemma}\label{lem:saliolas-idems-to-ours-for-hom}
    As $\stab_{S_n}(X)\cap \stab_{S_n}(Y)$-representations, 
    \begin{align*}
        \widehat{E}_Y \C \mathcal{F}_n \widehat{E}_X \cong  E_Y \C \mathcal{F}_n E_X.
    \end{align*}
\end{lemma}

\begin{proof}
   Recall from \cref{lem:invariant-idems-are-sums-of-idems} that there exists an invertible element $u \in \left(\C \mathcal{F}_n \right)^{S_n}$ for which $u\widehat{E}_Xu^{-1} = E_X$ and $u\widehat{E}_Yu^{-1} = E_Y.$ The $S_n$-equivariant conjugation automorphism on $\C \mathcal{F}_n$ mapping $m \mapsto umu^{-1}$ restricts to a $\stab_{S_n}(X) \cap \stab_{S_n}(Y)$-equivariant isomorphism $\widehat{E}_Y \C \mathcal{F}_n \widehat{E}_X \to E_Y \C \mathcal{F}_n E_X.$

   
\end{proof}

We are now ready to prove \cref{prop:change-coho}.

\begin{lemma}\label{prop:change-coho}
    Let $X \leq Y \in \Pi_n$ with $\ell(X) - \ell(Y) = k.$ As representations of $\stab_{S_n}(X) \cap \stab_{S_n}(Y),$
    \[E_Y \C \mathcal{F}_n E_X \cong \Tilde{H}^{k - 2}\left(X, Y\right) \otimes \det(Y) \otimes \det(X).\]
\end{lemma}
\begin{proof}
By \cref{lem:saliolas-idems-to-ours-for-hom}, it suffices to prove that \[\widehat{E}_Y \C \mathcal{F}_n \widehat{E}_X \cong \Tilde{H}^{k - 2}\left(X, Y\right) \otimes \det(Y) \otimes \det(X).\]

First, assume that $X = Y$ so that $k = 0.$ As discussed at the end of \cref{sec:Saliola'smap},  $\widehat{E}_X \C \mathcal{F}_n \widehat{E}_X$ carries the structure of the trivial representation of $\stab_{S_n}(X).$ 
On the other hand, $\Tilde{H}^{-2}(X, X)$ carries the trivial representation of $\stab_{S_n}(X),$ as discussed in \cref{rem:group-poset}. Since $\det(X) \otimes \det(X)$ is also the trivial representation, $\Tilde{H}^{-2}\left(X, X\right) \otimes \det(X) \otimes \det(X)$ carries the trivial representation too, proving the lemma in this case.

Now assume that $X < Y,$ so that $k > 0.$ The maximal chains of $(X, Y)$ are of length $k - 2$ and in particular, $C_{k - 1}(X, Y) = 0.$ So, $\ker \left(\partial_{k - 2}^\ast\right) = C_{k - 2}(X, Y)$ is spanned by the unrefinable chains of the open interval $(X, Y).$  Hence, the map $\Psi$ sending the chain $p_1 \lessdot p_2 \lessdot \ldots \lessdot p_{k - 1}$ in $C_{k - 2}(X, Y)$ to the element $\varepsilon_X \to \varepsilon_{p_1} \to \varepsilon_{p_2}\to \ldots \to \varepsilon_{p_{k - 1}} \to \varepsilon_{Y} \in \C \mathcal{Q}_n$ is a vector space isomorphism from $\ker \left( \partial_{k - 2}^\ast \right)$ to $\varepsilon_Y \C \mathcal{Q}_n \varepsilon_X.$ By \cref{rem:top-cohomology} and \cref{lem:kernel}, $\Psi$ restricts to a vector space isomorphism  from $\im \left(\partial_{k - 3}^\ast \right)$ to $\varepsilon_Y \left(\ker \ \varphi \right) \varepsilon_X.$ 
     
By definition of the $S_n$-action on $\C \mathcal{Q}_n$ (\cref{eqn:quiver-action}), $\Psi$ is $\stab_{S_n}(X) \cap \stab_{S_n}(Y)$-equivariant with respect to a twist by $\det(X) \otimes \det(Y).$ Hence, as $\stab_{S_n}(X) \cap \stab_{S_n}(Y)$-representations, 
    \begin{align*}
        \left(\varepsilon_Y \C \mathcal{Q} \varepsilon_X \right)/\left(\varepsilon_Y \left(\ker \ \varphi \right) \varepsilon_X\right) &\cong \left(\ker\left(\partial_{k - 2}^\ast\right)\otimes \det(X) \otimes \det(Y)\right) / \left(\im\left(\partial_{k - 3}^\ast\right)\otimes \det(X) \otimes \det(Y)\right)\\
        &\cong  \left(\ker\left(\partial_{k - 2}^\ast\right) / \im\left(\partial_{k - 3}^\ast \right)\right)\otimes \det(X) \otimes \det(Y) \\
        &\cong \Tilde{H}^{k - 2}\left(X, Y\right) \otimes \det(X) \otimes \det(Y).
    \end{align*}
    The lemma then follows from equation \cref{quiver-double-idem-spaces}.
\end{proof}

\begin{cor}\label{cor:induced-cohomology}
   Let $\ell(\mu) - \ell(\lambda) = k.$ If $\mu$ does not refine $\lambda,$ then $E_\lambda \C \mathcal{F}_n E_\mu = 0.$ Otherwise, as $S_n$-representations,
    \begin{align*}
        E_{\lambda}\C \mathcal{F}_nE_{\mu} \cong \bigoplus_{ [X \leq Y]} \left(\Tilde{H}_{k - 2}\left(X, Y\right) \otimes \det(Y) \otimes \det(X)\right)\Big\uparrow_{\stab_{S_n}(X) \cap \stab_{S_n}(Y)}^{S_n},
    \end{align*}
     where the direct sum is over $S_n$-orbits of pairs $X \leq Y$ in $\Pi_n$ with $X \in \mu, Y \in \lambda.$
\end{cor}

\begin{proof}
 Let $G$ be a finite group and let $V$ and $W$ be representations of $G$ over $\C.$ Write $\ast$ to denote the dual of a representation.
Since the vector space isomorphism $V^\ast \otimes W^\ast \to \left(V \otimes W\right)^\ast$ sending $f \otimes g$ to the functional $(v \otimes w \mapsto f(v)g(w))$ is $G$-equivariant, $\left(V \otimes W\right)^\ast \cong V^\ast \otimes W^\ast$ as $G$-representations.

Observe that the representation $\det(Y) \otimes \det(X)$ of $\stab_{S_n}(X) \cap \stab_{S_n}(Y)$ is self-dual, since it takes values in $\{+1, -1\}.$ Hence, as $\stab_{S_n}(X) \cap \stab_{S_n}(Y)$-representations,
\begin{align*}
    \left(\Tilde{H}_{k - 2}\left(X, Y\right) \otimes \det(Y) \otimes \det(X)\right)^{\ast} &\cong  \left(\Tilde{H}_{k - 2}\left(X, Y\right) \right)^\ast \otimes \left(\det(Y) \otimes \det(X)\right)^\ast\\
    &\cong \Tilde{H}^{k - 2}\left(X, Y\right) \otimes \det(Y) \otimes \det(X).
\end{align*}
Now, note that taking duals commutes with induction (see, for instance Exercise 4.6 of \cite{webbrepntheory}.) Thus, as $S_n$-representations,
\begin{align*}
    E_\lambda \C \mathcal{F}_n E_\mu &\cong \bigoplus_{ [X \leq Y]} \left(E_Y \C \mathcal{F}_n E_X\right)\Big\uparrow_{\stab_{S_n}(X) \cap \stab_{S_n}(Y)}^{S_n}\tag{\cref{prop:double-idemp-sapces}}\\
    &\cong\bigoplus_{ [X \leq Y]} \Tilde{H}^{k - 2}\left(X, Y\right) \otimes \det(Y) \otimes \det(X)\Big\uparrow_{\stab_{S_n}(X) \cap \stab_{S_n}(Y)}^{S_n} \tag{\cref{prop:change-coho}}\\
    &\cong\bigoplus_{ [X \leq Y]} \left(\Tilde{H}_{k - 2}\left(X, Y\right) \otimes \det(Y) \otimes \det(X)\right)^{\ast}\Big\uparrow_{\stab_{S_n}(X) \cap \stab_{S_n}(Y)}^{S_n} \\
    &\cong \bigoplus_{ [X \leq Y]}\left(\Tilde{H}_{k - 2}\left(X, Y\right) \otimes \det(Y) \otimes \det(X)\Big\uparrow_{\stab_{S_n}(X) \cap \stab_{S_n}(Y)}^{S_n}\right)^\ast\\
     &\cong \bigoplus_{ [X \leq Y]}\Tilde{H}_{k - 2}\left(X, Y\right) \otimes \det(Y) \otimes \det(X)\Big\uparrow_{\stab_{S_n}(X) \cap \stab_{S_n}(Y)}^{S_n},
\end{align*}
where the last line follows from the fact that all $S_n$-representations are self-dual (as they are in fact defined over the rational numbers).
\end{proof}

    \subsection{Upper intervals}\label{sec:base-case}
        In this section, we fill in the gaps to step 2 of the proof outline in \cref{sec:outline-proof}. In particular, we complete the proof of our ``base case'' (\cref{prop:base-case}) by proving  \cref{lem:upper-intervals-plethysm} in \cref{sec:reduction-of-base-case} and \cref{lem:necklace-bijection} in \cref{sec:necklace-bij}.

\subsubsection{Stabilizer representations}

We study the two representations involved in \cref{eqn:base-case-eqn}: the homology of upper intervals in the partition lattice and the determinant representation. The former has been studied by Sundaram in \cite{SundaramAdvances}. To explain her result, we first must define some notation.

Let $\lambda = 1^{m_1}2^{m_2}\ldots \vdash n$ be a partition. Given our view of $S_{m_i}[S_i]$ as a subgroup of $S_{m_i i}$ (see \cref{ex:wr-prod-1}), we consider the product of wreath products $\prod_{i}S_{m_i}[S_i]$ as a subgroup of $S_{m_1 + 2 m_2 + 3 m_3 + \ldots}$ in the expected manner. See the example below.

\begin{example}\label{ex:wreath-prod} Let $\lambda = (4^3, 2^2, 1).$ We consider $S_1 \times S_2[S_2] \times S_3[S_4]$ as a subgroup of $S_{17}$ generated by the following permutations
\begin{itemize}
    \item $i = 1, m_i = 1$: None (just the identity)
    \item $i = 2, m_i = 2:$ \color{red}{$(2, 3)$, $(4, 5)$}, \normalcolor and \color{blue}{$(2, 4)(3, 5)$}. \normalcolor
    \item $i = 4, m_i = 3:$ \color{red}{$(6, 7)$, $(7, 8)$, $(8, 9)$, $(10, 11)$, $(11, 12)$, $(12, 13)$, $(14, 15)$, $(15, 16)$, $(16, 17),$} \normalcolor as well as  \color{blue}{$(6,10)(7,11)(8,12)(9,13)$} \normalcolor and \color{blue}{$(10, 14)(11, 15)(12,16)(13,17).$}
\end{itemize}

See \cref{fig:stabx} for a pictorial representation of these generators.
    \begin{figure}
        \begin{center}
    \renewcommand{\arraystretch}{2}

\begin{NiceTabular}{c}[hvlines] $1$
\end{NiceTabular} \hspace{1cm}\begin{NiceTabular}{cc}[hvlines, create-medium-nodes]
   $2$   & $3$  \\
  $4$ &  $5$
  \CodeAfter
  \begin{tikzpicture} 
        \draw[red, very thick, <->] ([yshift=0.5mm]2-2.north west) .. controls +(-.1,.3) and +(.1,.3) ..  ([yshift=0.5mm]2-1.north east) ; 
        \draw[red, very thick, <->] ([yshift=0.5mm]1-2.north west) .. controls +(-.1,.3) and +(.1,.3) ..  ([yshift=0.5mm]1-1.north east) ; 
    \draw[blue, ultra thick, <->] ([xshift=.3cm]1-2.south east) .. controls +(.3,-.1) and +(.3, .1) ..  ([xshift=.3cm]2-2.north east) ; 
  \end{tikzpicture}
\end{NiceTabular}\hspace{1cm}
    \begin{NiceTabular}{cccc}[hvlines, create-medium-nodes]
   $6$   & $7$ & $8$ & $9$ \\
  $10$ &  $11$   & $12$   & $13$   \\
  $14$ & $15$   & $16$   & $17$ \CodeAfter
  \begin{tikzpicture} 
        \draw[red, very thick, <->] ([yshift=0.5mm]3-2.north west) .. controls +(-.1,.3) and +(.1,.3) ..  ([yshift=0.5mm]3-1.north east) ; 
    \draw[red, very thick, <->] ([yshift=0.5mm]3-4.north west) .. controls +(-.1,.3) and +(.1,.3) ..  ([yshift=0.5mm]3-3.north east) ; 
    \draw[red, very thick, <->] ([yshift=0.5mm]3-3.north west) .. controls +(-.1,.3) and +(.1,.3) ..  ([yshift=0.5mm]3-2.north east) ; 
        \draw[red, very thick, <->] ([yshift=0.5mm]1-2.north west) .. controls +(-.2,.3) and +(.2,.3) ..  ([yshift=0.5mm]1-1.north east) ; 
    \draw[red, very thick, <->] ([yshift=0.5mm]1-4.north west) .. controls +(-.2,.3) and +(.2,.3) ..  ([yshift=0.5mm]1-3.north east) ; 
    \draw[red, very thick, <->] ([yshift=0.5mm]1-3.north west) .. controls +(-.2,.3) and +(.2,.3) ..  ([yshift=0.5mm]1-2.north east) ; 
      \draw[red, very thick, <->] ([yshift=0.5mm]2-2.north west) .. controls +(-.1,.3) and +(.1,.3) ..  ([yshift=0.5mm]2-1.north east) ; 
    \draw[red, very thick, <->] ([yshift=0.5mm]2-4.north west) .. controls +(-.1,.3) and +(.1,.3) ..  ([yshift=0.5mm]2-3.north east) ; 
    \draw[red, very thick, <->] ([yshift=0.5mm]2-3.north west) .. controls +(-.1,.3) and +(.1,.3) ..  ([yshift=0.5mm]2-2.north east) ; 
    \draw[blue, ultra thick, <->] ([xshift=.4cm]1-4.south east) .. controls +(.3,-.2) and +(.3,.2) ..  ([xshift=.3cm]2-4.north east) ;
    \draw[blue, ultra thick, <->] ([xshift=.3cm]2-4.south east) .. controls +(.3,-.2) and +(.3,.2) ..  ([xshift=.3cm]3-4.north east) ;
  \end{tikzpicture}
\end{NiceTabular}
\end{center}
        \caption{$S_1 \times S_2[S_2] \times S_3[S_4]$ as a subgroup of $S_{17}$}
        \label{fig:stabx}
    \end{figure}
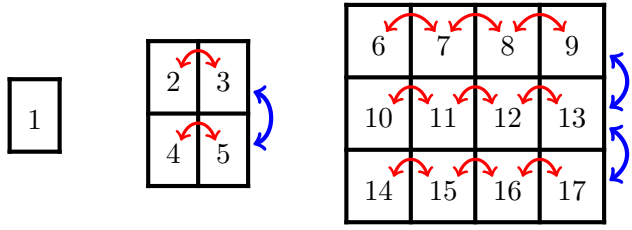
\end{example}

Given a partition $\lambda = 1^{m_1}2^{m_2}\ldots \vdash n,$ define an associated set partition \emph{$X_\lambda$} of $[n]$  with $\{1\},\{2\},\ldots, \{m_1\}$ as its blocks of size one, $\{m_1 + 1, m_1 + 2\},$ $\{m_1 + 3, m_1 + 4\}, \ldots, \{m_1 + 2m_2 - 1, m_1 + 2m_2\}$ as its blocks of size two, and so on. For example, \[X_{(4^3, 2^2, 1)} = \{ \{1\}, \{2, 3\}, \{4, 5\}, \{6, 7, 8, 9\}, \{10, 11, 12, 13\}, \{14, 15, 16, 17\}\}.\]  

Note that $\stab_{S_n}(X_\lambda) = \prod_i S_{m_i}[S_i].$ It is clear that the subgroup $\prod_i S_{m_i}[S_i]$ fixes the set partition $X_\lambda,$ and it is routine to check that this is the entire stabilizer subgroup. 

Let $W_1, W_2, W_3, \ldots$ be representations of $S_1, S_2, S_3 \ldots$ and let $V$ be a representation of $S_r.$ Assume $m_1 + m_2 + \ldots = r$ and that the restriction $V \big \downarrow^{S_r}_{\prod_{i}S_{m_i}}$ decomposes into $\prod_{i}S_{m_i}$-irreducibles as \[ \bigoplus_{(\lambda^1, \lambda ^2, \ldots)} \left(\chi^{\lambda^1} \otimes \chi^{\lambda^2} \otimes \ldots \right)^{\oplus c(\lambda^1, \lambda^2, \ldots)}.\]

Following the notation of Sundaram in \cite{SundaramAdvances}, define \emph{$V \Big \downarrow^{S_r}_{\prod_i S_{m_i}}[W_1 \otimes W_{2} \otimes \ldots]$} to be the following representation of $\prod_i S_{m_i}[S_i]:$\[\bigoplus_{(\lambda^1, \lambda^2, \ldots)} \left(\chi^{\lambda^1}\left[W_1\right] \otimes \chi^{\lambda^2}\left[W_2\right] \otimes \ldots \right)^{\oplus c(\lambda^1, \lambda^2, \cdots)}.\] 

We are now able  to state Sundaram's result on the homology of upper intervals in $\Pi_n.$

\begin{prop}[Sundaram, in proof of Theorem 1.4 of \cite{SundaramAdvances}]\label{prop:sundaram-Ln}
Let $\mu = 1^{m_1}2^{m_2}\ldots \vdash n$ be a {partition}\footnote{Technically, Sundaram assumes $\mu \neq (n)$ in her proof, but given our convention that $\Tilde{H}_{-2}(X, X)$ carries the trivial representation, the proposition holds for $\mu = (n)$ too.} with $\ell(\mu) = r.$ Recall that $\hat{1}$ denotes the maximal element $X_{(n)}$ of $\Pi_n.$ As representations of $\stab_{S_n}(X_\mu),$
\[\Tilde{H}_{r-3}\left(X_\mu, \hat{1}\right) \cong \Tilde{H}_{r - 3}\left(\Pi_r\right)\Big \downarrow^{S_r}_{\prod_{i} S_{m_i}}[\mathbb{1}_{S_1} \otimes \mathbb{1}_{S_2} \otimes \ldots].\]
\end{prop}

The determinant representation of $X_\mu$ has a simple formulation.

\begin{lemma}\label{lem:det(X_0/[n])}
    Let $\mu = 1^{m_1}2^{m_2}\ldots \vdash n$ be a partition with $\ell(\mu) = r.$ As representations of $\stab_{S_n}(X_\mu),$
\[\det(X_\mu) \cong \bigotimes_{i\geq1} \sgn_{S_{m_i}}[\mathbb{1}_{S_i}].\]

\end{lemma}
\begin{proof}
 Let $e_1, e_2, \ldots, e_n$ be standard basis vectors for $\mathbb{R}^n$. For each block $\left(X_{\mu}\right)_i$ of $X_\mu$, define $v_i$ be the sum \[v_i:= \sum_{j \in \left(X_{\mu}\right)_i}e_j.\] 
 One orientation for the subspace of $\mathbb{R}^n$ associated to $X_\mu$ is the ordered basis $v_1, v_2, \ldots, v_r.$ Since both $\det(X_\mu)$ and $\bigotimes_{i \geq 1}\sgn_{S_{m_i}}[\mathbb{1}_{S_i}]$ are one-dimensional characters, it suffices to check that they agree on a generating set for $\stab_{S_n}(X_\mu).$ The subgroup $\stab_{S_n}(X_\mu) = \prod_i S_{m_i}[S_i]$ is generated by (i) \textit{outer transpositions} of the form $(1, 1, \ldots, 1; s_i)$ in one of the copies of $S_{m_j}[S_j]$ and (ii) \textit{inner transpositions} of the form $(1, \ldots, 1, s_j, 1, \ldots, 1; 1)$ in one of the copies $S_{m_j}[S_j].$  An outer transposition swaps two basis vectors $v_i, v_{i + 1},$ negating the orientation. An inner transposition fixes each $v_j,$ mantaining the orientation. The characters of the generators on $\bigotimes_{i \geq 1}\sgn_{S_{m_i}}[\mathbb{1}_{S_i}]$ agree. 
\end{proof}

The following lemma will be useful in studying the representation $\Tilde{H}_{r - 3}\left(X_\mu, \hat{1}\right) \otimes \det(X_\mu).$

\begin{lemma} \label{lem:mult-wreath-prod-reps}
    Let $V_1, V_2$ be representations of $S_n$ and $W_1, W_2$ be representations of a finite group $G. $ Then, as $S_n[G]$ representations,\[V_1[W_1] \otimes V_2[W_2] \cong (V_1 \otimes V_2)[W_1 \otimes W_2].\] 
\end{lemma}
\begin{proof}
  It is straightforward to check that the vector space isomorphism 
    \begin{align*}
      \left((\otimes^n W_1)\otimes V_1 \right)\otimes \left((\otimes^n W_2)\otimes V_2 \right)&\longrightarrow (\otimes^n (W_1 \otimes W_2))\otimes (V_1 \otimes V_2)\\
        ((\otimes_{i}w^1_i)\otimes v^1)\otimes ((\otimes_{i} w^2_i) \otimes v^2) &\mapsto (\otimes_{i}(w_i^1 \otimes w_i^2)) \otimes (v^1 \otimes v^2).
    \end{align*}
commutes with the action of $S_n[G].$
\end{proof}

\begin{cor}\label{lem:upper-interval-stabilizer-rep}
    Let $\mu = 1^{m_1}2^{m_2}\ldots \vdash n$ be a partition with $\ell(\mu) = r.$ As $\stab_{S_n}(X_\mu)$-representations,
    \begin{align*}
        \Tilde{H}_{r - 3}\left(X_\mu, \hat{1}\right) \otimes \det(X_\mu) \cong \mathscr{L}_r \Big \downarrow^{S_r}_{\prod_{i} S_{m_i}}[\mathbb{1}_{S_1} \otimes \mathbb{1}_{S_2} \otimes \ldots].
    \end{align*}
\end{cor}
\begin{proof}
     Using \cref{prop:sundaram-Ln} and \cref{lem:det(X_0/[n])} to rewrite $\Tilde{H}_{r - 3}\left(X_\mu, \hat{1}\right)$ and $\det(X_\mu)$, we have that as $\stab_{S_n}(X_\mu)$-representations,
 \begin{align*}
     \Tilde{H}_{r - 3}\left(X_\mu, \hat{1}\right) \otimes \det(X_\mu) \cong \left(\Tilde{H}_{r - 3}(\Pi_r) \Big \downarrow^{S_r}_{\prod_{i} S_{m_i}}[\mathbb{1}_{S_1} \otimes \mathbb{1}_{S_2} \otimes \ldots ]\right) \otimes \left(\bigotimes_i \sgn_{S_{m_i}}[\mathbb{1}_{S_i}]\right).
 \end{align*}

 By expanding $\Tilde{H}_{r - 3}(\Pi_r) \Big \downarrow^{S_r}_{\prod_{i} S_{m_i}}[\mathbb{1}_{S_1} \otimes \mathbb{1}_{S_2} \otimes \ldots ]$, applying \cref{lem:mult-wreath-prod-reps}, then recondensing, the above is equivalent to
\begin{align*}
    \left(\Tilde{H}_{r - 3}(\Pi_r) \Big \downarrow^{S_r}_{\prod_{i} S_{m_i}} \otimes \bigotimes_i \sgn_{S_{m_i}}\right)[\mathbb{1}_{S_1} \otimes \mathbb{1}_{S_2} \otimes \ldots ].
\end{align*}
However, note that as $\prod_i S_{m_i}$-representations, $\bigotimes_i \sgn_{S_{m_i}} \cong \sgn_{S_r} \big\downarrow^{S_r}_{\prod_iS_{m_i}}.$ Hence, we can rewrite the above as 
 
\begin{align*}
     \left(\Tilde{H}_{r - 3}(\Pi_r) \Big \downarrow^{S_r}_{\prod_{i} S_{m_i}} \otimes \sgn_{S_r} \Big \downarrow^{S_r}_{\prod_i S_{m_i}}\right)&[\mathbb{1}_{S_1} \otimes \mathbb{1}_{S_2} \otimes \ldots ]\\
     &\cong \left(\Tilde{H}_{r - 3}(\Pi_r) \otimes \sgn_{S_r} \right)\Big \downarrow^{S_r}_{\prod_{i } S_{m_i}} [\mathbb{1}_{S_1} \otimes \mathbb{1}_{S_2} \otimes \ldots]\\
     &\cong \mathscr{L}_r \Big \downarrow^{S_r}_{\prod_{i} S_{m_i}}[\mathbb{1}_{S_1} \otimes \mathbb{1}_{S_2} \otimes \ldots].
    \end{align*}
\end{proof}

\subsubsection{Plethysm of formal power series}\label{def:plethysm-formal-power-series}
    
Proving \cref{lem:upper-intervals-plethysm} and \cref{lem:necklace-bijection} involves \emph{plethysm of formal power series,} which we explain below using the conventions of Wachs in \cite[\S 2.4]{WachsPosetTopology}.

Let $g$ be a formal power series in the variables $x_1, x_2, \ldots, z_1, z_2, \ldots $ which has \textit{nonnegative integer coefficients}. Form some sequence \emph{$\mathcal{M}(g)$} of the monomials of $g,$ where the number of times each monomial appears in the sequence $\mathcal{M}(g)$ is precisely its coefficient in $g.$  For a symmetric function $f$ in $x_1, x_2, \ldots $ the plethysm \emph{$f[g]$} is defined by replacing $x_1$ in $f$ with the first monomial of $\mathcal{M}(g),$ $x_2$ with the second monomial of $\mathcal{M}(g),$ and so on. Note that the plethysm $f[g]$ is well-defined because the order of $\mathcal{M}(g)$ does not matter as $f$ is symmetric.  We assume that all symmetric functions are written in the variables $x_1, x_2, \ldots. $ If $g$ has no $z_i$'s and is also symmetric in the $x$'s, the definition of plethysm of formal power series agrees with the ordinary definition of plethysm of $f[g].$

The following two examples play key roles in the proof of \cref{lem:necklace-bijection}.

\begin{example}\label{ex:total-ordering-hi-sum}
    Let \[\displaystyle g = \sum_{i \geq 1}z_ih_i = \sum_{\substack{\text{partitions}\\ \nu \neq \emptyset}}z_{\ell(\nu)}\mathbf{x}_{\nu}.\] The monomials in $g$ each have coefficient one.  Since there are a countable number of monomials in $g$ there exists some bijection between the monomials and $x_1, x_2, \ldots.$ To compute $f[g]$ for a symmetric function $f,$ one would then replace each $x_i$ in $f$ with its corresponding monomial of $g.$
\end{example}

\begin{example} \label{ex:total-ordering-zwhw}
 Fix a Lyndon word $w = w_1w_2 \ldots w_k.$ Let \[g = \mathbf{z}_wh_w = \prod_{i = 1}^kz_{w_i}h_{w_i} = \prod_{i= 1}^k z_{w_i} \left(\sum_{\substack{\text{partitions }\nu:\\ \ell(\nu) = w_i}}\mathbf{x}_{\nu}\right).\] Unlike the previous example, there are some monomials with coefficients higher than one in $g.$ As an example, set $w = 12;$  the monomial $z_1z_2x_{2}x_3x_4 = \mathbf{z}_{12}\mathbf{x}_{234}$ has coefficient three in $g = (z_1h_1)(z_2h_2)$, counting which of the $x_2,$ $x_3,$ or $x_4$ arise from $h_1.$ To distinguish between repeated monomials, we consider instead distinct ordered pairs \[(z_{w_1}\mathbf{x}_{\nu(1)}, z_{w_2}\mathbf{x}_{\nu(2)}, \ldots, z_{w_k}\mathbf{x}_{\nu(k)}),\] where each $\nu(i)$ is a partition with length $w_i.$ For instance, the three ordered pairs producing $\mathbf{z}_{12}\mathbf{x}_{234}$ are \[(z_1\mathbf{x}_2, z_2\mathbf{x}_{34}), 
     (z_1\mathbf{x}_{3}, z_2 \mathbf{x}_{24}), \text{ and } (z_1 \mathbf{x}_4, z_2\mathbf{x}_{23}).\]

     There are countably many such ordered pairs so there exists some bijection matching the ordered pairs to $x_1, x_2, \ldots.$ To compute $f[g]$ for a symmetric function $f,$ one would then replace each $x_i$ in $f$ with the monomial of $g$ associated to the corresponding ordered pair.
\end{example}

\subsubsection{Proof of \cref{lem:upper-intervals-plethysm}} \label{sec:reduction-of-base-case}

To manipulate the generating function $\sum_{\mu \neq 0}\mathbf{z}_\mu \cdot \ch\left(E_{|\mu|}\C\mathcal{F}_{|\mu|}E_\mu\right),$ we will need a technical but powerful lemma. The following lemma is an application of a formula of Wachs from \cite[Theorem 5.5]{WachsWhitneyHomology}\footnote{Wachs uses different notation here.}, which is a refinement of a result of Sundaram in \cite[Lemma 1.5]{SundaramAdvances}.
\begin{lemma}[Wachs, refining Sundaram]\label{lem:Wachs-plethysm-gen-fn}
For each positive integer $i$, pick two $S_i$-modules: $V_i$ and $W_i.$ Then,
\begin{align*}
    \sum_{r \geq 1}\sum_{\substack{\mu =\\ 1^{m_1} 2^{m_2} \ldots:\\ m_1 + m_2 + \ldots = r} }\mathbf{z}_\mu \cdot \ch \left(V_r \Big \downarrow_{\prod_{i \geq 1} S_{m_i}}^{S_r} [W_1 \otimes W_2 \otimes \ldots ]\Bigg \uparrow_{\prod_iS_{m_i}[S_i]}^{S_{|\mu|}}\right) = \sum_{r \geq 1} \ch\left(V_r\right)\left[\sum_{i \geq 1} z_i \cdot \ch\left(W_i\right)\right].
\end{align*}
\end{lemma}
With this lemma, we are now ready to prove \cref{lem:upper-intervals-plethysm}. 
\begin{lemma} \label{lem:upper-intervals-plethysm}
    \[\sum_{\mu \neq \emptyset} \mathbf{z}_\mu \cdot \ch \left(E_{|\mu|}\C \mathcal{F}_{|\mu|}E_\mu \right)  = \sum_{r \geq 1} L_r[z_1h_1 + z_2h_2  + \ldots].\]
\end{lemma}
\begin{proof}
Let $\mu \neq \emptyset$ be a partition of $n$ of the form $1^{m_1}2^{m_2}\ldots$ with $ \ell(\mu) = r.$ 

By \cref{cor:induced-cohomology}, \[\ch \left(E_{|\mu|}\C \mathcal{F}_nE_\mu\right) = \ch \left(\Tilde{H}_{r - 3}\left(X_\mu, \hat{1}\right) \otimes \det(X_\mu) \Big \uparrow_{\prod_i S_{m_i}[S_i]}^{S_{n}} \right).\]

By \cref{lem:upper-interval-stabilizer-rep}, 
    \begin{align*}
        \Tilde{H}_{r - 3}\left(X_\mu, \hat{1}\right) \otimes \det(X_\mu) \cong \mathscr{L}_r \Big \downarrow^{S_r}_{\prod_{i \geq 1} S_{m_i}}[\mathbb{1}_{S_1} \otimes \mathbb{1}_{S_2} \otimes \ldots].
    \end{align*}

    Therefore,
    \[\ch(E_{|\mu|}\C \mathcal{F}_{|\mu|} E_\mu) = \ch \left( \left(\mathscr{L}_r \Big \downarrow_{\prod_{i \geq 1} S_{m_i}}^{S_r}[\mathbb{1}_{S_1} \otimes \mathbb{1}_{S_2}\otimes \ldots ]\right)\Bigg \uparrow_{\prod_{i\geq 1}S_{m_i}[S_i]}^{S_n}\right).\]
    
   Substituting $V_i = \mathscr{L}_i$ and $W_i = \mathbb{1}_{S_i}$ into  \cref{lem:Wachs-plethysm-gen-fn}, we have that

   \begin{align*}
       \sum_{\mu \neq \emptyset }\mathbf{z}_\mu \cdot \ch \left(E_{|\mu|}\C \mathcal{F}_{|\mu|}E_\mu\right) &= \sum_{r\geq 1}\sum_{\substack{\mu = \\1^{m_1} 2^{m_2} \ldots:\\ m_1 + m_2 + \ldots = r}}\mathbf{z}_\mu \cdot \ch \left(E_{|\mu|}\C \mathcal{F}_{|\mu|}E_{\mu}\right)\\
       &= \sum_{r \geq 1} L_{r}[z_1 h_1 + z_2 h_2 + \ldots ].
   \end{align*}
    \end{proof}

\subsubsection{Necklaces}

Let $\mathscr{A}$ be a totally ordered alphabet. A \emph{necklace} of length $k \geq 1$ on $\mathscr{A}$ is an equivalence class of a word on $\mathscr{A}$ with $k$ letters up to cyclic rotation. We will write the necklace represented by the word $(w_1, w_2, \ldots, w_k)$ as \emph{$[w_1, w_2, \ldots, w_k].$} Define \emph{$\mathscr{N}_k(\mathscr{A})$} to be the set of necklaces on  $\mathscr{A}$ with $k$ letters, and \emph{$\mathscr{N}(\mathscr{A})$} to be the set of necklaces on $\mathscr{A}$ of any (nonzero, finite) length. 

\begin{definition}\label{lem:bijection-necklaces-to-lyndon-powers}
Let $\mathscr{A} = \{1, 2, \ldots\}$ be the standard ordered alphabet. We define a map  \[f:\mathscr{N}(\mathscr{A}) \longrightarrow \bigsqcup_{\substack{\text{Lyndon}\\ w}}\{(w, 1), (w, 2), (w, 3), \ldots\}\] as follows. Let $\eta$ be a necklace on $\mathscr{A}$ with a lexicographically minimal representative word $\alpha = (\alpha_1, \alpha_2, \ldots, \alpha_k).$ Notice that $\alpha$ is either a Lyndon word, or a power $i \geq 1$ of a single Lyndon word, so can be written as $\alpha = w^i$ for a Lyndon word $w$. The Lyndon word $w$ and power $i$ are unique since Lyndon factorization is unique.  Let $f$ be the map which sends $\eta$ to the pair $\left(w,i\right).$ 
\end{definition}

Although we won't need this fact, the map $f$ is a bijection since $(w, i) \mapsto [w^i]$ can be checked to be an inverse (using that the $w$ are Lyndon words). 

\begin{example}\label{ex:mini-necklace-bijection}
$f\left([2, 2, 1, 1, 2, 2, 1, 1, 2, 2, 1, 1]\right)  = \left(\left(1,1,2, 2\right), 3\right).$
\end{example}

A necklace is \emph{primitive} if its representative words are not powers of shorter words. For example, the necklace $[2, 2, 1, 1, 2, 2, 1, 1, 2, 2, 1, 1]$ in \cref{ex:mini-necklace-bijection} is \textit{not} primitive since it is represented by the word $(2, 2, 1, 1)^3.$  Let \emph{$\mathscr{P}\mathscr{N}_k(\mathscr{A})$} denote the set of primitive necklaces of length $k$ on the alphabet $\mathscr{A}$  and let \emph{$\mathscr{P}\mathscr{N}(\mathscr{A})$ }denote the set of primitive necklaces on $\mathscr{A}$ of any (nonempty, finite) length.

Recall the notion of a sequence $\mathcal{M}(g)$ associated to a formal power series $g$ from \cref{def:plethysm-formal-power-series}. If $g$ has a countable number of monomials, we can theoretically view $\mathcal{M}(g)$ as a totally ordered alphabet (although we won't need to refer to the order) and form words and necklaces on it. Given a word $w= \left(\mathbf{z}_{\nu(1)}\mathbf{x}_{\mu(1)}, \mathbf{z}_{\nu(2)}\mathbf{x}_{\mu(2)},\ldots, \mathbf{z}_{\nu(k)}\mathbf{x}_{\mu(k)}\right)$ on $\mathcal{M}(g)$ (where each $\nu(i), \mu(i)$ are partitions), define its \emph{evaluation} \emph{$\mathrm{eval}(w)$} to be the product \[\mathrm{eval}(w) := \prod_{i=1}^k\mathbf{z}_{\nu(i)}\mathbf{x}_{\mu(i)}.\]

 Since evaluation does not depend on the order of the letters in $w,$ we can define it on necklaces in $\mathscr{N}(\mathcal{M}(g)),$ by $\mathrm{eval}([w]) = \mathrm{eval}(w).$ \begin{example} \label{ex:eval-example}
     Recall the examples in \cref{def:plethysm-formal-power-series}. Let $\eta$ and $\tau$ be the following necklaces in $\mathscr{PN}\left(\mathcal{M}\left( z_1h_1 + z_2h_2 + \ldots \right)\right)$ and $\mathscr{PN}\left(\mathcal{M}\left(\mathbf{z}_{223}h_{223}\right)\right),$ respectively. Note that we are using our ordered pair notation for the necklaces on $\mathcal{M}(\mathbf{z}_{223}h_{223}).$ 
     \begin{align*}  \eta &= \left[z_3\mathbf{x}_{244}, z_2\mathbf{x}_{44}, z_2\mathbf{x}_{23}, z_3\mathbf{x}_{266}, z_2\mathbf{x}_{44}, z_2\mathbf{x}_{23}, z_3\mathbf{x}_{246}, z_2\mathbf{x}_{44},  z_2\mathbf{x}_{23} \right],\\
    \tau &= \left[\left(z_2\mathbf{x}_{44}, z_2\mathbf{x}_{23}, z_3\mathbf{x}_{266}\right), \left(z_2\mathbf{x}_{44}, z_2\mathbf{x}_{23}, z_3\mathbf{x}_{246}\right), \left(z_2\mathbf{x}_{44},  z_2\mathbf{x}_{23}, z_3\mathbf{x}_{244}\right)\right].
\end{align*}
Then, \[\mathrm{eval}(\eta) = \mathrm{eval}(\tau) = \mathbf{z}_{2^6 3^3}\mathbf{x}_{2^6 3^3 4^9 6^3} = z_2^6z_3^3 x_2^6x_3^3 x_4^9x_6^3.\]
 \end{example}
    
\subsubsection{Necklace bijection
}\label{sec:necklace-bij}
We define an evaluation-preserving map \[\Psi: \mathscr{PN}\left(\mathcal{M}\left(z_1 h_1 + z_2h_2 + \ldots\right)\right) \longrightarrow \bigsqcup_{\substack{\text{Lyndon}\\ w}}\mathscr{PN}\left(\mathcal{M}\left(\mathbf{z}_wh_w\right)\right).\] 

Consider the necklace\[\eta = \left[z_{\ell\left(\nu(1)\right)}\mathbf{x}_{\nu(1)}, z_{\ell\left(\nu(2)\right)}\mathbf{x}_{\nu(2)}, \ldots, z_{\ell \left(\nu(k)\right)}\mathbf{x}_{\nu(k)}\right]\] in $\mathscr{PN}\left(\mathcal{M}\left(z_1h_1 + z_2 h_2 + \ldots \right)\right).$ Define an associated necklace on $\{1, 2, \ldots\}$ by taking the subscripts of $z$ in $\eta.$ Specifically, define \[z(\eta) := \left[\ell \left(\nu(1)\right), \ell\left(\nu(2)\right), \ldots, \ell\left(\nu(k)\right)\right].\] 
Recall the map $f$ in \cref{lem:bijection-necklaces-to-lyndon-powers}; we have $f([z(\eta)]) = \left(w, k / i\right)$ for some Lyndon word $w$ and integer $i,$ with $[z(\eta)] = [w^{k / i}].$ Because $w$ is a Lyndon word, if we write $\overline{m}$ to mean $m$ modulo $k,$ there are \textit{precisely} $k/i$ values of $0 \leq j \leq k - 1$ such that
\begin{align}\label{eqn:z(eta)}
    w = (w_1,w_2, \ldots, w_i) =  \left(\ell \left(\nu(\overline{j + 1})\right), \ell\left(\nu(\overline{j + 2})\right), \ldots, \ell\left(\nu(\overline{j + i})\right)\right).
\end{align}

In words, the map $\Psi$ will (i) pick one of the $k/i$ representative words of  $\eta$ whose prefix has $w_1, w_2, \ldots, w_i$ as the subscripts of $z$ then (ii) group together the first $i$ letters, the next $i$ letters, and so on. The result is a word on $\mathcal{M}(\mathbf{z}_wh_w)$ and $\Psi$ sends $\eta$ to the necklace represented by that word. More formally,  $\Psi(\eta)$ is the following necklace on $\mathcal{M}(\mathbf{z}_wh_w)$ with $k / i$ beads:

\begin{align*}
\Psi(\eta) := \Bigg[&{\left(z_{w_1}\mathbf{x}_{\nu(\overline{j + 1})},z_{w_2}\mathbf{x}_{\nu(\overline{j + 2})}, \ldots, z_{w_i}\mathbf{x}_{\nu(\overline{j + i})} \right)},\\
&{\left(z_{w_1}\mathbf{x}_{\nu(\overline{j +i + 1})},z_{w_2}\mathbf{x}_{\nu(\overline{j +i + 2})}, \ldots, z_{w_i}\mathbf{x}_{\nu(\overline{j + 2i})} \right)},\\
& \ldots,\\
&{\left(z_{w_1}\mathbf{x}_{\nu(\overline{j + k - i + 1})},z_{w_2}\mathbf{x}_{\nu(\overline{j + k - i + 2})}, \ldots, z_{w_i}\mathbf{x}_{\nu(\overline{j + k})} \right)} \Bigg].  
\end{align*}

The map $\Psi$ is well-defined because the resulting necklace does not depend on the choice of $j$ from \cref{eqn:z(eta)}. The necklace $\Psi(\eta)$ is primitive because $\eta$ is, and $\Psi$ is evaluation preserving by construction.

\begin{example}
Let $\eta$ be the following primitive necklace in $\mathscr{PN}\left(\mathcal{M}(z_1h_1 + z_2h_2 + \ldots)\right):$
\[\eta = \left[z_3\mathbf{x}_{244}, \color{blue}{z_2\mathbf{x}_{44}}\normalcolor, z_2\mathbf{x}_{23}, z_3\mathbf{x}_{266}, \color{blue}{z_2\mathbf{x}_{44}}\normalcolor, z_2\mathbf{x}_{23}, z_3\mathbf{x}_{246}, \color{blue}{z_2\mathbf{x}_{44}}\normalcolor,  z_2\mathbf{x}_{23} \right].\]

The necklace $z(\eta)$ is $[3, 2, 2, 3, 2, 2, 3, 2, 2],$ which $f$ maps to the ordered pair $\left((2, 2, 3), 3\right).$ The possible choices for of beads which we can cyclically rotate to the front before grouping are colored blue. $\Psi(\eta)$ is the following primitive necklace on $\mathscr{PN}\left(\mathcal{M}\left(\mathbf{z}_{223}h_{223}\right)\right):$ 
\begin{align*}
    \left[\left(z_2\mathbf{x}_{44}, z_2\mathbf{x}_{23}, z_3\mathbf{x}_{266}\right), \left(z_2\mathbf{x}_{44}, z_2\mathbf{x}_{23}, z_3\mathbf{x}_{246}\right), \left(z_2\mathbf{x}_{44},  z_2\mathbf{x}_{23}, z_3\mathbf{x}_{244}\right)\right].
\end{align*}
As explained in \cref{ex:eval-example}, $\eta$ and $\Psi(\eta)$ have the same evaluation.
\end{example}

\begin{prop}\label{prop:eval-preserve-bij}
    The map $\Psi$ is a bijection.
\end{prop}
\begin{proof}
We prove that $\Psi$ is a bijection by providing an inverse map \[\Theta:  \bigsqcup_{\substack{\text{Lyndon}\\ w}}\mathscr{PN}\left(\mathcal{M}\left(\mathbf{z}_wh_w\right)\right)\longrightarrow \mathscr{PN}\left(\mathcal{M}\left(z_1 h_1 + z_2h_2 + \ldots\right)\right).\]
 Let $w = (w_1, w_2, \ldots, w_\ell)$ be a Lyndon word and let $\tau$ be the following primitive necklace in $\mathscr{PN}(\mathcal{M}(\mathbf{z}_w h_w)):$ 
\begin{align*}
\tau =\left[{\left(z_{w_1}\mathbf{x}_{\nu(1,1)},z_{w_2}\mathbf{x}_{\nu(1,2)}, \ldots, z_{w_\ell}\mathbf{x}_{\nu(1,\ell)} \right)},\ldots,{\left(z_{w_{1}}\mathbf{x}_{\nu(m,1)},z_{w_{2}}\mathbf{x}_{\nu(m,2)}, \ldots, z_{w_{\ell}}\mathbf{x}_{\nu(m,\ell)} \right)} \right].
\end{align*}

 Define $\Theta$ to be the map which ``removes parentheses,'' sending $\tau$ to the following necklace on $\mathcal{M}(z_1h_1 + z_2h_2 + \ldots),$
\[\left[z_{w_1}\mathbf{x}_{\nu(1, 1)}, z_{w_2}\mathbf{x}_{\nu(1, 2)}, \ldots, z_{w_\ell}\mathbf{x}_{\nu(1, \ell)}, z_{w_1}\mathbf{x}_{\nu(2,1)}, \ldots, z_{w_\ell }\mathbf{x}_{\nu(2,\ell)}, \ldots,  z_{w_\ell}\mathbf{x}_{\nu(m, \ell)}\right].\] We claim $\Theta(\tau)$ is a primitive necklace. If there were a prefix $z_{w_1}\mathbf{x}_{\nu(1, 1)} \ldots z_{w_i}\mathbf{x}_{\nu(j, i)}$ with $i < \ell$ and power $k$ for which \[\left(z_{w_1}\mathbf{x}_{\nu(1, 1)} \ldots z_{w_i}\mathbf{x}_{\nu(j, i)}\right)^k = z_{w_1}\mathbf{x}_{\nu(1, 1)}\ldots z_{w_\ell}\mathbf{x}_{\nu(m, \ell)},\] then by restricting our attention to the coefficients of $z,$ we'd see $(w^{j - 1}w_1\cdots w_i)^k = w^{m},$ which implies that $w_1w_2 \cdots w_{\ell - i} = w_{i + 1}\cdots w_\ell .$ This is not possible since $w$ is Lyndon, and an equivalent characterization of $w$ being Lyndon is that it is strictly lexicographically smaller than all of its proper suffixes (see \cite[Proposition 5.12]{Lothaire_1997}). There are also no factorizations of the form
\[\left(z_{w_1}\mathbf{x}_{\nu(1, 1)} \ldots z_{w_\ell}\mathbf{x}_{\nu(j, \ell)}\right)^k = z_{w_1}\mathbf{x}_{\nu(1, 1)}\ldots z_{w_\ell}\mathbf{x}_{\nu(m, \ell)}\] since $\tau$ is primitive. Hence, $\Theta(\tau)$ is a {primitive} necklace of $\mathcal{M}(z_1h_1 + z_2h_2 + \cdots).$ It is straightforward to check that $\Theta \circ \Psi$ is the identity mapping. The fact that $w$ is Lyndon implies $z \left(\Theta(\tau)\right)$ is necessarily $w^m,$ hence  $ \Psi \circ \Theta$ is the identity mapping as well, completing the proof.
\end{proof}

The symmetric function $L_n,$ and in fact all the higher Lie symmetric functions $L_\lambda,$ have combinatorial interpretations involving primitive necklaces (see {Gessel--Reutenauer's}\footnote{Gessel--Reutenauer used this necklace interpretation to prove their remarkable theorem mentioned in the proof of \cref{prop:sign-iso}.} work:  \cite[Equation 2.1]{gesselreutenauer}). The next lemma is an application of this interpretation.

\begin{lemma}\label{lem:Lie-plethysm-in-terms-of-necklaces}
 Let $g$ be a formal power series in $x_1, x_2, \ldots, z_1, z_2, \ldots $ which has nonnegative integer coefficients. Then,
\begin{align*}
        \sum_{n \geq 1}L_n[g] = \sum_{\eta \in \mathscr{PN}(\mathcal{M}(g))} \mathrm{eval}(\eta).
    \end{align*}
\end{lemma}
\begin{proof}
 In \cite[Equation 2.1]{gesselreutenauer} (by setting $\lambda = n$), Gessel--Reutenauer explain that \[L_n = \sum_{\tau \in \mathscr{PN}_n(x_1, x_2, \ldots)}\mathrm{eval}(\tau).\] The lemma then follows from the definition of plethysm of formal power series (\cref{def:plethysm-formal-power-series}). 
\end{proof}

We are finally ready to prove \cref{lem:necklace-bijection}.
\begin{lemma} \label{lem:necklace-bijection}
    \[ \sum_{r \geq 1} L_r[z_1h_1 + z_2h_2  + \ldots] = \sum_{\substack{\text{Lyndon}\\w}} \ \sum_{m \geq 1}\mathbf{z}_{w^m}L_m[h_w].\]
\end{lemma}

\begin{proof} 
By \cref{prop:eval-preserve-bij}, there is an evaluation-preserving bijection between \begin{align*}
\bigsqcup_{\substack{\text{Lyndon}\\ w}}\mathscr{PN}\left(\mathcal{M}\left(\mathbf{z}_wh_w\right)\right) \text{ and } \mathscr{PN}\left(\mathcal{M}\left(z_1 h_1 + z_2h_2 + \ldots\right)\right)  .
\end{align*}
Hence, we have the equality
\begin{align*}
   \sum_{\substack{\text{Lyndon}\\ w}}\sum_{\substack{\eta \in\\ \mathscr{PN}(\mathcal{M}(\mathbf{z}_wh_w))}}\mathrm{eval}(\eta) = \sum_{\substack{\eta \in\\ \mathscr{PN}(\mathcal{M}(z_1h_1 + \ldots))}}\mathrm{eval}(\eta).
\end{align*}

Thus, by \cref{lem:Lie-plethysm-in-terms-of-necklaces},
\begin{align*}
\sum_{\substack{\text{Lyndon} \\ w}}\sum_{m \geq 1}L_m[\mathbf{z}_wh_w]
=  \sum_{r \geq 1}L_r[z_1h_1 + z_2h_2 + \ldots].
\end{align*}
The lemma then follows from the observation that 
\begin{align*}
L_m[\mathbf{z}_wh_w]
= \mathbf{z}_{w^{m}}L_m[h_w]. 
\end{align*}
\end{proof}

    \subsection{General intervals}\label{sec:general-case}
        In this section, we complete the proofs from step 3 of the proof outline in \cref{sec:outline-proof}. In \cref{subsec:stabilizers}, we combinatorially interpret the $S_n$-orbits of $X \leq Y$ and identify their stabilizer subgroups. In \cref{subsec:homology-intervals-lattice}, we compute the representations on the homology of intervals of $\Pi_n$ and their associated determinant twists, allowing us to understand the representations $E_Y\C \mathcal{F}_n E_X$ in terms of step 2. 
\subsubsection{Identifying intersections of stabilizers} \label{subsec:stabilizers}
\begin{definition}
In a \textbf{multiset partition}, repeated elements are viewed as indistinguishable. For example, there are precisely four multiset partitions of the multiset $\{1, 1, 2\}:$ $\{112\}, \{1, 12\}, \{11, 2\}, \{1, 1, 2\}.$ Given partitions $\lambda, \mu$ of the same size we define a set \emph{$\lambda(\mu)$} of multiset partitions.  In particular, if $\mu$ does not refine $\lambda,$ set $\lambda(\mu)$ to be the empty set. Otherwise, define $\lambda(\mu)$ to be the set of multiset partitions of the multiset $\{\mu_i: 1 \leq i \leq \ell(\mu)\}$ which have block sizes summing to the parts of $\lambda$.
\end{definition}
\begin{example}
\begin{itemize}
    \item Let $\mu = (5, 3, 2, 2, 1,1)$ and $\lambda = (6, 5, 3).$ Then, \[\lambda(\mu) = \left\{\{ 2211 , 5 , 3\},\{ 51, 32, 21\}, \{ 51 , 221 , 3\}, \{ 321 , 5 , 21\} \right \}. \]
    \item Let $\mu = (2, 2, 1, 1, 1, 1)$ and $\lambda = (3,3,2)$. Then, 
    \[\lambda(\mu) = \left\{ \{21, 21, 11\}, \{21 , 111 , 2\}\right\}.\]
\end{itemize}
\end{example}

\begin{lemma} \label{lem:bij-w-lambda(mu)}
There is a bijection 
\begin{align*}
    \left\{S_n\text{-orbits }[X \leq Y]: X \in \mu, Y \in \lambda\right\}&\longleftrightarrow \lambda(\mu).
\end{align*}
\end{lemma}
\begin{proof}
Consider the map which sends $[X \leq Y]$ to the multiset partition in $\lambda(\mu)$ which has one block $\nu$ for each block $Y_i$ of $Y,$ where $\nu$ is comprised of the sizes $\nu_1, \nu_2, \ldots, \nu_{\ell(\nu)}$ of the blocks in $X_{j_1}, \ldots, X_{j_{\ell(\nu)}}$ which combine to form $Y_i.$ To see that this map is well-defined, observe that for $\pi \in S_n,$ the blocks of $\pi(X)$ which form $\pi(Y_i)$ are $\pi\left(X_{j_1}\right), \ldots, \pi \left(X_{j_{\ell(\nu)}}\right),$ whose sizes are preserved.

Now, we prove the map is a bijection. First, note that each orbit has a representative of the form $X_\mu \leq Y,$ since $S_n$ acts transitively on $X$ with block sizes $\mu.$  Fixing a multiset partition $S$ in $\lambda(\mu)$, one can always build a set partition $Y$ with block sizes $\lambda$ for which $[X_\mu \leq Y]$ maps to $S$ by combining blocks of the appropriate sizes in $X_\mu$ to form $Y$. Hence, the map is surjective. To see injectivity, assume $[X_\mu \leq Y]$ and $[X_\mu \leq Y']$ map to the same multiset partition in $\lambda(\mu)$. By construction of the map, one can then make wholesale permutations between\textit{ blocks of the same sizes} in $X_\mu$ to get from $Y$ to $Y'$. A permutation formed from such swaps is necessarily in $\stab_{S_n}(X_\mu)$, so $[X_\mu \leq Y] = [X_\mu \leq Y'],$ as desired.
\end{proof}

It will be useful for us to specify a canonical representative of each orbit $[X \leq Y];$ equivalently, given an $S \in \lambda(\mu),$ we specify a canonical $X_S \leq Y_S$ for which $[X_S \leq Y_S]$ maps to $S.$ 
\begin{definition}\label{def:XS-YS}(Definition of $X_S, Y_S$)
Fix a multiset partition $S \in \lambda(\mu).$ Define \emph{$Y_S: = X_\lambda.$} Place an ordering on the blocks of $Y = Y_1, Y_2, \ldots, Y_{\ell(\lambda)}$ so that the minimum elements of the blocks increase. We place an ordering also on the blocks $\nu$ of $S = \{\nu^{m_\nu}\}$: first by size of $|\nu|,$ then breaking ties by increasing in lexicographical order (when the parts of $\nu$ are ordered to weakly decrease). In this way, each block in $Y$ is labelled by a partition $\nu \in S.$ To form the set partition $X_S$ from $Y_S,$ partition each block $B$ of $Y_S$ as follows. If $B'$s associated partition is $\nu = 1^{m_1}2^{m_2}\ldots$ then let the $m_1$ smallest numbers of $B$ be blocks of size one, then let the next $2m_2$ smallest numbers of $B$ be blocks of size $2$ (grouping together the smallest two, then the next smallest two, etc) and so on.
\end{definition}

\begin{example}\label{Ex:XS.YS}
    As an example, consider $S = \left\{3, 3, 3, 211, 211, 22, 22\right\} \in \lambda(\mu),$ for  $\lambda = 4^4 3^3$ and $\mu = 3^3 2^6 1^4.$ The partitions $\nu \in S$ are already written in the order described in \cref{def:XS-YS}. $Y_S$ (with its associated labels by the $\nu \in S$) is
   \begin{align*}Y_S = \left\{\underbrace{\{1, 2, 3\}}_{\nu = 3}, \underbrace{\{4, 5, 6\}}_{\nu = 3}, \underbrace{\{7, 8, 9\}}_{\nu = 3}, \underbrace{\{10, 11, 12, 13\}}_{\nu = 211}, \underbrace{\{14, 15, 16, 17\}}_{\nu = 211}, \underbrace{\{18, 19, 20, 21\}}_{\nu = 22}, \underbrace{\{22, 23, 24, 25\}}_{\nu = 22}\right\}.
   \end{align*} 
   Using the algorithm from \cref{def:XS-YS}, we obtain $X_S$ as
   \begin{align*}
        X_S =& \Big\{\{1, 2, 3\}, \{4, 5, 6\}, \{7, 8, 9\}, \{10\}, \{11\}, \{12, 13\}, \{14\}, \{15\}, \{16, 17\}, \{18, 19\},\\& \{20, 21\}, \{22, 23\}, \{24, 25\} \Big\}.
   \end{align*}
\end{example}
For a fixed multiset partition $S = \{\nu^{m_{\nu}}: \nu \in S\} \in \lambda(\mu)$ with $|\lambda| = n,$ consider the subgroup of $S_n$ generated by following two types of permutations:
\begin{itemize}
    \item \textbf{Outer:} For each $\nu \in S:$ the transpositions which \textit{wholesale swap} two blocks of $Y_S$ labelled by the same partition $\nu$, and leave the remaining integers alone. (I.e. if $Y_i, Y_j$ are two blocks labelled by $\nu,$ the permutation which swaps the smallest numbers of $Y_i,Y_j,$ then the second smallest numbers, and so on, but fixes all $k \notin Y_i \cup Y_j$.)
    \item \textbf{Inner:} For each block $B$ of $Y_S,$ any permutation which \textit{maintains} the set partition of $B$ given by its sublocks in $X_S$ (and fixes all $i \notin B$). 
\end{itemize}
The union of these types of permutations generate a subgroup of $S_n$ isomorphic to \[\prod_{\nu \in S} S_{m_\nu}\left[\stab_{S_{|\nu|}}\left(X_\nu\right)\right].\] In fact, this shall be our definition of how to view $\prod_{\nu \in S} S_{m_\nu}\left[\stab_{S_{|\nu|}}\left(X_\nu\right)\right]$ as a subgroup of $S_n.$ Note that it embeds naturally into $\prod_{\nu}S_{m_\nu}\left[S_{|\nu|}\right],$ which in turn embeds naturally into $S_{m_1, 2 m_2, 3 m_3, \ldots}$ (writing $\lambda \vdash n$ as $1^{m_1}2^{m_2}\ldots$). 

\begin{example}\label{ex:stab-subgroup} We point out how to view $S_{3}[S_3]\times S_2\left[S_2[S_1] \times S_1[S_2]\right] \times S_2\left[S_2[S_2]\right]$ as a subgroup of $S_{25}$ pictorially in \cref{fig:stab-int}. The nineteen generators are indicated by the 19 arrows. Given the similarities to \cref{ex:wr-prod-1} and \cref{ex:wreath-prod}, we do not write them all out here. However, we will explain that the generators corresponding to the two teal arrows are the permutations \color{teal}{$(10, 14)(11, 15)(12,16)(13,17)$} \normalcolor and \color{teal}{$(18, 22)(19, 23)(20, 24)(21,25).$} \color{black}{} \normalcolor

 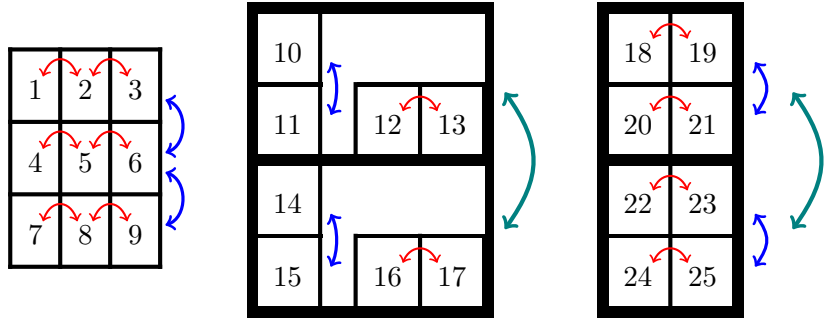
\begin{figure}[h]
 \normalcolor

    \centering
\def\thickhline{\noalign{\hrule height 4pt}}

\begin{center}
    \renewcommand{\arraystretch}{2}
\begin{NiceTabular}{ccc}[hvlines, create-medium-nodes]
   $1$ & $2$   & $3$  \\
  $4$ &  $5$ & $6$\\   
  $7$ &  $8$ & $9$
  \CodeAfter
  \begin{tikzpicture} 
        \draw[red,  thick, <->] ([yshift=0.5mm]2-2.north west) .. controls +(-.1,.3) and +(.1,.3) ..  ([yshift=0.5mm]2-1.north east) ; 
        \draw[red,  thick, <->] ([yshift=0.5mm]2-3.north west) .. controls +(-.1,.3) and +(.1,.3) ..  ([yshift=0.5mm]2-2.north east) ; 
        \draw[red,  thick, <->] ([yshift=0.5mm]1-2.north west) .. controls +(-.1,.3) and +(.1,.3) ..  ([yshift=0.5mm]1-1.north east) ; 
        \draw[red,  thick, <->] ([yshift=0.5mm]1-3.north west) .. controls +(-.1,.3) and +(.1,.3) ..  ([yshift=0.5mm]1-2.north east) ; 
        \draw[red,  thick, <->] ([yshift=0.5mm]3-2.north west) .. controls +(-.1,.3) and +(.1,.3) ..  ([yshift=0.5mm]3-1.north east) ; 
        \draw[red,  thick, <->] ([yshift=0.5mm]3-3.north west) .. controls +(-.1,.3) and +(.1,.3) ..  ([yshift=0.5mm]3-2.north east) ; 
    \draw[blue, very thick, <->] ([xshift=.3cm]1-3.south east) .. controls +(.3,-.1) and +(.3, .1) ..  ([xshift=.3cm]2-3.north east) ; 
        \draw[blue, very thick, <->] ([xshift=.3cm]2-3.south east) .. controls +(.3,-.1) and +(.3, .1) ..  ([xshift=.3cm]3-3.north east) ; 
  \end{tikzpicture}
\end{NiceTabular}\hspace{1cm}\begin{NiceTabular}{!{\vrule width 4pt}c|c|c|c!{\vrule width 4pt}}[create-medium-nodes]
\thickhline
 $10$
  &   \multicolumn{3}{c!{\vrule width 3pt}}{} \\
 \cline{1} \cline{3-4}
 $11$ & & $12$ & $13$\\ \thickhline
$14$  &  \multicolumn{3}{c!{\vrule width 3pt}}{} \\
\cline{1} \cline{3-4}
 $15$ & & $16$ & $17$\\
 \thickhline
 \CodeAfter
   \begin{tikzpicture} 
        \draw[red, thick, <->] ([yshift=0.5mm]2-3.north east) .. controls +(.1,.25) and +(-.1,.25) ..  ([yshift=0.5mm]2-4.north west) ; 
        \draw[red, thick, <->] ([yshift=0.5mm]4-3.north east) .. controls +(.1,.25) and +(-.1,.25) ..  ([yshift=0.5mm]4-4.north west) ; 
        \draw[blue, very thick, <->] ([xshift=.37cm]1-1.south east) .. controls +(.1, -.2) and +(.1,.2) ..  ([xshift=0.37cm]2-1.north east) ; 
        \draw[blue, very thick, <->] ([xshift=0.37cm]3-1.south east) .. controls +(.1, -.2) and +(.1,.2) ..  ([xshift=0.37cm]4-1.north east) ; 
        \draw[teal, ultra thick, <->] ([xshift=.5cm, yshift=3mm]2-4.north east) .. controls +(.5, -.7) and +(.5,.7) ..  ([xshift=.5cm, yshift=5mm]4-4.north east) ; 
    
  \end{tikzpicture}
\end{NiceTabular} 
\hspace{1cm} \begin{NiceTabular}{!{\vrule width 4pt}c|c!{\vrule width 4pt}}[create-medium-nodes]
\thickhline
 $18$ & $19$  \\ \hline
 $20$ & $21$  \\ \thickhline
$22$  & $23$  \\ \hline
 $24$ & $25$  \\
 \thickhline
 \CodeAfter
   \begin{tikzpicture} 
        \draw[blue, very thick, <->] ([xshift=5mm]1-2.south east) .. controls +(.2, -.3) and +(.2,.3) ..  ([xshift=5mm]2-2.north east) ; 
        
        \draw[blue, very thick, <->] ([xshift=5mm]3-2.south east) .. controls +(.2, -.3) and +(.2,.3) ..  ([xshift=5mm]4-2.north east) ; 
                \draw[teal, ultra thick, <->] ([xshift=1cm, yshift=3mm]2-2.north east) .. controls +(.5, -.7) and +(.5,.7) ..  ([xshift=1cm, yshift=5mm]4-2.north east) ;
                  \draw[red, thick, <->] ([yshift=0.5mm]4-1.north east) .. controls +(.1,.25) and +(-.1,.25) ..  ([yshift=0.5mm]4-2.north west) ; 
                   \draw[red, thick, <->] ([yshift=0.5mm]1-1.north east) .. controls +(.1,.25) and +(-.1,.25) ..  ([yshift=0.5mm]1-2.north west) ; 
                    \draw[red, thick, <->] ([yshift=0.5mm]2-1.north east) .. controls +(.1,.25) and +(-.1,.25) ..  ([yshift=0.5mm]2-2.north west) ; 
                     \draw[red, thick, <->] ([yshift=0.5mm]3-1.north east) .. controls +(.1,.25) and +(-.1,.25) ..  ([yshift=0.5mm]3-2.north west) ; 
  \end{tikzpicture}
\end{NiceTabular} 
\end{center}

        \caption{$S_{3}[S_3]\times S_2\left[S_2[S_1] \times S_1[S_2]\right] \times S_2\left[S_2[S_2]\right]$ as a subgroup of $S_{25}.$}
        \label{fig:stab-int}
    \end{figure}
\end{example}

\begin{lemma} \label{lem:normalizer-intersections}

Let $S = \{\nu^{m_\nu}: \text{partition }\nu\}$ be a multiset partition in $\lambda(\mu).$ Then, 
\[\stab_{S_n}(X_S) \cap \stab_{S_n}(Y_S) = \prod_{\nu \in S} S_{m_\nu}\left [\stab_{S_{|\nu|}}(X_\nu)\right].\] 
\end{lemma}

\begin{proof}
By construction, every generator (and so every element) of $\prod_{\nu \in S} S_{m_\nu}\left[\stab_{S_{|\nu|}}\left(X_\nu\right)\right]$ fixes both $Y_S$ and $X_S.$ Hence, $\prod_{\nu \in S} S_{m_\nu} \left[\stab_{S_{|\nu|}}\left(X_\nu\right)\right]$ is a subgroup of $\stab_{S_n}(X_S) \cap \stab_{S_n}(Y_S).$  To prove equality of the subgroups, we shall show their sizes are the same. 

Observe that $\stab_{S_n}\left(X_S \leq Y_S\right) = \stab_{S_n}(X_S) \cap \stab_{S_n}\left(Y_S\right).$ By the orbit--stabilizer theorem, it suffices to prove
\begin{align}\label{eqn:stabilizer}
    \left \lvert [X_S \leq Y_S]\right  \rvert = \frac{n!}{\left \lvert \prod_{\nu \in S}S_{m_\nu}\left[\stab_{S_{|\nu|}}\left(X_\nu\right)\right]\right \rvert}.
\end{align}
Given a partition $\nu = 1^{n_1} 2^{n_2} \ldots, $ define a constant $N_\nu = \prod_{i}n_i! (i!)^{n_i}.$
Observe that $\left \lvert \stab_{S_{|\nu|}}\left(X_\nu\right)\right \rvert = N_\nu,$ so that the right side of \cref{eqn:stabilizer} can be rewritten as
\[\frac{n!}{\left \lvert \prod_{\nu \in S}S_{m_\nu}\left[\stab_{S_{|\nu|}}\left(X_\nu\right)\right]\right \rvert} = \frac{n!}{\prod_{\nu \in S}m_\nu! \left(N_\nu\right)^{m_\nu}}.\]

The left side of \cref{eqn:stabilizer} counts the number of distinct pairs $X \leq Y \in \Pi_n$ so that $[X \leq Y]$ maps to $S.$ The number of ways to pick a set partition $Y \in \lambda$ is $\frac{n!}{N_\lambda}.$ To pick a set partition $X$ refining $Y$ so that $[X \leq Y]$ maps to $S,$ we must (i) pick which $m_\nu$ blocks of $Y$ of size $i$ get assigned to be partitioned into blocks of sizes for each $\nu \in S$ with $\nu \vdash i$ then (ii) count the number of ways to break each block of type $\nu$ up. 

Let $m_i(\lambda)$ denote the number of parts of size $i$ within the partition $\lambda.$ The number of ways to assign the blocks of size $i$ the partitions $\nu \vdash i$ is $\frac{m_i(\lambda)!}{\prod_{\nu \vdash i}m_\nu!}.$ The number of ways to break up a block of size $i$ into blocks with sizes the parts of $\nu$ is $\frac{i!}{N_\nu}.$ Putting it all together, the left side of \cref{eqn:stabilizer} is
\begin{align*}
\frac{n!}{N_\lambda} \prod_{i} \left(\frac{m_i(\lambda)!}{\prod_{\nu \vdash i} m_\nu!} \prod_{\nu \vdash i}\left(\frac{i!}{N_\nu}\right)^{m_\nu}\right) &=n! \prod_{i} \left(\frac{1}{\prod_{\nu \vdash i} m_\nu!} \prod_{\nu \vdash i}\left(\frac{1}{N_\nu}\right)^{m_\nu}\right)\\
&= \frac{n!}{\prod_\nu m_\nu! \left (N_\nu\right)^{m_\nu}},
\end{align*}
as desired.
\end{proof}

\subsubsection{Twisted homology of intervals in the partition lattice}\label{subsec:homology-intervals-lattice}

\begin{lemma} \label{lem:det-rep}
Let $S = \{\nu^{m_\nu}: \text{partition }\nu\}$ be a multiset partition in $\lambda(\mu).$ As representations of $ \prod_{\nu \in S}S_{m_\nu}[\stab_{S_{|\nu|}}(X_\nu)]$, 
\begin{align*}
    \det\left(X_{S}\right) \otimes \det\left(Y_S\right) \cong \bigotimes_{\nu \in S}\rho(\nu) \left[\det(X_\nu)\right],
\end{align*}
where $\rho(\nu) = \mathbb{1}_{S_{m_\nu}}$ if $\ell(\nu)$ is odd and $\rho(\nu) = \sgn_{S_{m_\nu}}$ if $\ell(\nu)$ is even.
\end{lemma}
\begin{proof}
First, we claim that as $\prod_{\nu}S_{m_\nu}\left[\stab_{S_{|\nu|}}\left(X_\nu\right)\right]$-representations, 
\begin{align}\label{eqn:det-rep}
    \det(X_S) \cong \bigotimes_{\nu \in S}\left(\rho(\nu) \otimes \sgn_{S_{m_\nu}}\right) \left[\det\left(X_\nu\right)\right],
\end{align}
Since both are one-dimensional characters, it suffices to check they agree on a generating set of $\prod_{\nu}S_{m_\nu}\left[\stab_{S_{|\nu|}}\left(X_\nu\right)\right].$ Order the blocks $B_1, B_2, \ldots, B_{k}$ of $X_S$ so that the minimal element of each block increases. An ordered basis for the space associated to $X_{S}$ is $v_1, v_2, \ldots, v_k,$
where $v_i$ is the sum of the standard vectors $e_j$ for the integers $j$ appearing in the block $B_i$ of $X_S.$ 

Recall the notion of outer versus inner generators from the discussion preceding \cref{ex:stab-subgroup}. An \textit{inner} generator $\pi$ associated to a block of $Y_S$ labelled by $\nu$ formed from the blocks  $B_{i + 1}, \ldots, B_{i + \ell(\nu)}$ of $X$ acts block diagonally on $v_1, v_2, \ldots, v_k$. In particular, it acts trivially on every basis vector except for the block $v_{i + 1}, \ldots, v_{i + \ell(\nu)},$ which it permutes amongst itself. The submatrix corresponding to $v_{i + 1}, \ldots, v_{i + \ell(\nu)}$ has determinant $\det(X_\nu)(\pi).$ An \textit{outer} transposition associated to $\nu$ wholesale swaps $\ell(\nu)$ basis vectors with another $\ell(\nu)$ basis vectors -- this is represented by a matrix with determinant $+1$ if $\ell(\nu)$ is even and $-1$ if $\ell(\nu)$ is odd. Hence, \cref{eqn:det-rep} holds.

We now show that as $\prod_{\nu}S_{m_\nu}\left[\stab_{S_{|\nu|}}\left(X_\nu\right)\right]$-representations, 
\begin{align}\label{eqn:det-rep2}
    \det(Y_S) \cong \bigotimes_{\nu \in S}\sgn_{S_{m_\nu}}\left[\mathbb{1}_{S_{|\nu|}}\right].
\end{align}

Again, it suffices to check on the generators. If $C_1, C_2, \ldots, C_r$ are the blocks of $Y_S,$ one ordered basis for $Y_S$ is $u_1, u_2, \ldots, u_r,$ where $u_i$ is the sum of the $v_j$ for which $B_j \subseteq C_i.$ The inner generators act trivially on this ordered basis and an outer generator swaps $u_i, u_j$ for two blocks $C_i, C_j$ of $Y_S$ indexed by the same partition, hence is orientation reversing. Thus, \cref{eqn:det-rep2} holds and the lemma then follows from \cref{lem:mult-wreath-prod-reps}.
\end{proof}

To compute the homology of intervals in the partition lattice, we will use the extremely useful tools Sundaram explains in \cite{SundaramJerusalem}. If $P$ and $Q$ are posets, let \emph{$P \times Q$} denote the \emph{cartesian product} of $P$ and $Q,$ which is the set $P\times Q$ under the relation $(p, q) \leq (p', q')$ if and only if $p \leq p'$ in $P$ and $q \leq q'$ in $Q.$ We write \emph{$P^n$} to denote the $n$-fold Cartesian product $\underbrace{P \times P \times \ldots \times P}_{n \text{ times}}.$

\begin{prop}\label{prop:poset-product}(Sundaram, \cite{SundaramJerusalem}).
Let $P$ and $Q$ be posets, each with a minimal element $\hat{0}$ and a maximal element $\hat{1}.$ Let $G_P$, $G_Q$ be groups acting by poset automorphisms on $P$ and $Q$ respectively. Assume that for $i \neq r$, the homology groups $\Tilde{H}_i(P) = 0$ and that for $j \neq s,$ the homology groups $\Tilde{H}_j(Q) = 0.$ 

\begin{enumerate}
    \item[(i)][Proposition 2.1 of \cite{SundaramJerusalem}] $\Tilde{H}_i(P \times Q) = 0$ unless $i = r + s + 2$, and as $G_P \times G_Q$ representations, \[\Tilde{H}_{r + s + 2}(P \times Q) \cong \Tilde{H}_{r}(P) \otimes \Tilde{H}_{s}(Q).\]
    \item[(ii)][{Proposition 2.3 of \cite{SundaramJerusalem}}\footnote{There is a very minor typo in Proposition 2.3 of \cite{SundaramJerusalem} in what dimension the homology should be concentrated in.}] $\Tilde{H}_i(P^n) = 0$ unless $i = n(r + 2) - 2$ and as $S_n[G_P]$-representations, 
    \begin{align*}
        \Tilde{H}_{n(r + 2) - 2}(P^n) \cong \begin{cases}
        \mathbb{1}_{S_n}\left[\Tilde{H}_r(P)\right] \text{ if $r$ is even},\\
        \mathbb{\sgn}_{S_n}\left[\Tilde{H}_r(P)\right] \text{ if $r$ is odd}.
    \end{cases}
    \end{align*}
\end{enumerate}
\end{prop} 

We note that the finite set partition lattices satisfy the hypotheses of \cref{prop:poset-product} since they are a special type of lattice called a \emph{geometric lattice.} In \cite{Folkman}, Folkman proved that geometric lattices of rank $r$ have nonzero homology only in dimension $r - 2.$ Since $\Pi_n$ is a geometric lattice of rank $n - 1,$ it has nonzero homology only in dimension $n - 3.$

We are now ready to compute the (twisted) homology representations of intervals in the partition lattice in terms of the twisted homology of upper intervals.

\begin{lemma}\label{prop:homology-of-intervals.}
  Fix a multiset partition $S = \{\nu^{m_\nu}: \nu \in S\} \in \lambda(\mu)$ and let $r = \ell(\mu) - \ell(\lambda).$ As representations of $\prod_{\nu \in S}S_{m_\nu}[\stab_{S_{|\nu|}}(X_\nu)],$ 
    \begin{align*}
        \Tilde{H}_{r - 2}\left(X_S, Y_S\right) \otimes \det(X_S) \otimes \det(Y_S) \cong \bigotimes_{\nu \in S} \mathbb{1}_{S_{m_\nu}} \left[ \Tilde{H}_{\ell(\nu) - 3} \left( X_\nu, \hat{1} \right) \otimes \det\left(X_\nu\right) \right].
    \end{align*}
\end{lemma}
\begin{proof}
    The interval $[Y_S, X_S]$ in $\Pi_n$ is (equivariantly) isomorphic to the following product of partition lattices: \[\prod_{\nu \in S}\left( [X_\nu, \hat{1}]\right)^{m_\nu},\]
with each $S_{m_\nu}[\stab_{S_{|\nu|}}(X_\nu)]$ acting on its corresponding $\left( [X_\nu, \hat{1}]\right)^{m_\nu}.$ Recall the definition of $\rho(\nu)$ from \cref{lem:det-rep}. By \cref{prop:poset-product}, as $S_{m_\nu}\left[\stab_{S_{|\nu|}}\left(X_\nu\right)\right]$ representations, 
\begin{align*}
    \Tilde{H}_{r - 2} \left(X_S, Y_S\right) \cong \bigotimes_{\nu \in S} \rho(\nu) \left[\Tilde{H}_{\ell(\nu) - 3}\left(X_\nu, \hat{1}\right)\right].
\end{align*}

Now, by applying \cref{lem:det-rep} and \cref{lem:mult-wreath-prod-reps}, we have that as $\prod_{\nu \in S}S_{m_\nu}[\stab_{S_{|\nu|}}(X_\nu)]$ representations,
\begin{align*}
\Tilde{H}_{r - 2}\left(X_S, Y_S\right) \otimes \det(X_S) \otimes \det(Y_S) \cong \bigotimes_{\nu \in S}\mathbb{1}_{S_{m_\nu}} \left[ \Tilde{H}_{\ell(\nu) - 3} \left( X_\nu, \hat{1} \right) \otimes \det\left(X_\nu\right) \right].
\end{align*}
\end{proof}
\subsection{Final steps}\label{subsec:final-step}
In this section, we complete step 4 of the proof outline in \cref{sec:outline-proof}, proving our main theorem. To do so, we must first state a lemma and proposition. The following lemma appears as Exercise 4.1.3 in \cite{Grinberg--Reiner}.

\begin{lemma} \label{lem:induction}
Let $G_1$, $G_2$ be finite groups, with subgroups $H_1$, $H_2$, respectively. Let $V_1$, $V_2$ be representations of $H_1$, $H_2$, respectively. Then, as $G_1 \times G_2$ representations,
\begin{align*}
    \left( V_1 \otimes V_2\right) \Big \uparrow_{H_1 \times H_2}^{G_1 \times G_2} \cong V_1 \Big \uparrow_{H_1}^{G_1} \otimes V_2 \Big \uparrow_{H_2}^{G_2}.
\end{align*}
\end{lemma}

\begin{prop}\label{prop:reduction-to-lambda=n}
Let $\lambda$ and $\mu$ be any partitions of $n.$ Write \emph{$m_\nu(S)$} to mean the multiplicity of a partition $\nu$ in a multiset partition $S.$
\begin{align*}
    \ch(E_\lambda \C\mathcal{F}_n E_\mu) = \sum_{S \in \lambda(\mu)}\prod_{\nu \in S}h_{m_\nu(S)}\left[\sum_{\substack{\text{Lyndon }\\w}}\sum_{\substack{m \geq 1:\\w^m \sim \nu}}L_{m}[h_{w}]\right].
\end{align*}
\end{prop}
\begin{proof}
If $\mu$ does not refine $\lambda,$ then  by \cref{prop:double-idemp-sapces}, $E_\lambda \C \mathcal{F}_n E_\mu = 0.$ On the right side, we also have $0,$ since $\lambda(\mu)$ is empty by definition.

Now, assume  $\mu$ refines $\lambda$ with $\ell(\mu) - \ell(\lambda) = r.$
 By \cref{cor:induced-cohomology}, as $S_n$-representations,
\[E_\lambda \C \mathcal{F}_n E_\mu \cong \bigoplus_{\substack{S_n-\text{orbits }[X \leq Y]:\\ X \in \mu\\ Y \in \lambda}} \Tilde{H}_{r - 2}\left(X, Y\right) \otimes \det(X) \otimes \det(Y)\Bigg \uparrow_{\stab_{S_n}(X) \cap \stab_{S_n}(Y)}^{S_n}.\]

Using \cref{lem:bij-w-lambda(mu)} and \cref{lem:normalizer-intersections},
    \begin{align*}
    E_\lambda \C \mathcal{F}_n E_\mu &\cong \bigoplus_{S \in \lambda(\mu)} \left(   \Tilde{H}_{r - 2}(X_S, Y_S) \otimes {\det(X_S) \otimes \det(Y_S)}\right)\Bigg\uparrow_{\prod_{\nu \in S}S_{m_\nu(S)}\left[\stab_{S_{|\nu|}}(X_\nu)\right]}^{S_n} ,
\end{align*}
 Using \cref{prop:homology-of-intervals.}, the above is isomorphic (as $S_n$-representations) to
\begin{align} \label{eqn:before-fixing-S}
\bigoplus_{S \in \lambda(\mu)}\left(\bigotimes_{\nu \in S} \mathbb{1}_{S_{m_\nu(S)}}\left[ \Tilde{H}_{\ell(\nu) - 3} \left( X_\nu, \hat{1}\right) \otimes \det \left({X_{\nu}}\right)\right]\right)\Bigg \uparrow_{\prod_{\nu}S_{m_{\nu}(S)}\left[\stab_{S_{|\nu|}}(X_\nu)\right]}^{S_n}.
  \end{align}

We fix a particular $S = \{\nu^{m_\nu}:  \text{partition }\nu\} \in \lambda(\mu)$, and manipulate the summand indexed by $S$ as follows.  By transitivty of induction  and \cref{lem:induction}, the summand indexed by $S$ above is isomorphic as an $S_n$-module to
\begin{align*}
\left(\bigotimes_{\nu \in S}\mathbb{1}_{S_{m_\nu}}\left[ \Tilde{H}_{\ell(\nu) - 3} \left(X_\nu, \hat{1}\right) \otimes \det\left({X_{\nu}}\right)\right]\Bigg\uparrow_{S_{m_\nu}\left[\stab_{S_{|\nu|}}(X_\nu)\right]}^{S_{m_\nu\cdot|\nu|}}\right)\Bigg \uparrow_{\Pi_{\nu}S_{m_\nu \cdot |\nu|}}^{S_n} \cong \\
\left(\bigotimes_{\nu \in S}\mathbb{1}_{S_{m_\nu}}\left[ \Tilde{H}_{\ell(\nu) - 3} \left(X_\nu, \hat{1}\right) \otimes \det \left({X_{\nu}}\right)\right]\Bigg \uparrow_{S_{m_\nu}\left[\stab_{S_{|\nu|}}(X_\nu)\right]}^{S_{m_\nu}\left[S_{|\nu|}\right]} \Bigg \uparrow_{S_{m_\nu}\left[S_{|\nu|}\right]}^{S_{m_\nu \cdot |\nu|}}\right)\Bigg\uparrow_{\Pi_{\nu}S_{m_\nu \cdot|\nu|}}^{S_n}.\\
\end{align*}
By \cref{conj:induction-wreath-prods}, this is equivalent to 
\begin{align*}
 \bigotimes_{\nu \in S}\left(\mathbb{1}_{S_{m_\nu}}\left[  \Tilde{H}_{\mathrm{\ell(\nu) - 3}} \left( X_\nu, \hat{1} \right) \otimes \det \left({X_{\nu}}\right)\Big \uparrow_{\stab_{S_{|\nu|}}(X_\nu)}^{S_{|\nu|}}\right] \Bigg \uparrow_{S_{m_\nu}\left[S_{|\nu|}\right]}^{S_{m_\nu \cdot |\nu|}}\right)\Bigg\uparrow_{\Pi_{\nu}S_{m_\nu \cdot |\nu|}}^{S_n}.
\end{align*}
However, by setting $\lambda = |\nu|, \mu=\nu$ in \cref{cor:induced-cohomology}, this is the same as
\begin{align*}
 \bigotimes_{\nu \in S}\left(\mathbb{1}_{S_{m_\nu}}\left[ E_{|\nu|}\C \mathcal{F}_{|\nu|} E_\nu\right] \Big \uparrow_{S_{m_\nu}\left[S_{|\nu|}\right]}^{S_{m_\nu \cdot |\nu|}} \right)\Bigg \uparrow_{\Pi_{\nu}S_{m_\nu \cdot |\nu|}}^{S_n},
\end{align*}
which has Frobenius characteristic $ \prod_{\nu\in S}h_{m_\nu}\left[ \ch( E_{|\nu|} \C \mathcal{F}_{|\nu|}E_\nu)\right].$ Returning to \cref{eqn:before-fixing-S} and summing over all $S \in \lambda(\mu)$ reveals that

\[\ch(E_\lambda \C \mathcal{F}_n E_\mu) = \sum_{S \in \lambda(\mu)} \prod_{\nu\in S}h_{m_\nu(S)}\left[ \ch( E_{|\nu|} \C \mathcal{F}_{|\nu|}E_\nu)\right].\]

The proposition then follows from \cref{prop:base-case}.
\end{proof}

This next lemma follows from repeatedly applying Lemma 1.5 of \cite{SundaramAdvances}.

\begin{lemma}\label{lem:plethysm-properties}
Let $f_1, f_2, \ldots, f_k$ be symmetric functions and $t$ an abstract variable. Then, \[\sum_{i \geq 0}t^i h_i[f_1 + f_2 + \ldots + f_k] = \prod_{j = 1}^k \sum_{\ell \geq 0}t^\ell h_\ell[f_j].\]
\end{lemma}

We are finally ready to prove our main theorem.

\begin{thm}\label{conj:typeA-comp-factors} There is an equality of generating functions
\begin{align}\label{eqn:ogf}
    \sum_{n \geq 0} \sum_{\substack{\lambda \vdash n\\ \mu \vdash n}}  \mathbf{y}_\lambda \mathbf{z}_\mu\cdot \ch(E_\lambda \C \mathcal{F}_n E_\mu) = \prod_{\substack{\text{Lyndon }\\ w}} \sum_{\substack{\text{partition}\\\rho}} \mathbf{y}_{\rho \cdot |w|}\mathbf{z}_{w^{|\rho|}}L_\rho[h_w].
\end{align}
\end{thm}

\begin{proof}
By \cref{prop:reduction-to-lambda=n}, it suffices to prove that 
\begin{align}\label{eqn-to-prove}
   \sum_{n \geq 0}\sum_{\substack{\lambda \vdash n\\ \mu \vdash n}} \mathbf{y}_\lambda \mathbf{z}_\mu \cdot  \sum_{S \in \lambda(\mu)}\prod_{\nu \in S}h_{m_\nu(S)}\left[\sum_{\substack{\text{Lyndon }\\w}}\sum_{\substack{m \geq 1:\\w^m \sim \nu}}L_{m}[h_{w}]\right] = \prod_{\substack{\text{Lyndon}\\w}} \sum_{\substack{\text{partition}\\\rho}} \mathbf{y}_{\rho \cdot |w|}\mathbf{z}_{w^{|\rho|}} L_\rho[h_w].
    \end{align}

We start by rewriting the right side of (\ref{eqn-to-prove}). Forming a partition $\rho$ is the same as choosing a multiplicity $i \geq 0$ for each part size $m > 0$. Since $L_{1^{i_1}2^{i_2} \ldots} = \prod_{m \geq 1} h_{i_m}[L_m]$ (\cref{eqn:def-higher-lie}),

\begin{align*}
\prod_{\substack{\text{Lyndon}\\w}} \sum_{\substack{\text{partition}\\ \rho}} \mathbf{y}_{\rho \cdot |w|}\mathbf{z}_{w^{|\rho|}} L_\rho[h_w] 
&= \prod_{\substack{\text{Lyndon}\\w}} \prod_{m \geq 1} \sum_{i \geq 0}\mathbf{y}_{(m \cdot |w|)^i}\mathbf{z}_{\left(w^{im}\right)} h_i[L_m][h_w].
\end{align*}

Using associativity of plethysm (\cref{prop:plethysm-properties}(1)) to rewrite 
$h_i \left[L_m\right] [h_w]=h_i\left[L_m [h_w]\right]$ and grouping together pairs $(m, w)$ for which $w^m$ rearranges to the same partition $\nu$ (written $w^m \sim \nu$), we can rewrite the above as 
\begin{align*}
\prod_{\substack{\text{partition}\\ \nu \neq \emptyset}} \prod_{\substack{(w, m):\\ w \text{ Lyndon},\\ m \geq 1,\\ w^m \sim \nu}} \sum_{i \geq 0} (\mathbf{y}_{|\nu|}\mathbf{z}_{\nu})^i h_i[L_{m}[h_w]].
\end{align*}

Setting $t = \mathbf{y}_{|\nu|}\mathbf{z}_\nu$ and using the $L_m[h_w]$ for $m \geq 1$ with $w^m \sim \nu$ as the $f_i$'s in \cref{lem:plethysm-properties}, we can rewrite the result as 

\[\prod_{\substack{\text{partition}\\ \nu \neq \emptyset}} \sum_{i \geq 0}(\mathbf{y}_{|\nu|}\mathbf{z}_\nu)^i h_i \left[ \sum_{\substack{\text{Lyndon}\\w}}\sum_{\substack{m \geq 1:\\ w^{m} \sim \nu}}L_{m}[h_w]\right].\]

Choosing a multiplicity $i \geq 0$ for each nonempty partition $\nu$ is the same as building a multiset partition $S$ (with block $\nu$ appearing $i = m_\nu(S)$ times).
{Thus, we can rewrite the above as} 
\[\sum_S \prod_{\nu \in S} (\mathbf{y}_{|\nu|}\mathbf{z}_\nu)^{m_\nu(S)} h_{m_\nu(S)} \left[ \sum_{\substack{\text{Lyndon}\\w}}\sum_{\substack{m \geq 1:\\ w^{m} \sim \nu}}L_{m}[h_w]\right].\]

A multiset partition $S$ lies in $\lambda(\mu)$ if and only if $\prod_{\nu \in S}\left(\mathbf{y}_{|\nu|}\mathbf{z}_\nu\right)^{m_\nu(S)} = \mathbf{y}_\lambda \mathbf{z}_\mu.$ Thus, we can rewrite the above to the desired generating function:
\[\sum_{n \geq 0} \sum_{\substack{\lambda \vdash n\\ \mu \vdash n}} \mathbf{y}_\lambda \mathbf{z}_\mu \cdot \sum_{S \in \lambda(\mu)}  \prod_{\nu \in S}  h_{m_\nu(S)} \left[ \sum_{\substack{\text{Lyndon}\\w}}\sum_{\substack{m \geq 1:\\ w^{m} \sim \nu}}L_{m}[h_w]\right].\qedhere\]
\end{proof}

\section{Generalizations} \label{sec:generalizations}
\subsection{Other Coxeter types}
Some, but not all, of this work holds in general Coxeter types. In this section, we (briefly) explain partial generalizations and point out obstacles to full generalizations. For brevity, we assume the reader is familiar with reflection group theory and the face algebra of a reflection arrangement. See \cite{kane} for a reference on reflection groups and \cite{Aguiar-Mahajan, saliolaquiverdescalgebra, saliolafacealgebra} for nice descriptions of the general face and descent algebras.

Let \emph{$W$} be a finite real reflection group with a set $S$ of simple reflections. Write \emph{$\mathcal{F}(W)$} for the face monoid of the associated reflection arrangement and $\C \mathcal{F}(W)$ for the face monoid algebra. Let \emph{$\Sigma(W)$} be $W$'s associated descent algebra over the complex numbers. Bidigare proved his theorem \cite[Theorem 3.8.1]{Bidigare} at the level of general reflection groups, so there is an antiisomorphism \emph{$\Phi: \left(\C \mathcal{F}(W)\right)^W \to \Sigma(W).$} 

 Write \emph{$\mathcal{L}(W)$} for the lattice of intersections of the reflection arrangement of $W$ ordered by \textit{reverse} inclusion. The reflection group $W$ acts on $\mathcal{L}(W)$ by poset automorphisms. There is again a support map \emph{$\sigma: \mathcal{F}(W) \to \mathcal{L}(W)$} which has the same nice properties as in type $A$ (it commutes with the action of $W,$ $\sigma(fg)  = \sigma(f) \wedge \sigma(g)$, etc). The \textit{pointwise stabilizer subgroup}  of an intersection $X \in \mathcal{L}(W)$ is $W_X:= \{w \in W: w(x) = x \text{ for all }x \in X\}$ and turns out to always be conjugate to a parabolic subgroup. (Recall that a parabolic subgroup is some subgroup $W_J$ of $W$ generated by some subset $J \subseteq S.$)
 
 The simple modules for Solomon's descent algebra are indexed by $W$-orbits of intersections $[X]$ for $X \in \mathcal{L}(W).$  We write $X \in [X']$ if $X$ is in the $W$-orbit of $X'$ and write $\mathcal{L}(W) / W$ for the set of $W$-orbits of intersections. Similarly, we write $\mathcal{F}(W) / W$ for the set of $W-$orbits $[f]$ of faces $f \in \mathcal{F}(W).$ Let \emph{$\{E_{[X]}: [X] \in \mathcal{L}(W) / W\}$} be a cfpoi for $\left(\C\mathcal{F}(W)\right)^W$ and let \emph{$\{M_{[X]}: [X] \in \mathcal{L}(W)  W\}$} denote the corresponding $\Sigma(W)$-simples.

The analogue of \cref{prop:decomp-isos} in general Coxeter type is the following proposition. 

 \begin{prop}\label{prop:analogue-W}
Let $\chi$ be an irreducible character of $W.$ The $\chi$-isotypic subspace of the $\C \mathcal{F}(W)$ decomposes into $\Sigma(W)$-modules as
 \begin{align*}
     \left(\C \mathcal{F}(W)\right)^\chi = \bigoplus_{\substack{[X] \in \mathcal{L}(W) / W}} \left(\C \mathcal{F}(W) E_{[X]}\right)^\chi,
 \end{align*}
 The $\C$-dimension of the submodule $  \left(\C \mathcal{F}(W) E_{[X]}\right)^\chi$ is \[\dim \left(\chi\right) \cdot \# \{[f] \in \mathcal{F}(W) / W: \sigma(f) \in [X]\} \cdot \left \langle \chi, \C [W / W_X]\right \rangle.\]
 \end{prop}
 \begin{proof}[Proof outline] The main ideas are:
   \begin{itemize}
        \item Direct analogues of  \cref{cor:supports-of-kb-idemps}, \cref{prop:bE_b-form-basis}, \cref{lem:invariant-idems-are-sums-of-idems}, and \cref{lem:invariant-saliola-lemma} all hold because Aguiar--Mahajan \cite{Aguiar-Mahajan} and Saliola \cite{saliolaquiverdescalgebra} proved their results on cfpois for general finite reflection groups.
        \item The direct analogue to \cref{prop:decomposing-monoid-algebra} also holds, giving a $W$-isomorphism between $\C \mathcal{F}(W)E_{[X]}$ and the $\C$-span of faces $f \in \mathcal{F}(W)$ with $\sigma(f) \in [X].$ Since $\{w: w(f) = f\} = W_{\sigma(f)}$ and $W_X$ is conjugate to $W_{X'}$ for $X' \in [X],$ the latter is $W$-isomorphic to the sum of coset representations\[\bigoplus_{\substack{[f] \in \mathcal{F}(W):\\ \sigma(f) \in [X]}} \C[W / W_{\sigma(f)}] \cong \bigoplus_{\substack{[f] \in \mathcal{F}(W):\\ \sigma(f) \in [X]}} \C[W / W_{X}].\]
     \end{itemize}
 \end{proof}
It would be nice if we could understand the composition multiplicities of the $\Sigma(W)$-simples $M_{[Y]}$ in each submodule $\left(\C \mathcal{F}(W) E_{[X]}\right)^\chi.$ As in type A, we are able to reduce this to a question in (twisted) equivariant poset cohomology. There is a well-defined partial ordering on the $W$-orbits  $[X] \in \mathcal{L}(W) / W$ given by declaring \emph{$[X] \leq [Y]$} if there exists $X' \in [X]$ and $ Y' \in [Y]$ such that $X'\leq Y'$ in $\mathcal{L}(W).$ For an intersection $X,$ let $N_W(X)$ denote the \textit{setwise} $W$-stabilizer of $X,$ i.e. $N_W(X):= \{w \in W: w(X) = X\}.$ (This notation is often used because $N_W(X)$ turns out to be the \textit{normalizer} of $W_X$ in $W.$) An analogue of \cref{cor:induced-cohomology} is below. 

 \begin{prop}\label{gen-of-cor:induced-cohomology}
      If $[X]\nleq [Y],$ then the composition multiplicity $\left[\left(\C \mathcal{F}(W) E_{[X]}\right)^\chi:M_{[Y]}\right] = 0.$ Otherwise, if $k = \dim_{\mathbb{R}}(X) - \dim_{\mathbb{R}}(Y),$ then \[\left[\left(\C \mathcal{F}(W) E_{[X]}\right)^\chi:M_{[Y]}\right] = \dim \left(\chi \right)\cdot \left \langle \chi, E_{[Y]}\C \mathcal{F}(W) E_{[X]}\right\rangle.\] Further, as $W$-modules, $E_{[Y]}\C \mathcal{F}(W) E_{[X]}$ is isomorphic to
     \begin{align*}
        \bigoplus_{[X' \leq Y']}\left (\Tilde{{H}}_{k - 2} \left(X', Y'\right) \otimes \det(X') \otimes \det(Y') \right)\Bigg \uparrow_{N_W(X') \cap N_W(Y')}^W.
     \end{align*}
     The sum is over $W$-orbits of pairs $[X' \leq Y']$ in $\mathcal{L}(W)$ where $X' \in [X], Y' \in [Y].$ 
 \end{prop}
 \begin{proof}[Proof outline] As in type $A$ (\cref{prop:transform-to-group-rep-q}), finite dimensional algebra tools reveal the composition multiplicity \[\left[\left ( \C \mathcal{F}(W) E_{[X]}\right)^\chi: M_{[Y]}\right] = \dim \left(\chi \right)\cdot \left \langle \chi, E_{[Y]}\C \mathcal{F}(W) E_{[X]}\right\rangle.\] 
 Given that the analogues of \cref{cor:supports-of-kb-idemps} and \cref{lem:invariant-idems-are-sums-of-idems} hold, we have the analogue of \cref{prop:double-idemp-sapces}. Additionally, Saliola's equivariant surjection from the path algebra of the quiver to the face algebra described in \cref{sec:Saliola'smap} was a uniform result, which allows us to reframe $E_Y \C \mathcal{F}_n E_X$ in terms of (twisted) poset cohomology by the argument in \cref{sec:precise-conversion}. Finally, the representations of real reflection groups are all self-dual as they are defined over the real numbers (see  \cite{BENARD1976318}), so the same tricks hold for converting cohomology to homology in the general case.
 \end{proof}

 We are also able to understand the sign isotypic subspace for general type, gaining an analogue to \cref{prop:sign-iso}. To state it, we first explain a known generalization of ``cycle type'' of an element $w \in W.$ 
Let $V$ be the ambient space of the reflection arrangement of $W$ (i.e. the intersection at the bottom of $\mathcal{L}(W)$).  Given an element $w \in W$, the space $\mathrm{Fix}(w):= \{v \in V: w(v) = v\}$ turns out to be an intersection in $\mathcal{L}(W)$ (\cite[Theorem 6.27]{orlik2013arrangements}). Let $\mathrm{cyc}(w)$ be the $W$-orbit $[\mathrm{Fix}(w)].$ 

In \cite[Thm 1.1]{blesshohlshock}, Blessenohl, Hohlweg, and Schocker beautifully generalize Gessel and Reutenauer's result (\cite[Theorem 2.1]{gesselreutenauer}) using connections between descent algebras and characters of reflection groups. One special case of their theorem proves $\left\langle \C W \Phi^{-1}\left(E_{[Y]}\right), \chi_{\sgn} \right \rangle$ is the number of elements $w \in W$ with cycle type $[Y]$ and descent set $S.$ 
 
 \begin{prop}\label{prop:sign-iso-general}
     As $\Sigma(W)$-modules, the sign isotypic subspace of the face algebra $\C \mathcal{F}(W)$ is isomorphic to the $\Sigma(W)$-simple $M_{\mathrm{cyc}(w_0)},$ where $w_0$ is the longest element of $W.$
 \end{prop}
 \begin{proof}[Proof idea]
     It can be shown from \cref{prop:analogue-W} that $\left(\C \mathcal{F}(W)\right)^{\sgn}$ is one dimensional and equal to $\left(\C \mathcal{F}(W) E_{[V]}\right)^{\sgn}$. Hence it is isomorphic to the simple $M_{[Y]}$ for the unique $[Y]$ for which  \[0 \neq \dim \left(E_{[Y]}\C\mathcal{F}(W)E_{[V]}\right)^{\sgn}.\]  Using the same argument as in \cref{sec:Uyemura--Reyes}, one can show $E_{[Y]}\C\mathcal{F}(W)E_{[V]}$ is isomorphic as a $W$-representation to $\C W \Phi^{-1}\left(E_{[Y]}\right).$ Hence,  \[\dim \left(E_{[Y]}\C\mathcal{F}(W)E_{[V]}\right)^{\sgn} = \dim \left(\C W\Phi^{-1}\left(E_{[Y]}\right)\right)^{\sgn}= \left \langle \C W \Phi^{-1} \left(E_{[Y]}\right), \chi_{\sgn}\right \rangle.\]
     The proposition then follows from Blessenohl--Hohlweg--Schocker's result.
 \end{proof}

 Unfortunately, \cref{prop:analogue-W} and \cref{prop:sign-iso-general} are as far as we are able to generalize our work \textit{uniformly} to general reflection groups. Indeed, the proof of our main theorem in type $A$ (\cref{conj:typeA-comp-factors}) relied heavily on symmetric function theory and the {equivariant} description of the homology of the lattice of intersections. As far as we are aware, there are no uniform generalizations of either of these notions. However, from at least a naive standpoint, there appears to be more hope in generalizing this story specifically to type $B.$ 

Since the hyperoctahedral groups are wreath products of symmetric groups, there is a theory of type $B$ symmetric functions: see \cite[\S 5 of Appendix B]{macdonald} and also \cite[\S 2]{Poirier}. 
The representation structure of the hyperoctahedral groups acting on the top (co)homology of the type $B$ intersection lattice has also been studied: see \cite{bergeronfreeliealgebra, GottliebWachs, HanlonTypeB}. Perhaps the most promising sign is that the type $B$ descent algebra Cartan invariants are known (unlike in general Coxeter type). They were computed by N. Bergeron in \cite{Bergerons-hyper2}, and are very closely connected to Garsia--Reutenauer's result in type $A$ (\cref{thm:garsia-reutenauer}). Hence, we pose the following question.



\begin{question}\label{quest:type-B}
    What is the type $B$ analogue to \cref{conj:typeA-comp-factors}?
\end{question}

Some useful first steps towards answering \cref{quest:type-B} would be to understand:
\begin{itemize}
    \item the equivariant structure of upper intervals in the type $B$ lattice of intersections (as Sundaram did for type $A$ in \cref{lem:upper-interval-stabilizer-rep}),
    \item type $B$ analogues of \cref{lem:upper-intervals-plethysm} and \cref{prop:base-case} (this may involve developing a hyperoctahedral analogues of Wachs's lemma (\cref{lem:Wachs-plethysm-gen-fn}) and applying Poirier's work in \cite[Theorem 16]{Poirier}, which interprets a wreath product analog of $L_\lambda$ as a necklace generating function), and
    \item the intersections of the setwise stabilizers of intersections, $N_W(X) \cap N_W(Y)$, and the $W$-orbits $[X \leq Y]$ (i.e. type $B$ analogues of \cref{lem:bij-w-lambda(mu)}, \cref{lem:normalizer-intersections}).\end{itemize}

\subsection{Left regular bands} \label{sec:LRBs}A {left regular band (or LRB)} is a finite semigroup $B$ for which $x^2 = x$ and $xyx = xy$ for all $x, y \in B.$ As discussed in the background (\cref{eqn:lrb-property}), the face monoid of the braid arrangement-- and in fact the face semigroups for all hyperplane arrangements--are examples of LRBs. LRBs  have been studied extensively in \cite{aguiar2006coxeter, BrownonLRBs, margolis2015combinatorial, MSS, saliolaquiverlrb, SteinbergRepMonoidsBook} and have fascinating connections with other mathematical areas such as Markov chains and poset topology. Many of the LRBs occurring in the literature come equipped with a natural group action and it is natural to consider the invariant theory of their semigroup rings. For instance, see \cite{BraunerComminsReiner} for a study of the invariant theory of the free LRB and its $q$-analog, which involved Stirling numbers, random-to-top shuffling, and the derangement symmetric function. Generalizations of the stories in \cite{BraunerComminsReiner} and this paper to general LRBs is ongoing thesis work of the author.

\bibliographystyle{abbrv}
\bibliography{bibliography}

\end{document}